\documentclass[11.5pt]{article}
\usepackage{}
\usepackage{bbm}
\usepackage[leqno]{amsmath}
\usepackage{amsfonts}
\usepackage{graphicx}

\usepackage{amsmath}
\usepackage{amssymb}
\usepackage{latexsym}
\usepackage{amsmath, amsfonts,amssymb, amsthm, euscript,makeidx,color,mathrsfs}

\usepackage{enumerate} 

\usepackage{enumitem}
\usepackage{indentfirst}
\usepackage{indentfirst}

\def\sqr#1#2{{\vcenter{\vbox{\hrule height.#2pt
              \hbox{\vrule width.#2pt height#1pt \kern#1pt \vrule width.#2pt}
              \hrule height.#2pt}}}}
\def\signed #1{{\unskip\nobreak\hfil\penalty50
              \hskip2em\hbox{}\nobreak\hfil#1
              \parfillskip=0pt \finalhyphendemerits=0 \par}}
\def\endpf{\signed {$\sqr69$}}

\def\3n{\negthinspace \negthinspace \negthinspace }
\def\2n{\negthinspace \negthinspace }
\def\1n{\negthinspace }

\def\dbE{\mathbb{E}}
\def\dbF{\mathbb{F}}

\def\dbH{\mathbb{H}}

\def\dbP{\mathbb{P}}

\def\dbR{\mathbb{R}}
\def\dbS{\mathbb{S}}

\def\sK{\mathscr{K}}

\def\sS{\mathscr{S}}

\def\sU{\mathscr{U}}

\def\sY{\mathscr{Y}}

\def\={\buildrel \triangle \over =}

\def\ds{\displaystyle}

\def\ns{\noalign{\ss}}
\def\a{\alpha}
\def\b{\beta}
\def\g{\gamma}
\def\d{\delta}
\def\e{\varepsilon}
\def\z{\zeta}
\def\k{\kappa}
\def\l{\lambda}
\def\m{\mu}
\def\n{\nu}
\def\si{\sigma}
\def\t{\tau}
\def\f{\varphi}
\def\th{\theta}
\def\o{\omega}

\def\i{\infty}
\def\D{\Delta}
\def\Th{\Theta}
\def\L{\Lambda}

\def\cB{{\cal B}}
\def\cC{{\cal C}}
\def\cD{{\cal D}}
\def\cE{{\cal E}}
\def\cF{{\cal F}}

\def\cH{{\cal H}}

\def\cK{{\cal K}}
\def\cL{{\cal L}}

\def\cP{{\cal P}}
\def\cQ{{\cal Q}}

\def\cS{{\cal S}}

\def\cU{{\cal U}}

\def\cY{{\cal Y}}
\def\cZ{{\cal Z}}

\def\no{\noindent}

\def\ss{\smallskip}
\def\ms{\medskip}

\def\q{\quad}
\def\qq{\qquad}

\def\liminf{\mathop{\underline{\rm lim}}}

\def\lan{\mathop{\langle}}
\def\ran{\mathop{\rangle}}

\def\esssup{\mathop{\rm esssup}}

\def\argmax{\mathop{\rm argmax}}

\def\h{\widehat}

\def\ti{\tilde}
\def\cd{\cdot}

\def\diam{\hbox{\rm diam$\,$}}

\def\tr{\hbox{\rm tr$\,$}}

\def\les{\leqslant}
\def\ges{\geqslant}

\def\({\Big (}
\def\){\Big )}
\def\[{\Big[}
\def\]{\Big]}
\def\bde{\begin{definition}}
\def\ede{\end{definition}}
\def\be{\begin{equation}}
\def\bel{\begin{equation}\label}
\def\ee{\end{equation}}
\def\bt{\begin{theorem}}
\def\et{\end{theorem}}
\def\bc{\begin{corollary}}
\def\ec{\end{corollary}}
\def\bl{\begin{lemma}}
\def\el{\end{lemma}}
\def\bp{\begin{proposition}}
\def\ep{\end{proposition}}
\def\bas{\begin{assumption}}
\def\eas{\end{assumption}}
\def\br{\begin{remark}}
\def\er{\end{remark}}
\def\ba{\begin{array}}
\def\ea{\end{array}}
\def\ed{\end{document}}

\def\square#1{\vbox{\hrule\hbox{\vrule height#1%
     \kern#1\vrule}\hrule}}
\def\rectangle#1#2{\vbox{\hrule\hbox{\vrule height#1%
     \kern#2\vrule}\hrule}}

\font\tenbb=msbm10 \font\sevenbb=msbm7 \font\fivebb=msbm5

\newfam\bbfam
\scriptscriptfont\bbfam=\fivebb \textfont\bbfam=\tenbb
\scriptfont\bbfam=\sevenbb

\newtheorem{lemma}{Lemma}[section]
\newtheorem{remark}{Remark}[section]

\newtheorem{theorem}{Theorem}[section]
\newtheorem{corollary}{Corollary}[section]

\newtheorem{definition}{Definition}[section]
\newtheorem{proposition}{Proposition}[section]
\newtheorem{assumption}{Assumption}[section]

\makeatletter
   
   \@addtoreset{equation}{section}
\makeatother

\begin{document}

\title{Infinite time horizon stochastic recursive   control problems with jumps: dynamic programming   and stochastic verification theorems
}
 \author{
Sheng Luo \\
{\small Department of Applied Mathematics, The Hong Kong Polytechnic University,  Kowloon,
Hong Kong, P. R. China.}\\
{\small \textit{E-mail: 22038664r@connect.polyu.hk}}\\
Xun Li\thanks{%
 X. Li acknowledges the financial support from the Research Grants Council of Hong Kong under grants (Nos.~15216720, 15221621, 15226922), PolyU 1-ZVXA, and 4-ZZLT.
 }\\
{\small Department of Applied Mathematics, The Hong Kong Polytechnic University,  Kowloon,
Hong Kong, P. R. China.}\\
{\small \textit{E-mail: li.xun@polyu.edu.hk  }}\\
 Qingmeng Wei{\thanks{%
Corresponding author. Q. Wei thanks  the financial support from NSF of Jilin Province for Outstanding Young Talents (No. 20230101365JC),  the National Key R\&D Program of China (No. 2023YFA1009002), the NSF of
P. R. China (Nos.  12371443, 11971099) and  the Fundamental Research Funds for the Central Universities (No. 2412023YQ003)}}\\
{\small School of Mathematics and Statistics, Northeast Normal University,
Changchun 130024, P. R. China.}\\
 {\small \textit{E-mail: weiqm100@nenu.edu.cn}}}
\date{\today}
\maketitle

\begin{abstract}
This paper is devoted to studying an infinite time horizon  stochastic recursive control problem  with jumps, where  infinite time  horizon stochastic differential equation   and backward stochastic differential equation   with jumps   describe  the state process and cost functional, respectively. For this, the first is to explore the  wellposedness and regularity   of  these two equations in $L^p$-sense ($p\ges2$). By  establishing the dynamic programming principle, we   relate the value function of the control problem with   integral-partial differential equation of HJB type in the sense of viscosity solutions.
On the other hand,   stochastic verification theorems are also  studied to provide   sufficient conditions  to verify  the   optimality of  the given admissible controls.
Such a study is carried out in the  framework   of  classical solutions but also in that of  viscosity solutions.
Our work emphasizes important differences from the approach for finite time horizon problems. In particular, we have to work in an $L^p$-setting for $p>4$ in order to study the verification theorem in viscosity sense. 

%

\end{abstract}


\section{Introduction}
With the development of nonlinear backward stochastic differential equations (BSDEs, for short)   initially introduced by Pardoux and   Peng \cite{PP-1990}, a kind of BSDEs  driven   by both a Brownian motion and  an independent  Poisson  random measure  was proposed by  Tang and   Li \cite{Tang-Li-1994}.  Such equations are customarily called  as BSDEs with jumps.
  Numerous studies  about  BSDEs with jumps  have   been published for different applications,  such as  in stochastic control, for partial differential equations (PDEs, for short) and  in  mathematical finance. Let us in particular emphasize the growing number on works on stochastic control problems involving BSDEs with jumps. These studies concern the dynamic programming principle (see, for instance 
\cite{BBP-1997, BHL-2011,  LLW-2021,  LP-2009,  LW-AMO}), the stochastic  maximum principle (the reader is, e.g., referred to \cite{Shen-Siu-2013, STW-2020}) as well as linear quadratic stochastic control problems (\cite{Wei-Xu-Yu-2023,  Yu-2017}).  
However, the papers cited above use finite time horizon BSDEs with jumps. The studies for the corresponding  cases involving infinite time horizon BSDEs are very few, and this namely concerning   the
  dynamic programming.

 Based on the above review of literatures, we shall focus on  a kind   of infinite horizon stochastic recursive control problems with jumps and carry out the study by the approach of dynamic programming.  To be precise, the controlled dynamical   system evolves
 following  an infinite horizon forward-backward stochastic differential equations (FBSDEs, for short),
  which   consists of a forward equation  describing the state process,   and a backward   equation  defining the cost functional.
 It will be seen later that this is a meaningful  problem, which can not be solved directly by  combing or extending the existing results.  As the main features of this work, infinite horizon and Poisson random measure indeed bring  us  some unexpected difficulties. Meanwhile,  the  study of the $L^p$-theory ($p\ges2$)  of infinite horizon FBSDEs with jumps and the associated control problems  help us to reveal  several  essential differences from the finite horizon problems. Compared with Buckdahn et al. \cite{BHL-2012},   the semi-convexity of our value function holds only when $p>4$.  For the stochastic verification theorem in the viscosity solution framework, $p>4$ is also unavoidable,  which never occur in the finite horizon  case, referring to \cite{CL-arxiv, GSZ1, GSZ2, LHW-2023, ZYL}. 
 In the following, we elaborate the  details and  also present  the main contributions   of this work.

%

 First of all,   compared with finite horizon problems,  the study on the   global integrality of our  controlled dynamical system is an unavoidable part.   We try to use  the dissipative condition   in dynamical
systems here, but need to  adjust it to adapt to our FBSDEs with jumps.
 However,   the usual  wellposedness, regularity and  stability in   $L^2$ sense are not enough for  our study on the control problem.  The  global integrality in $L^p$ (with $p>2$) framework turns out to be the first really difficulty. Meanwhile, the Poisson random measure makes the $L^p$-theory much harder to reach.
We make the great effort to reconcile these two aspects. A  new dissipative condition is concluded by the subtle procedures and  technical derivations.
 The new  condition  not only  ensures the  global integrality  in $L^p$  (with $p>2$) sense, but also is applicable  to   another  jump-diffusion  systems, referring to Propositions \ref{LemmaA-1} and   \ref{LemmaA-1.2}.
Moreover, it  can
degenerate   into the classical ones by setting $p=2$  or further omitting the jump terms.

Based on the well-prepared research on  infinite horizon FBSDEs with jumps,
  an infinite horizon optimal recursive  control problem (Problem (OC)$_p$) is  formulated   for any given  $p\ges 2$,  denoted by Problem (OC)$_p$ in Section 3.
   We make a  systematic   study  of  Problem (OC)$_p$, especially the properties of the value function $V_p(\cd)$ defined in \eqref{bar u}, such as the Lipschitz continuity, the linear growth and the semi-convexity.
The reason of studying Problem (OC)$_p$ with $p\geq 2$ rather than  the usual  $p=2$, is that the   value function $V_p(\cd)$ may have different properties under the different $p$, in particular,  the semi-convexity.
 In fact,  changing  $p $ effects  the conditions imposed on the coefficients  $b,\si,\g,  f $ of the controlled system  through  the assumption  {\bf(C1)$_p$} and the set $\cU_{0,\i}^p$ of admissible controls, and, by consequence, the behaviour   of the state processes $X^{x;u}$ and the cost functionals $Y^{x;u}$.
No matter which value   $p $ may take, once fixed, the Lipschitz continuity and  linear growth of $V_p(\cd)$ can be concluded under the corresponding assumptions. But the situation  is quite different for  the semi-convexity of $V_p(\cd)$, which turns out to hold only for $p>4$; the reader is referred to  the  proof of Proposition \ref{Le-Semi}. Unlike   finite time horizon stochastic control problems (such as in \cite{BHL-2012,  Jing-2013, LHW-2023, YZ}),
this is a completely novel situation we meet with our  infinite time horizon case.

Subsequently, we establish the dynamic programming principle (DPP, for short) of our Problem (OC)$_p$ (with $p\geq 2$) by introducing a measurable and measure-preserving shift. The value function $V_p(\cd)$ is not necessarily smooth   under our conditions, so that a kind of weak solution  is considered here, in order to characterize $V_p(\cd)$ as solution of an associated    PDE. That is the viscosity solution, which was introduced  first    by Crandall and Lions \cite{CL-1983}. We    also refer  to    Crandall et al. \cite{CLP-1992}.  More precisely, our value function is proved to solve   an HJB equation  with  integral-differential operator  in viscosity sense.
We point out  that     the study of this part does not need the semi-convexity of $V_p(\cd)$. Therefore, the above conclusion holds for any given $p\geq 2$.
In this manner, our stochastic control problem   provides  a  probabilistic interpretation for a kind of integral-PDEs of   HJB  type     through its value function.

Finally, we  put our attention to the verification of the optimality of  controls for Problem (OC)$_p$, i.e., we are interested in the study of an  adequate   stochastic verification theorem. It is a crucial issue for the application of the control theory.
This work is  carried out   in  two different frameworks,  that for a classical solution and that for a viscosity solution.
We give   sufficient conditions for the verification of the optimality of controls  for  our Problem (OC)$_p$, and this for our both cases that of a classical solution and that of a viscosity solution  of the HJB equation  with  integral-differential operator.
It should be pointed out that in the case of a viscosity solution,  the semi-convexity of the solution of HJB equation \eqref{HJB-W-2} is crucial.  For this reason we show  in Section 3  that the value function  $V_p(\cd)$ of  Problem (OC)$_p$ as a viscosity solution of \eqref{HJB-W-2}  really
 satisfies this property for $p>4$ under some additional conditions on the coefficients.  This indicates that  the stochastic verification theorem for Problem (OC)$_p$  in the viscosity solution framework  holds just  true for $p>4$, rather than for $p\geq2$. This    is a major difference  from the  finite time horizon stochastic control problems.

%
%
%

The structure  of this manuscript is as follows. As the research basis, we first study the well-posedness and regularity  of infinite horizon SDEs and BSDEs with jumps in Section 2. Then,    infinite horizon stochastic recursive control denoted by Problem (OC)$_p$   is formulated in Section 3. The properties of the value function are  also studied in this section.    In Section 4   we  associate  Problem (OC)$_p$   with an  HJB equation with integral-differential  operators by   establishing  the DPP.
 Finally, Section 5  studies the  stochastic verification theorem  for Problem (OC)$_p$  in both the classical and the viscosity solutions frameworks.

\section{Infinite horizon   forward-backward stochastic differential equations with jumps}

Let us begin with introducing   the underlying probability space $(\Omega,\mathcal{F},\dbP)$, which is given as the completed product space of the following both spaces:
 %
\begin{itemize}
  \item $(\Omega_1,{\mathcal{F}_1},\dbP_1)$  is a classical Wiener space. More precisely,  $\Omega_1=C_0(0,\i;\mathbb{R}^d)$ is the set of all continuous functions   $\o:[0,\i)\to\mathbb{R}^d$ with $\o(0)=0$; $\mathcal{F}_1$ is the completed Borel $\sigma$-field on $\Omega_1$; $\dbP_1$ is the Wiener measure making the canonical process $B_s(\omega)=\omega(s)$, $s\ges0$, $\omega\in\Omega_1$,    a $d$-dimensional Brownian motion.
Denote the filtration on $(\Omega_1,{\mathcal{F}_1},\dbP_1)$ by $\mathbb{F}^B=(\mathcal{F}_s^B)_{s\ges 0}$ with
 $\mathcal{F}_s^B:=\sigma\big\{B_r, r\in[0,s]\big\}\vee \mathcal{N}_{\dbP_1},$
where $\mathcal{N}_{\dbP_1}$ is the collection of all $\dbP_1$-null sets.

  \item  $(\Omega_2,{\mathcal{F}_2},\dbP_2)$ is  a Poisson space. That is, $\Omega_2$ is the set of all point functions ${p}: D_{{p}} \rightarrow E$ with $D_{{p}}\subseteq  [0,\i)$ being  countable   and $E=\mathbb{R}^{l}\backslash\{0\}$ being equipped with its Borel $\sigma$-field  $\mathcal{B}(E)$. Denote by $\mathcal{F}_2$    the smallest $\sigma$-field on $\Omega_2$ such that the coordinate mapping ${p}\rightarrow \mu\big({p},(s,t]\times \Delta\big),$  $ \Delta\in\mathcal{B}(E),$  $ t>s\ges 0$,  is measurable.
%
%
   $\dbP_2$ is    the probability on the measurable space $(\Omega_2,\mathcal {F}_2)$  ensuring   the coordinate measure $\mu({p},\mathrm dt\mathrm de)$ to be  a Poisson random measure with the compensator $\hat{\mu}(\mathrm dt,\mathrm de):=\mathrm dt\l(\mathrm{d}e)$ and the process $\big\{\tilde{\mu}\big((s,t]\times A\big)\big\}_{s\les t}:=\big\{(\mu-\hat{\mu})\big((s,t]\times A\big)\}_{s\les t}$  to be   a martingale, for any $A\in\mathcal {B}(E)$ satisfying $\l(A)<\infty$. Here $\l$ is assumed to be a $\sigma$-finite measure on $\big(E,\mathcal {B}(E)\big)$ with $\displaystyle \int_E\big(1\wedge|e|^2\big)\l(\mathrm{d}e)<\infty.$
The filtration generated by $\mu$ on  $(\Omega_2,{\mathcal{F}_2},\dbP_2)$ is denoted by $(\mathcal{F}_s^\mu)_{s\ges 0}$ with
   $\mathcal{F}_s^\mu:=\big(\bigcap\limits_{r>s}\dot{\mathcal{F}}_r^\mu\big)\vee\mathcal{N}_{\dbP_2}$, where
 $\dot{\mathcal{F}}_r^{\mu}:=\sigma\big\{\mu\big((t,\t]\times \Delta\big),\ 0\les t\les \t\les r, \ \Delta\in \mathcal{B}(E)\big\},\ r\ges 0.$

\end{itemize}
  Then, by setting  $\Omega:=\Omega_1\times \Omega_2$, $\mathcal {F}:=\mathcal {F}_1\otimes \mathcal {F}_2$, $\dbP:=\dbP_1\otimes \dbP_2$  with $\mathcal {F}$ being  completed with respect to $\dbP$, we accomplish  the introduction of  the underlying probability space $(\Omega,\mathcal{F},\dbP)$. Also, we denote by $\mathbb{F}=(\mathcal{F}_s)_{s\ges 0}$    the filtration on  $(\Omega,{\mathcal{F}},\dbP)$ with $\mathcal{F}_s:=   \mathcal{F}_s^B\otimes\mathcal{F}_s^\mu$  augmented by all $\dbP$-null sets.

 Let $\dbH$  be any  Euclidean space or any subspace of  an Euclidean space, whose  Euclidean  inner product and  norm are denoted by $\lan\cd,\cd\ran$ and $|\cd|$, respectively. %
 For $t\ges0$, $T>t$ and $ p\ges 1$, we introduce the following  spaces of   processes or variables:
\noindent \begin{itemize}[leftmargin=*]
\item $L_{\cF_t}^{p}(\Omega;\mathbb{H}):=   \Big\{\xi \!\mid \!\xi\! : \Omega   \rightarrow \mathbb{H}$ is  $\cF_t $-measurable with $\mathbb{E}  |\xi|^{p} < \infty \Big\};$

\item $L^p_{\l}\big(E;\mathbb{H}\big):=   \Big\{K \!\mid\! K\!: E \rightarrow \mathbb{H}$ is $\mathcal{B}(E)$-measurable with $|
K|_{\l,p}:= \displaystyle  \(\int_E|K(e)|^p\l(\mathrm{d}e)\)^\frac1p  <\i  \Big\};$


 \item $ \cS_\mathbb{F}^{p}(t,T;\mathbb{H}) :=   \Big\{\varphi \!\mid\!\varphi\! :\Omega\times[t,T]  \rightarrow \mathbb{H}$ is  $\dbF $-progressively measurable with  $ \dbE\[\esssup\limits_{ s \in  [t,T]} |\varphi _{s} |^p\]   < \infty \Big\};$

 \item $L_\mathbb{F}^{p}(t,\infty;\mathbb{H}):=   \Big\{\varphi \!\mid\!\varphi\! :\Omega\times[t,\infty)  \rightarrow \mathbb{H}$ is  $\dbF $-progressively measurable with $\mathbb{E} \displaystyle \[\int_{t}^{\infty}|\varphi _{s}|^{p}\mathrm ds\]  < \infty \Big\};$

\item $\cK _\mathbb{F}^{  p}( t,\infty; \mathbb{H}):=   \Big\{K\! \mid\! K\!:\Omega\times [t,\infty)\times E \rightarrow \mathbb{H}$ is $\mathcal {P}_{t,\infty}$\footnote{$\cP_{t,\i}$ denotes the $\sigma$-field of $\dbF$-predictable subsets of $\Omega\times [t,\i)$. }$\otimes\mathcal{B}(E)$-measurable

     \hskip 8cm with $  \ds\mathbb{E}  \[ \int_t^\infty \int_E\!|K_s(e)|^p\l(\mathrm{d}e)\mathrm ds \]<  \infty \Big\}; $

 \item $L_\mathbb{F}^{p}(\Omega;L^2(t,\infty;\mathbb{H})):=   \Big\{\varphi \!\mid\!\varphi \!:\Omega\times[t,\infty)  \rightarrow \mathbb{H}$ is  $\dbF $-progressively measurable

     \hskip 8cm  with  $\mathbb{E} \displaystyle \[\(\int_{t}^{\infty}|\varphi _{s}|^{2}\mathrm ds\)^\frac p2\] \!<\! \infty \Big\};$

%

\item $  \sK_\mathbb{F}^{p}( t,\infty; \mathbb{H}):=   \Big\{K \mid\! K \!:\Omega\times [t,\infty)\times E  \to  \mathbb{H}$ is $\mathcal {P}_{t,\infty}\otimes\mathcal{B}(E)$-measurable with

   \hskip 8cm with   $ \mathbb{E} \[\(  \displaystyle\int_t^\infty\! \! \int_E  |K_s(e)|^2\l(\mathrm{d}e)\mathrm ds   \)^\frac p2  \]
   <   \infty  \Big\};  $

%

\item $\ds\sS^p_\dbF(t,\i):=  L_\mathbb{F}^{p}(\Omega;L^2(t,\infty;\mathbb{R}^m)) \times L_\mathbb{F}^{p}(\Omega;L^2(t,\infty;\mathbb{R}^{m\times d}))\times \sK_\mathbb{F}^{p}( t,\infty; \mathbb{R}^m)$.

\end{itemize}
%


\subsection{Infinite horizon     stochastic differential equations with jumps}
 For $p\ges 2$,    $(t,\xi)\in [0,\i)\times L_{\cF_t}^p(\Omega;\dbR^n)$, we consider the following infinite horizon SDE with jumps,
 \bel{SDE}
 \left\{\ba{ll}
 \ds\!\!\!\! \mathrm  dX_s=b(s,X_s)\mathrm ds+\si(s,X_s) \mathrm dB_s+\int_E\g(s,e,X_{s-})\ti{\m}(\mathrm ds,\mathrm de),\ s\ges t,\\
  \ds\!\!\!\! X_t=\xi,
 \ea\right.\ee
where the coefficients $b:\Omega\times[0,\i)\times\dbR^n \to\dbR^n$,  $\si:\Omega\times[0,\i)\times\dbR^n\to\dbR^{n\times d}$, $\g: \Omega\times[0,\i)\times E \times\dbR^n\to\dbR^n$ are assumed to satisfy the following conditions.

 \ss

 {\bf (A1)}     There exist three  constants $L_b,  L_\si, L_1 \ges 0 $, and  a mapping $l_\g(\cd)\in L_{\l}^2(E; {[0,1]}) \cap L_{\l}^p(E; {[0,1]}) $   such that, for all $s\ges 0$, $x,x'\in\dbR^n$, $e\in E$,
  $$\ba{ll}\ds |b(s,x)-b(s,x')|\les L_b|x-x'|,\  |\si(s,x)-\si(s,x')|\les L_\si  |x-x'|,\\
  \ds  |\g(s,e,x)-\g(s,e,x') | \les L_1l_\g(e)  |x-x'|,\q \dbP\mbox{-a.s.}\ea$$

 {\bf (A2)}    $b(\cd,0)\in  L_\dbF^p(0,\i;\dbR^n)$, $\si(\cd,0)\in L_\dbF^p( 0,\i;\dbR^{n\times d} )$, $ \g(\cd,\cd,0)\in \cK _\mathbb{F}^{  p}( 0 ,\infty; \mathbb{R}^n)  $.

\ss

 {\bf (A3)} There exists a constant $\eta_b>0$ such that, for all $s\ges 0$, $x,x'\in\dbR^n$,
 $$\ds\lan b(s,x)-b(s,x'),x-x'\ran \les -\eta_b|x-x'|^2,\q \dbP\mbox{-a.s.}$$

 {\bf (A4)}     $b(\cd,0)\in  L_\dbF^p (\Omega;L^2(0,\i;\dbR^n))$, $\si(\cd,0)\in L_\dbF^p  (\Omega;L^2(0,\i;\dbR^{n\times d}))$, $ \g(\cd,\cd,0)\in \sK _\mathbb{F}^{  p}(0,\infty; \mathbb{R}^n) $.

\ms

 It is classical that,   under {\bf (A1)} and  {\bf (A2)} with $p\ges 2$,      for  $(t,\xi)\in [0,\i)\times L_{\cF_t}^p(\Omega;\dbR^n)$  and $T>t$, SDE \eqref{SDE} admits a unique solution $X(\cd)\in \cS_\dbF^p(t,T;\dbR^n)$.
In the following,   we first  present a property of the global solution $X(\cd)\in L_\dbF^p(t,\i;\dbR^n)$  of SDE \eqref{SDE}   under the hypotheses  {\bf (A1)},  {\bf (A2)} and {\bf (A4)}.

\bl\label{Xto0} \sl Let  {\bf (A1)}, {\bf (A2)} and {\bf (A4)}  hold true with some $p\ges 2$. If $X \in L_\dbF^p(t,\i;\dbR^n)$ is the  unique solution of SDE with jumps  \eqref{SDE}, then $\lim\limits_{s\to \i}\dbE^{\cF_t}\big[|X(s)|^p\big]=0$, $\dbP$-a.s.

\el
The  proof of Lemma \ref{Xto0} is  similar to  Proposition 2.1 in \cite{Yu-2017}, which concerns   the case $p=2$. For $p>2$, we only need to use the techniques in \cite{Yu-2017} and
     It\^o's formula for $ |   X_\cd  |^p$  ($p\ges 2$) as follows:
\bel{X-ito-p}\ba{ll}
 \ds     |X_s |^p -|\xi |^p  =\! \int_t^s   \[\frac p2|X_r |^{p-2} \big(2\lan X_r  ,  b(r,X_r )\ran +  |\si(r,X_r )  |^2\big)      +   p(\frac p2-1)|X_r |^{p-4} |\si(r,X_r )^\top X_r|^2 \]    \mathrm dr   \\
 \ns\ds\q  +\int_t^s\int_E\[|X_{r}  + \g(r,e,X_{r} ) |^p- |X_{r} |^p-p |X_{r}|^{p-2}    \lan X_{r} ,\g(r,e,X_{r} )\ran   \] \l(\mathrm de)\mathrm d r \\
 \ns\ds\q   +\int_t^sp|X_r |^{p-2}    \lan X_r , \si(r,X_r )    \mathrm dB_r\ran+\int_t^s\int_E\[|X_{r-}  + \g(r,e,X_{r-} ) |^p- |X_{r-} |^p\]\ti{\m}(  \mathrm dr,\mathrm d e) ,\q s\ges t.
 \ea\ee
  We notice  that, all the integrals on the right  side of \eqref{X-ito-p} make  sense under our conditions  {\bf (A1)}, {\bf (A2)}, {\bf (A4)}.   The readers may refer to Theorem 93 in \cite{Situ}.
\ss

Before the study, we recall a useful    inequality for  the Poisson stochastic  integral, see  \cite{FK-1989} or  Theorem 4.4.23  in  \cite{Applebaum-2009}.
For the readers' convenience, we present its proof.

 \bl\label{Lemma-2.2} \sl Let $T>t$ and $h(\cd,\cd)\in \cK_\dbF^p(t,T;\dbR^n)\cap \sK_\dbF^p(t,T;\dbR^n)$ with $p\ges 2$. Then,
 $$ \dbE^{\cF_t}\[\sup\limits_{s\in[t,T]}\Big|\int_t^s\!\int_Eh(r,e)\ti{\m}( \mathrm dr, \mathrm de)\Big|^p\]\les  C_p\dbE^{\cF_t}\[ \int_t^T\!\int_E|h(r,e)|^p\l(\mathrm de)\mathrm dr  +  \(\int_t^T\!\int_E|h(r,e)|^2\l(\mathrm de)\mathrm dr\)^\frac p2  \].
 $$
 Here, $\dbE^{\cF_t}[\cdot]:=\dbE[\cdot\mid\cF_t]$.
 \el

 \begin{proof}
  Putting $\ds I_s:=\int_t^sh(r,e)\ti{\m}( \mathrm dr, \mathrm de)$, $s\in[t,T]$, and applying It\^o's formula to $|I_s|^p$, we get
 $$\ba{ll}
\ds \dbE^{\cF_t}[|I_s|^p]=\dbE^{\cF_t}\[\int_t^s\int_E\big(|I_r+h(r,e)|^p-|I_r|^p-p|I_r|^{p-2}\lan I_r,h(r,e)\ran \big) \l(\mathrm de)\mathrm dr\]\\
\ns\ds\les  p(p-1)\dbE^{\cF_t}\[\int_t^s\int_E\int_0^1(1-\th) |I_r+\th h(r,e)|^{p-2} |h(r,e)|^2\mathrm d\th \l(\mathrm de)\mathrm dr\] \\
 \ns\ds\les p(p-1)\dbE^{\cF_t}\[\int_t^s\int_E\int_0^1(1-\th) 2^{p-2}\big(|I_r|^{p-2} |h(r,e)|^2+|h(r,e)|^p\big)\mathrm d\th \l(\mathrm de)\mathrm dr\].\ \ea$$
By using that $I_\cd$ is a martingale and applying   Doob's inequality, we get
  $$\ba{ll}
\ds \dbE^{\cF_t}\[\sup\limits_{s\in[t,T]}|I_s|^p\]\les   C_p\dbE^{\cF_t}\[\int_t^T\int_E \big(|I_r|^{p-2} |h(r,e)|^2+|h(r,e)|^p\big)  \l(\mathrm de)\mathrm dr\]\\
%
\ds \les    C_p\dbE^{\cF_t}\[ \sup\limits_{r\in[t,T]}|I_r|^{p-2}\int_t^T\int_E  |h(r,e)|^2   \l(\mathrm de)\mathrm dr+\int_t^T\int_E   |h(r,e)|^p   \l(\mathrm de)\mathrm dr\] \\
 \ds \les    \frac12 \dbE^{\cF_t}\[ \sup\limits_{r\in[t,T]}|I_r|^{p }\]+C_p\dbE^{\cF_t}\[\(\int_t^T\int_E  |h(r,e)|^2   \l(\mathrm de)\mathrm dr\)^\frac p2+\int_t^T\int_E   |h(r,e)|^p   \l(\mathrm de)\mathrm dr\] . \ea$$
 Therefore, the stated result is obtained.

 \end{proof}

Next, we study the    global  integrability of  the solution of the infinite horizon  SDE with  jumps \eqref{SDE}.
For this, we introduce  the following abbreviations,
\bel{cp}\ba{ll}
\ds c_p: =p(p-1)\big(2^{-1}I_{\{2< p<3\}}+ 2^{p-4}I_{\{p=2\} \cup\{ p\ges 3\}}\big),\\
 \ds \eta_{b,p}:= 2\eta_b-(p-1)L_\si  ^2-\frac{2c_p}p L_{\g,2} ^2-  c_p  L_{\g,p} ^p,\q L_{\g, p} :=L_1 \(\!\int_E\!|l_\g(e)|^p\l(\mathrm de)\!\)^\frac 1p . \ea\ee

\br\label{eta-p-2}\sl
  We claim that,  for  any given  $p>2$,    $\eta_{b,p}>0 $  implies $\eta_{b,2}>0 $.
In fact, $\eta_{b,2} = 2\eta_b- L_\si  ^2-   L_{\g,2} ^2 $,   and
$$\ba{ll}
\ds \eta_{b,2}-\eta_{b,p}=2\eta_b- L_\si  ^2-   L_{\g,2} ^2 -2\eta_b+(p-1)L_\si  ^2+\frac{2c_p}p L_{\g,2} ^2+  c_p  L_{\g,p} ^p\\
 \ds =   (p-2)L_\si  ^2+ (\frac{2c_p}p -1)L_{\g,2} ^2+ c_p  L_{\g,p} ^p >0, \mbox{ when }  p>2.
\ea$$

\er

\bp\label{LemmaA-1}\sl Let $p\ges 2$ be arbitrarily  given, and  assume    {\bf (A1)}-{\bf (A3)}  and  $\eta_{b,p}>0$ hold true.   Then the solution $X(\cd)$ of the   infinite horizon  SDE  with jumps \eqref{SDE}
  belongs to  $ L_\dbF^p(t,\i;\dbR^n) $. Moreover, there exists some constant $C_p>0$ depending on $p$  and $\eta_b, L_b,L_\si,L_1,\ell_{\g}(\cd)$, such that,   for all $s\ges t,$ $\e>0$, $\dbP$-a.s.
\bel{SDE-est-1}\ba{ll}
 \ds {\rm (i)}\ \dbE^{\cF_t}\[\int_t^\i   | X_r  |^p\mathrm dr \]\les C_p \(|\xi |^p+\dbE^{\cF_t} \[\int_t^\i   \big( | b(r,0)|^p +  |  \si(r,0)  |^p   +|\g(r,\cd,0) |^p_{\l,p}  \big)    \mathrm dr \] \),\\
\ns\ds  {\rm (ii)}\  \dbE^{\cF_t}\big [ |X_s |^{p}\big] \les C_p\( |\xi |^p    + \dbE ^{\cF_t}\[\int_t^s   \big( | b(r,0)|^p +  |  \si(r,0)  |^p   +|\g(r,\cd,0) |^p_{\l,p}  \big)   \mathrm dr\]\)e^{-\frac p2(\eta_{b,p}-\e)(s-t)} , \\
\ns\ds  {\rm (iii)}\  \dbE^{\cF_t}\[\sup\limits_{s\in[t,\i)}   |X_s |^{p} \] \les    C_p\( |\xi |^p     + \dbE^{\cF_t} \[\int_t^\i  \big( | b(r,0)|^p +  |  \si(r,0)  |^p   +|\g(r,\cd,0) |^p_{\l,p}  \big)    \mathrm dr \]\) .
\ea \ee

\ep

\begin{proof}
First, note that, for any $r\ges t$, $e\in E$, $\dbP$-a.s.,
$$\ba{ll}
 \ds |X_{r-}  + \g(r,e,X_{r-} ) |^p- |X_{r-} |^{p}-p |X_{r-} |^{p-2}    \lan X_{r-} ,\g(r,e,X_{r-} )\ran \\
\ns\ds  \les  p(p-1)\int_0^1(1-\th) |X_{r-}  +\th \g(r,e,X_{r-} ) |^{p-2}|\g(r,e,X_{r-}) |^2   \mathrm d\th \\
\ns\ds  \les\left\{\ba{ll}
             \ds  \!\!\! p(p-1) 2^{-1}\big( |X_{r-} |^{p-2}|\g(r,e,X_{r-} ) |^2+|\g(r,e,X_{r-} ) |^p  \big), \q\ 2< p<3,   \\
             \ns\ds \!\!\!  p(p-1) 2^{p-4}\big( |X_{r-} |^{p-2}|\g(r,e,X_{r-} ) |^2+|\g(r,e,X_{r-} ) |^p  \big), \q p=2,\   p\ges 3,
              \ea  \right.\\
   \ns\ds  =  c_p\big( |X_{r-} |^{p-2}|\g(r,e,X_{r-} ) |^2+|\g(r,e,X_{r-} ) |^p \big ) ,

\ea$$
where $c_p $ is the constant defined  in \eqref{cp}.
Then,  from  It\^o's formula \eqref{X-ito-p}, for any  $s\ges t$, $\dbP$-a.s., we have
\bel{X-ito-1}\ba{ll}
 \ds     |X_s |^p \les   |\xi |^p  + \int_t^s\! \frac p2|X_r |^{p-2} \big(2\lan X_r  ,  b(r,X_r )\ran +(p-1)|  \si(r,X_r )  |^2 \big)     \mathrm dr\\
 \ns\ds\qq    +  \int_t^s       \int_Ec_p \big(|X_{r-} |^{p-2}|\g(r,e,X_{r-}  ) |^2 +|\g(r,e,X_{r-} ) |^p  \big  )        \m(\mathrm dr,\mathrm d e)\\
 \ns\ds\qq   +\int_t^sp|X_r |^{p-2}    \lan X_r , \si(r,X_r )    \mathrm dB_r\ran +\int_t^sp|X_{r-}|^{p-2}    \lan X_{r-} ,\int_E \g(r,e,X_{r-} ) \ti{\m}(  \mathrm dr,\mathrm d e)\ran.
\ea\ee

Taking the conditional expectation $\dbE^{\cF_t}[\cd] $ on the both sides of \eqref{X-ito-1}, we get
\no\bel{X-1}\ba{ll}
 \ds   \dbE^{\cF_t}[ |X_s |^{p}]\les  |\xi |^p + \dbE^{\cF_t}\[ \int_t^s c_p |\g(r,\cd,X_{r} ) |^p_{\l,p}          \mathrm dr \]  \\
 \ns\ds +\dbE^{\cF_t}\[\!  \int_t^s\!\frac p2|X_r |^{p-2} \(2\lan X_r,  b(r,X_r )\ran +(p-1)|  \si(r,X_r )  |^2+ \frac{2c_p}p   |\g(r,\cd,X_r ) |^2_{\l,2}\)     \mathrm dr \]  \!\! \\
 \ns\ds =|\xi |^p+ \dbE^{\cF_t}\[  \int_t^s\int_E  c_p  |\g(r,e,X_r )-\g(r,e,0)+\g(r,e,0 )|^p \l(\mathrm de)\mathrm dr\]\\
 \ns\ds \q +\dbE^{\cF_t}\[\int_t^s\frac p2|X_r |^{p-2} \(2\lan X_r,  b(r,X_r )-b(r,0)\ran +(p-1)|  \si(r,X_r )-\si(r,0)  |^2\\
 \ns\ds \hskip3 cm+\frac{2c_p}p  |\g(r,\cd,X_r )-\g(r,\cd,0) |_{\l,2}^2\)     \mathrm dr \] \\
  \ns\ds \q  +\dbE^{\cF_t}\[\int_t^s \frac p2|X_r |^{p-2} \( 2\lan X_r,   b(r,0)\ran +2(p-1) \lan  \si(r,X_r )-\si(r,0) , \si(r,0) \ran + (p-1)|   \si(r,0)  |^2\\
      \ns\ds\hskip3.5cm +\frac{4c_p}p \int_E \lan \g(r,e,X_r )-\g(r,e,0),\g(r,e,0) \ran\l(\mathrm de) +\frac{2c_p}p | \g(r,\cd,0) |_{\l,2}^2\)     \mathrm dr  \]    \\
 \ds\les  |\xi |^p +\frac p2  \(-2\eta_b+(p-1)L_\si  ^2+\frac{2c_p}p  L_{\g,2} ^2+\e\)\dbE^{\cF_t}\[\int_t^s  | X_r  |^p\mathrm dr\]\\
   \ns\ds\q     +C_{p,\e}\dbE^{\cF_t}\[\int_t^s   \big( | b(r,0)|^p +  |  \si(r,0)  |^p   +|\g(r,\cd,0) |^p_{\l,p}   \big)   \mathrm dr\]\\
 \ns\ds \q  + \dbE^{\cF_t}\[  \int_t^s\int_E   c_p \big( (L_1  l_\g(e))^p| X_r |^p+|\g(r,e,0 ) |^p+C\sum\limits_{k=1}^{[p]} (L_1l_\g(e))^{p-k}| X_r |^{p-k}|\g(r,e,0 ) |^k\big) \l(\mathrm de)\mathrm dr\]    \\
  \ns\ds\les  |\xi |^p +\frac p2  \(-2\eta_b+(p-1)L_\si  ^2+\frac{2c_p}p  L_{\g,2} ^2+  c_p  L_{\g,p} ^p+2\e\)\dbE^{\cF_t}\[\int_t^s  | X_r  |^p\mathrm dr\]\\
   \ns\ds\q     +C_{p,\e}\dbE^{\cF_t}\[\int_t^s   \big( | b(r,0)|^p +  |  \si(r,0)  |^p   +|\g(r,\cd,0) |^p_{\l,p}   \big)   \mathrm dr\],\q \dbP\mbox{-a.s.},
\ea\ee
  where we have used the Young inequality and  {the {Banach-Saks inequality}, for which we refer to Page 133 in \cite{Kuang}}. 

Recall that $ \eta_{b,p} >0 $. Thus,  taking $\e>0$   sufficiently small such that $\eta_{b,p}-2\e>0$ and letting $s\to\i$, we get, for some $C_p$ depending on $\e>0$,
\bel{X-p-1}
   \dbE^{\cF_t}\[  \int_t^\i   | X_r  |^p\mathrm dr\] \les C_p \(|\xi |^p+\dbE^{\cF_t}\[ \int_t^\i   \big( | b(r,0)|^p +  |  \si(r,0)  |^p   +|\g(r,\cd,0) |^p_{\l,p}  \big)    \mathrm dr \]\),\  \dbP\mbox{-a.s.}, \ t\in\dbR^+.\!\!\!
 \ee
Moreover, applying Gronwall's inequality to \eqref{X-1}, we see that  \eqref{SDE-est-1}-(ii) holds true.
%

Furthermore, from \eqref{X-ito-1}, by employing the   techniques already  used for \eqref{X-1}, we have,  $\dbP$-a.s.
\bel{X-2}\ba{ll}
 \ds   \dbE^{\cF_t}\[\sup\limits_{s\in[t,\i)}   |X_s |^{p} \] \les C_p\(  |\xi |^p +  \dbE^{\cF_t}\[  \int_t^\i  \big( | X_r|^p +| b(r,0)|^p +  |  \si(r,0)  |^p   +|\g(r,\cd,0) |^p_{\l,p} \big)   \mathrm dr \]\) \\
 \ns\ds\hskip4cm +\dbE^{\cF_t}\[\sup\limits_{s\in[t,\i)}\Big|\! \int_t^sp|X_r |^{p-2}    \lan X_r , \si(r,X_r )    \mathrm dB_r\ran \Big|\]\\
  \ns\ds\hskip4cm+\dbE^{\cF_t}\[\sup\limits_{s\in[t,\i)}\!\Big| \!\int_t^sp|X_{r-}  |^{p-2}    \lan X_{r-} ,\int_E \g(r,e,X_{r-}  ) \ti{\m}(  \mathrm dr,\mathrm d e)\ran \!\Big| \] \!\!\!\!\\
\ns\ds  \les  C_p\(  |\xi |^p +  \dbE^{\cF_t}\[  \int_t^\i  \big(|X_r |^ p + | b(r,0)|^p +  |  \si(r,0)  |^p   +|\g(r,\cd,0) |^p_{\l,p} \big)   \mathrm dr \] \)  + \frac12\dbE^{\cF_t}\[\sup\limits_{s\in[t,\i)}   |X_s|^{p} \],
\ea\ee
where we have used the  following estimates,
%
$$\ba{ll}
 \ds    \dbE^{\cF_t}\[\sup\limits_{s\in[t,\i)} \Big|\int_t^s p|X_r |^{p-2}    \lan X_r , \si(r,X_r )    \mathrm dB_r\ran \Big| \]  \les  2\dbE^{\cF_t}\[ \( \int_t^\i p^2 |X_r |^{2p-2}    | \si(r,X_r )|^2   \mathrm  dr \)^\frac12\] \\
\ns\ds\les 2 \dbE^{\cF_t}\[ \sup\limits_{s\in[t,\i)}|X_s |^\frac p2  \(\int_t^\i p^2|X_r |^{p-2}    | \si(r,X_r )|^2   \mathrm    dr \)^\frac12\] \\
\ns\ds\les  \frac14\dbE^{\cF_t}\[ \sup\limits_{s\in[t,\i)}|X_s |^ p  \]+C_p \dbE^{\cF_t}\[ \int_t^\i  |X_r |^{p-2}    | \si(r,X_r )|^2 \mathrm      dr  \] \\

\ns\ds\les  \frac14\dbE^{\cF_t}\[ \sup\limits_{s\in[t,\i)}|X_s |^ p  \]+ C_{p}\dbE^{\cF_t}\[ \int_t^\i (|X_r |^{p}    +  | \si(r,0 )|^p )  \mathrm dr   \] ,\q \dbP\mbox{-a.s.},
\ea$$
and
$$\ba{ll}
 \ds    \dbE^{\cF_t}\[\sup\limits_{s\in[t,\i)}  \Big| \int_t^sp |X_{r-}  |^{p-2}    \lan X_{r-}  ,\int_E \g(r,e,X_{r-}  ) \ti{\m}(  \mathrm dr,\mathrm d e)\ran \Big| \] \\
\ns\ds\les \frac14\dbE^{\cF_t}\[ \sup\limits_{s\in[t,\i)}|X_s |^ p  \]+ C_{p}\dbE^{\cF_t}\[ \int_t^\i (|X_r |^{p}    +  | \g(r,\cd,0 )|_\l^p )  \mathrm   dr   \],\q \dbP\mbox{-a.s.}
\ea$$
Finally, combining \eqref{X-2}   with \eqref{X-p-1},   we obtain \eqref{SDE-est-1}-(iii).
 %
\end{proof}

Furthermore, we  have the following
integrability of $X(\cd)$:

\bp\label{LemmaA-1.2}\sl
For any given $p\ges 2$, assume the conditions  {\bf (A1)}-{\bf (A4)} and   $\eta_{b,p}>0$.   Then,
 \bel{SDE-est-2}
  \ba{ll}
  \ds \dbE^{\cF_t}\[\(\int_t^\i   | X_r  |^2\mathrm dr\)^\frac p2 \]\les  C_p\Big\{|\xi|^p  +\dbE^{\cF_t}\[ \(\int_t^\i    \big(|b(r,0)|^2+  |\si(r,0)  |^2+  |\g(r,\cd,0 )|^2_{\l,2}\big)      \mathrm dr \)^\frac p2 \]\\
 \ns\ds\hskip4.8cm    +\dbE^{\cF_t}\[  \int_t^\i  \( | b(r,0)|^p +  |  \si(r,0)  |^p   +|\g(r,\cd,0) |^p_{\l,p} \)\mathrm dr \]  \Big\} .
 \ea
 \ee

\ep
%

\begin{proof}
When $p=2$, by  using the  techniques of \eqref{X-1}, we have from \eqref{X-ito-p},
$$\ba{ll}
 \ds    |X_s |^2    =|\xi|^2+ \int_t^s   \big(2\lan X_r ,  b(r,X_r )\ran +|  \si(r,X_r )  |^2 +|\g(r,\cd,X_r ) |_{\l,2}^2 \big)     \mathrm dr    \\
 \ns\ds\qq\q  +2 \int_t^s    \lan X_r , \si(r,X_r )    \mathrm dB_r\ran  +  \int_t^s\int_E   \big( |\g(r,e,X_{r-}  ) |^2+2\lan X_{r-}  , \g(r,e,X_{r-}  ) \ran  \big) \ti{\m}( \mathrm dr, \mathrm de) \\

  %
 \ns\ds \les  |\xi|^2+  (-2\eta_b+L_\si ^2+ L_{\g,2} ^2+\e) \int_t^s | X_r|^2      \mathrm dr   +C_\e \int_t^s   \big( |   b(r,0)  |^2+|   \si(r,0)  |^2   +|\g(r,\cd,0 ) |_{\l,2}^2\big)     \mathrm dr \\
  \ds\q  + \int_t^s\int_E   |\g(r,e,X_{r-} ) |^2   \m ( \mathrm dr, \mathrm de) - \int_t^s\int_E   |\g(r,e,X_r ) |^2  \l ( \mathrm de)  \mathrm dr   \\
  \ns\ds\q +2 \int_t^s    \lan X_r , \si(r,X_r )    \mathrm dB_r\ran  + \int_t^s\int_E   2\lan X_{r-}  , \g(r,e,X_{r-}  ) \ran    \ti{\m}( \mathrm dr, \mathrm de) ,\q \dbP\mbox{-a.s.}
\ea$$
Then,
\bel{X-2-1}\ba{ll}
 \ds    |X_s |^2 + ( \eta_{b,2} -\e) \int_t^s  | X_r|^2      \mathrm dr +\int_t^s  \int_E|\g(r,e,X_r)  |^2 \l(\mathrm de)\mathrm dr  \\
\ns\ds \les  C\(|\xi|^2   +\int_t^s   \big( |   b (r,0)  |^2+|   \si(r,0)  |^2  +|\g(r,\cd,0 ) |_{\l,2}^2  \big)     \mathrm dr  \\
 \ns\ds\q +\int_t^s  \int_E2\big( L_1^2l_\g(e)^2| X_{r-} |^2 + | \g(r,e,0) |^2\big)\m(\mathrm dr,\mathrm de) \\
 \ns\ds\q  +2 \int_t^s    \lan X_r , \si(r,X_r )    \mathrm dB_r\ran  +2 \int_t^s\int_E    \lan X_{r-} , \g(r,e,X_{r-} )    \ti{\m}( \mathrm dr, \mathrm de)\ran\),\q \dbP\mbox{-a.s.}
\ea\ee
For what follows, we put
  $$  \L_s:= |\xi|^p            +  \(\int_t^s    \big(|b(r,0)|^2+  |\si(r,0)  |^2  +|\g(r,\cd,0 ) |_{\l,2}^2 \big)      \mathrm dr \)^\frac p2 ,\q s\ges t.$$
  Due to Remark \ref{eta-p-2},
  we can  choose    $\e>0$  satisfying $\eta_{b,2}-\e>0$.
  Then,  for $p\ges2$, from \eqref{X-2-1} we have
{$$\ba{ll}
 \ds\dbE^{\cF_t}\[  \(    \int_t^s  |X_r|^2\mathrm dr \)^\frac p2\]
 \les  C_p\dbE^{\cF_t}\[ \L_s   + {\(\int_t^s  \!\int_El_\g(e)^2| X_{r-} |^2  \m(\mathrm dr,\mathrm de)  \)^\frac p2}  +\(\int_t^s\!  \int_E | \g(r,e,0) |^2\m(\mathrm dr,\mathrm de) \)^\frac p2 \\
 \ns\ds\hskip 4.5cm  + \Big|\int_t^s    \lan X_r , \si(r,X_r )    \mathrm dB_r\ran\Big|^\frac p2  + {\Big| \int_t^s\!\int_E    \lan X_{r-} , \g(r,e,X_{r-} )    \ti{\m}( \mathrm dr, \mathrm de)\ran\Big|^\frac p2}\]\\
 \ns\ds\hskip3.5cm  = C_p\dbE^{\cF_t}[ \L_s ]  +  C_p(a_s+b_s+c_s+d_s),
%
\ea$$
where, for $s\ges t$,
 $$\ba{ll}
 \ds a_s:=\dbE^{\cF_t}\[\(\int_t^s  \int_El_\g(e)^2| X_{r-} |^2  \m(\mathrm dr,\mathrm de)  \)^\frac p2\],\qq  b_s:= \dbE^{\cF_t}\[\(\int_t^s  \int_E | \g(r,e,0) |^2\m(\mathrm dr,\mathrm de) \)^\frac p2 \],\\
\ns\ds c_s:=\dbE^{\cF_t}\[\Big| \int_t^s    \lan X_r , \si(r,X_r )    \mathrm dB_r\ran \Big|^\frac p2   \] ,\qq   d_s: =\dbE^{\cF_t}\[   {\Big| \int_t^s\int_E    \lan X_{r-} , \g(r,e,X_{r-} )    \ti{\m}( \mathrm dr, \mathrm de)\ran\Big|^\frac p2}\].
\ea
$$

We point out that,   $l_\g(e)\in[0,1],$ $e\in E$,  assumed in {\bf (A1)},  implies that,
\bel{a-s}   a_s=\dbE^{\cF_t}\[\(\int_t^s  \int_El_\g(e)^2| X_{r-} |^2  \m(\mathrm dr,\mathrm de)  \)^\frac p2\]\les C_p\dbE^{\cF_t}\[  {\int_t^s  \int_El_\g(e)^2  |X_r|^p   \l (\mathrm de)\mathrm dr}\].
\ee
The proof  is similar to that of   inequality (3.15) in Li, Wei \cite{LW-SPA}.

For  $p\ges 2$, using the BDG inequality  and Lemma \ref{Lemma-2.2} yields
\bel{b-s}\ba{ll}
 \ds b_s
\les C_p \dbE^{\cF_t}\[\sup\limits_{\t\in[t,s]}\Big|\int_t^\t  \int_E\g(r,e,0)  \ti \m(\mathrm dr,\mathrm de)  \Big|^p \]\\
\ns\ds\q  \les C_p\dbE^{\cF_t}\[\(\int_t^s  \int_E | \g(r,e,0) |^2  \l(\mathrm de)\mathrm dr\)^\frac p2 + \int_t^s  \int_E  | \g(r,e,0) |^p  \l(\mathrm de)\mathrm dr \].
\ea
\ee
Using the BDG inequality again, we have
$$\ba{ll}
 \ds c_s
\les\dbE^{\cF_t}\[ \( \int_t^s  |X_r |^2\cd| \si(r,X_r ) |^2   \mathrm dr\)^\frac p4  \]\\
%
\ns\ds\q \les C_p\dbE^{\cF_t}\[\sup\limits_{r\in[t,s]}|X_r|^{\frac p2}\(\int_t^s    | \si(r,0 ) |^2  \mathrm dr\)^\frac p4  + \sup\limits_{r\in[t,s]}|X_r|^{\frac p2}\(\int_t^s    |  X_r  |^2    \mathrm dr\)^\frac p4\]\\
\ns\ds\q \les \dbE^{\cF_t}\[C_p\sup\limits_{r\in[t,s]}|X_r|^{p}+\(\int_t^s    | \si(r,0 ) |^2  \mathrm dr\)^\frac p2  + \frac12\(\int_t^s    |  X_r  |^2    \mathrm dr\)^\frac p2\].
\ea
$$
 Finally, for $d_s$, by using \eqref{a-s} and  \eqref{b-s}, we obtain
$$\ba{ll}
 \ds d_s
\les C_p \dbE^{\cF_t}\[   {\( \int_t^s\int_E   | X_{r-} |^2\cd| \g(r,e,X_{r-} ) | ^2   {\m}( \mathrm dr, \mathrm de) \)^\frac p4}\] \\
%
\ns\ds\q \les C_p \dbE^{\cF_t}\[  \sup\limits_{r\in[t,s]}| X_r |^\frac p2 {\( \int_t^s\int_E  | \g(r,e,X_{r-} ) | ^2   {\m}( \mathrm dr, \mathrm de)\)^\frac p4}\] \\
\ns\ds\q \les C_p \dbE^{\cF_t}\[  \sup\limits_{r\in[t,s]}| X_r |^  p+ {\( \int_t^s\int_E  | \g(r,e,X_{r-} ) | ^2   {\m}( \mathrm dr, \mathrm de) \)^\frac p2}\] \\
\ns\ds\q \les C_p \dbE^{\cF_t}\[  \sup\limits_{r\in[t,s]}| X_r |^  p +{\( \int_t^s\int_E \big( \ell_\g(e)^2|X_{r-}|^2 + | \g(r,e,0 ) | ^2 \big) {\m}( \mathrm dr, \mathrm de)\ran\)^\frac p2}\]\\
\ns\ds\q  \les C_p \dbE^{\cF_t}\[  \sup\limits_{r\in[t,s]}| X_r |^  p \]+C_p (a_s+b_s) .

\ea$$
}
By combining the above estimates,
 we get
$$\ba{ll}
 \ds\dbE^{\cF_t}\[  \(    \int_t^s  |X_r|^2\mathrm dr \)^\frac p2\]
\les \frac 12 \dbE^{\cF_t}\[\(\int_t^s    |  X_r  |^2    \mathrm dr\)^\frac p2\] \\
 \ns\ds\hskip 4cm+C_p\dbE^{\cF_t}\[ \L_s      +  \int_t^s  |\g(r,\cd,0 ) |_{\l,p}^ p     \mathrm dr   + \sup\limits_{r\in[t,s]} |X_r|^p   +   \int_t^s  | X_r |^p \mathrm dr \] .
\ea$$
%
Finally, letting $s\to\i$, from the latter estimate and    \eqref{SDE-est-1}-(i), (iii)   we get \eqref{SDE-est-2}.
%
%

\end{proof}

From  the Propositions  \ref{LemmaA-1} and  \ref{LemmaA-1.2}, we also get the following  continuous  dependence of the solutions of the infinite horizon  SDEs with  jumps \eqref{SDE}.

 \bc\label{LemmaA-2}\sl  Denote  by  $\bar X$ the solution to \eqref{SDE} corresponding to another initial state $\bar \xi\in L_{\cF_t}^p(\Omega;\dbR^n)$ and coefficients $\bar b,\bar \si,\bar \g$.

 {\rm(i)} Assume the conditions of Proposition \ref{LemmaA-1}  hold.  If the coefficients  $\bar b,\bar \si,\bar \g$
     satisfy   {\bf (A1)}-{\bf (A3)} and  $\eta_{\bar b,p}>0$,   then we have, $\dbP$-a.s.
$$\ba{ll}
 \ds  \dbE^{\cF_t} \[\sup\limits_{s\in[t,\i)}   |X_r -\bar X_r|^{p}+ \int_t^\i  |X_r-\bar X_r|^p\mathrm dr \]\les C_p |\xi-\bar \xi|^p  \\
\ns\ds \q  +    C_p \dbE ^{\cF_t} \[ \int_t^\i   \big(|b(r,\bar X_r)  - \bar b(r,\bar X_r)|^p+  |\si(r,\bar X_r) - \bar \si(r,\bar X_r) |^p + |\g(r,\cd,\bar X_r)  - \bar \g(r,\cd,\bar X_r)|^p_{\l,p}\big)   \mathrm dr\].
\ea$$

{\rm(ii)} Under  the conditions of Proposition \ref{LemmaA-1.2}, if the coefficients  $\bar b,\bar \si,\bar \g$ satisfy   {\bf (A1)}-{\bf (A4)} and  $\eta_{\bar b,p}>0$, then,
$$
  \ba{ll}
  \ds \dbE^{\cF_t}\[\(\int_t^\i   | X_r -\bar X_r |^2\mathrm dr\)^\frac p2 \]\les C_p  |\xi-\bar \xi |^p  \\
  \ns\ds   \  +  C_p  \dbE^{\cF_t} \[\(\!\int_t^\i\!   \big( | b(r,\bar X_r) - \bar b(r,\bar X_r)|^2 +  |  \si(r,\bar X_r)  - \bar \si(r,\bar X_r) |^2   +|\g(r,\cd,\bar X_r)  - \bar \g(r,\cd,\bar X_r)|^2_{\l,2}  \big)    \mathrm dr\)^\frac p2\]\\
 \ns\ds   \ +  C_p  \dbE^{\cF_t} \[ \int_t^\i   \big( | b(r,\bar X_r) - \bar b(r,\bar X_r)|^p +  |  \si(r,\bar X_r) - \bar \si(r,\bar X_r)  |^p  +|\g(r,\cd,\bar X_r)  - \bar \g(r,\cd,\bar X_r)|^p_{\l,p} \big) \mathrm dr \]   ,\ \dbP\mbox{-a.s.}
 \ea
$$

\ec

\begin{proof}
Setting $\widehat X:=X-\bar X$, let us consider the following infinite horizon SDE  with jumps
$$
 \left\{\ba{ll}
 \ds\!\!\!\! \mathrm  d\widehat X_s=\widehat b(s, \widehat X_s)\mathrm ds+\widehat \si(s, \widehat X_s) \mathrm dB_s+\int_E\widehat \g(s,e, \widehat X_{s-})\ti{\m}(\mathrm ds,\mathrm de),\ s\ges t,\\
 \ds\!\!\!\! \widehat X_t=\widehat \xi,
 \ea\right.$$
where  $\widehat \xi:=\xi-\bar\xi$, $\widehat b(s,x):=b(s,x+\bar X_s)-\bar b(s,\bar X_s) $, $\widehat \si(s,x):=\si(s,x+\bar X_s)-\bar \si(s,\bar X_s) $ and  $\widehat \g(s,e,x):=\g(s,e,x+\bar X_{s-})-\bar \g(s,e, \bar X_{s-}) $, $(\o,s,x,e)\in\Omega\times [0,\i)\times\dbR^n\times E$.
Then, by applying the Propositions   \ref{LemmaA-1} and   \ref{LemmaA-1.2} to the above SDE, we  get the both stated results.

\end{proof}

 %
%
%

\subsection{Infinite horizon backward stochastic differential equations with jumps}
In this subsection, we consider the following infinite horizon   BSDE with jumps,
 \bel{BSDE}
\ba{ll}
 \ds\!\!\!  Y _s\!=\!Y_T+\int_s^Tf\big(r,X _r,Y _r,Z _r,  \int_E\psi( r,e,K _r(e))  \l(\mathrm{d}e) \big)\mathrm dr-\int_s^TZ _r\mathrm dB_r\\
 \ds\qq\qq\qq  -\int_s^T\int_EK _r(e)\ti{\m}(\mathrm dr,\mathrm de),\ s\in[t,T],  \mbox{ for all }    0\les t  \les T<\i,
 \ea\ee
 where $X$ satisfies  SDE \eqref{SDE}.
For $p\ges 2,$ we assume   the driver  $f:\Omega\times[0,\i)\times\dbR^n\times \dbR^m\times \dbR^{m\times d}\times \dbR^m\to\dbR^m$ and the function $\psi:\Omega\times[0,\i) \times E \times \dbR^m\ \to \dbR^m$  to satisfy

 \ss

  {\bf (B1)}
  (i)  There exist  nonnegative  constants $L_x, L_y, L_z, L_ k$ such that, for all $s\ges 0$, $(x,y,z,k)$, $(x',y',z',k')\in \dbR^n\times\dbR^m\times\dbR^{m\times d}\times   \dbR^m  $,
   $$ |f(s,x,y,z,k)-f(s,x',y',z',k')| \les L_x|x-x'|+L_y  |y-y'| +L_z |z-z'|+L_k|k-k'|, \ \dbP\mbox{-a.s.};$$
 \hskip 1.5cm (ii) $f(\cd,0,0,0,0)\in L_\dbF^p(\Omega;L^2(0,\i;\dbR^m))$.

 {\bf (B2)} There exists a constant $\eta_f>0$  such that, for all $s\ges 0$, $(x,z,k)\in \dbR^n\times\dbR^{m\times d}\times   \dbR^m  $, $y, y'\in\dbR^m$,
  $$\ds\lan f(s,x,y,z,k )-f(s,x,y',z,k ),y-y'\ran \les -\eta_f|y-y'|^2.$$

 {\bf (B3)}   There exists a mapping  $\rho(\cd)\in   L_\l^2(E;\dbR^+) $ such that, for all $s\ges 0$, $k,k'\in \dbR^m$, $e\in E$,   $\dbP$-a.s.,
$$|\psi(s,e,k)-\psi(s,e,k')|\les \rho(e)|k-k'| \mbox{ and }  \dbE\[ \int_0^\i|\psi(r,\cd,0) |^2_{\l,1}\,\mathrm dr\]<\i  .$$

%
%
%
%
%
Let us state the following  result, which is similar to  Corollary 2.3 in \cite{Yu-2017}.
\bl\label{Yto0} \sl Let the conditions {\bf (A1)}, {\bf (A2)} and {\bf (B1)}  hold  with $p=2$. If $(Y ,Z ,K )\in \sS^2_\dbF(  t,\i)  $ is the unique solution of BSDE with jumps  \eqref{BSDE}, then $\lim\limits_{s\to \i}\dbE^{\cF_t}[|Y(s)|^2]=0$, $\dbP$-a.s.

\el

Next,  we    study an a prior estimate of the infinite horizon  BSDE with jumps \eqref{BSDE}.

\bp\label{Lemma-BSDE-est-2}\sl
Given any $p\ges 2$, assume the conditions of Proposition \ref{LemmaA-1.2},  {\bf (B1)}-{\bf (B3)}   and  $\bar\eta_f:= \eta_f -\frac{L_z ^2+L_k^2|\rho(\cd)|_{\l,2}^2}{2}>0$. Let $(Y ,Z ,K )\in \sS^p_\dbF(t,\i)$   be a solution to  BSDE \eqref{BSDE}. Then, we have, $\dbP$-a.s.
 \bel{Y-p-sup-1}\ba{ll}
 \ds \! \dbE^{\cF_t} \[\sup\limits_{s\in[t,\i)} |Y_s |^p+\(\int_t^\i \!|Y _r|^2\mathrm dr \)^\frac p2+\(\int_t^\i \!|Z _r|^2\mathrm dr \)^\frac p2 +\(\int_t^\i    |K_r(\cd)| ^2_{  \l,2}  \mathrm dr    \)^\frac p2 \]  \\
\ns\ds \les C_p   \Big\{ | \xi |^p  +\dbE^{\cF_t} \[     \(\int_t^\i\!\big(| b(r,0)|^2 +  |  \si(r,0)  |^2   +|\g(r,\cd,0) |^2_{\l,2} +|\psi(r,\cd,0) |^2_{\l,1}  +|f(r,0,0,0,0 )|^2\big)\mathrm dr \)^\frac p2 \]  \!\!\!\! \!\! \\
\ns\ds\qq \q+   \dbE^{\cF_t} \[  \int_t^\i\!   \big( | b(r,0)|^p +  |  \si(r,0)  |^p   +|\g(r,\cd,0) |^p_{\l,p}   \big)\mathrm dr       \]  \Big\}  .
\ea\ee
\ep

\begin{proof}
Note that, when $p\ges 2$, for  any process,  $\f(\cd) \in L_\dbF^p(\Omega;L^2(0,\i;\dbH))  $ implies $\f(\cd) \in L_\dbF^2(0,\i;\dbH) $.  Therefore, under the assumed conditions with  $p\ges2$,  the following  conclusion can be done for $p=2$.  In particular, we always have  $\lim\limits_{s\to \i}\dbE^{\cF_t}[|Y_s|^2]=0$,   $\dbP$-a.s.

For any $T>t$, by applying It\^o's formula to  $ |Y_\cd |^2 $ on $[t,T]$,  for all $s\in[t,T]$, we have
$$\ba{ll}
 \ds |Y_s |^2 +   \int_s^T  |Z _r|^2 \mathrm dr +\int_s^T \int_E|K_r(e)| ^2 \m(\mathrm dr,\mathrm d e )\\
   \ns\ds=|Y_T|^2-2\int_s^T  \lan Y_r,Z_r\mathrm dB_r\ran  - 2\int_s^T \int_E\lan Y_{r-},K_r(e) \ti{\m}(\mathrm dr,\mathrm de)\ran \\
\ns\ds\  +2 \int_s^T  \lan  Y_r, f\big(r,X _r,Y _r,Z _r,\int_E \psi(r,K _r(e),e)\l(\mathrm{d}e)\big) \ran\mathrm  dr \\
\ns\ds \les |Y_T |^2-2\int_s^T  \lan Y_r,Z_r\mathrm dB_r\ran  - 2\int_s^T \int_E\lan Y_{r-},K_r(e) \ti{\m}(\mathrm dr,\mathrm de)\ran-2\eta_f \int_s^T|Y_r |^{2 }\mathrm dr\\
\ns\ds\  +2 \int_s^T \!|Y_r |  \big(L_x|X_r|+L_z |Z_r|+L_k\int_E |\psi(r,e,0)|\l(\mathrm{d}e) +L_k\int_E |K_r(e) |\rho(e)\l(\mathrm{d}e)  +|f(r,0,0,0,0)| \big) \mathrm dr  \\
\ns\ds \les|Y_T |^2-2\int_s^T  \lan Y_r,Z_r\mathrm dB_r\ran  - 2\int_s^T \!\!\int_E\lan Y_{r-},K_r(e) \ti{\m}(\mathrm dr,\mathrm de)\ran\\
\ns\ds\    - \big( 2\eta_f - L_z ^2-L_k^2|\rho(\cd)|_{\l,2}^2 -8\e\big)\int_s^T |Y_r|^2  \mathrm dr  +  \frac{1 }{2\e}  \int_s^T\big(L_x^2|X_r |^{2 }+L_k^2  |\psi(r,\cd,0)|_{\l,1}^2+ |f(r,0,0,0,0 )|^2\big)\mathrm dr  \\
\ns\ds\   + \frac { L_z ^2}{  L_z ^2+\e }\int_s^T |Z _r|^2\mathrm dr+ \frac{  L_k^2|\rho(\cd)|_{\l,2}^2}{ L_k^2|\rho(\cd)|_{\l,2}^2+\e  }\int_s^T |K _r(\cd)|_{\l,2}^2\mathrm dr  ,\q \dbP\mbox{-a.s.}
\ea$$
Then, setting $\bar\eta_f:= \eta_f -\frac{L_z ^2+L_k^2|\rho(\cd)|_{\l,2}^2}{2} $, we have
\bel{Y-2-12}\ba{ll}
 \ds |Y_s |^2 +2 \big( \bar\eta_f -4\e\big)\int_s^T |Y_r|^2  \mathrm dr+  \frac { \e}{  L_z ^2+\e }\int_s^T |Z _r|^2\mathrm dr + \int_s^T \int_E|K_r(e)| ^2 \ti\m(\mathrm dr,\mathrm d e )\\
   \ns\ds  +\frac{  \e}{ L_k^2|\rho(\cd)|_{\l,2}^2+\e  }\int_s^T |K _r(\cd)|_{\l,2}^2\mathrm dr    \\
\ns\ds \les|Y_T |^2      +  \frac{1 }{2\e}  \int_s^T\!\!\big(L_x^2|X_r |^{2 }+L_k^2|\psi(r,\cd,0)|_{\l,1}^2+ |f(r,0,0,0,0 )|^2\big)\mathrm dr \\
 \ns\ds\q -2\int_s^T  \lan Y_r,Z_r\mathrm dB_r\ran  - 2\int_s^T  \int_E \lan Y_{r-},K_r(e) \ti{\m}(\mathrm dr,\mathrm de)\ran.
\ea\ee

Thanks to   $\bar\eta_f >0$, by taking $\e>0$ such that $\bar\eta_f   -4\e>0$,  we have, for all $s<T$,   $\dbP$-a.s.
\bel{Y-2-11}\ba{ll}
 \ds |Y_s |^2 + \int_s^T \big(|Y_r|^2    +    |Z _r|^2 +|K _r(\cd)|_{\l,2}^2\big)  \mathrm dr   + \int_s^T \int_E|K_r(e)| ^2 \ti\m(\mathrm dr,\mathrm d e )   \\
\ns\ds \les C\(|Y_T |^2      +    \int_s^T\!\!\big( |X_r |^{2 }+|\psi(r,\cd,0)|_{\l,1}^2+ |f(r,0,0,0,0 )|^2\big)\mathrm dr \\
 \ns\ds\qq- \int_s^T  \lan Y_r,Z_r\mathrm dB_r\ran  -  \int_s^T   \int_E \lan Y_{r-},K_r(e) \ti{\m}(\mathrm dr,\mathrm de)\ran\).
\ea\ee
Thus,  for all $\t\in [t, s],$ $\dbP$-a.s.,
\bel{Y-2-2}\ba{ll}
 \ds \dbE^{\cF_\t}\[ |Y_s |^2 +    \int_s^T\big( |Y_r|^2  +|Z _r|^2 + |K_r(\cd)|_{\l,2}^2\big) \mathrm dr   \] \\
\ns\ds \les  C\dbE ^{\cF_\t}\!\[ |Y_T |^2 +   \int_s^T\big( |X_r |^{ 2} + |\psi(r,\cd,0)|_{\l,1}^2+  |f(r,0,0,0,0 )|^2\big)\mathrm dr \]  .\!\!
\ea\ee

Moreover, from the integral form of the BSDE with jumps \eqref{BSDE}, using   \eqref{Y-2-2}, for any $T>\t$, we get
\bel{Y-2-2-1}\ba{ll}
   \ds\!\!\dbE^{\cF_\t}\[ \sup\limits_{s\in[\t,T]}|Y_s|^2\]\\
    \ns\ds \les C\dbE^{\cF_\t}\[ |Y_T |^2 + \int_\t^T  \big(  |f\big(r,0,0,0,0)  |^2+|\psi(r,\cd,0)|_{\l,1}^2+ |X _r|^2+|Y _r|^2+|Z _r|^2+ |K _r(\cd) |_{\l,2}^2 \big)\mathrm dr  \\
  \ns\ds \les  C\dbE ^{\cF_\t}\[ |Y_T |^2 +   \int_\t^T\big( |X_r |^{ 2}  + |\psi(r,\cd,0)|_{\l,1}^2+|f(r,0,0,0,0 )|^2\big)\mathrm dr \] ,\q \dbP\mbox{-a.s.}
\ea\ee
Therefore, letting $T\to\i$ in \eqref{Y-2-2} and \eqref{Y-2-2-1}, using \eqref{SDE-est-1}  and Lemma \ref{Yto0}, we get
$$\ba{ll}
 \ds\dbE^{\cF_\t}\[ \sup\limits_{s\in[\t,\i]}|Y_s|^2 + \int_\t^\i\big( |Y_r|^2  +|Z _r|^2 + |K_r(\cd)|_{\l,2}^2\big) \mathrm dr  \]  \\
     \ns\ds \les C \dbE^{\cF_\t}\[ |X_\t|^2+ \int_\t^\i\!\big(| b(r,0)|^2 +  |  \si(r,0)  |^2   +|\g(r,\cd,0) |^2_{\l,2}+    |\psi(r,\cd,0)|_{\l,1}^2+|f (r,0,0,0,0) |^2 \big)\mathrm dr   \],\  \dbP\mbox{-a.s.}
\ea$$
Hence \eqref{Y-p-sup-1}  for  $p=2$ holds true.

%

Next, taking $\t=s$   in the latter estimate,   we get, for all  $0\les t\les s<\i$,
\bel{Y-2}   |Y_s |^2   \les  C\(|X_s |^2 +\dbE ^{\cF_s}\[   \int_t^\i\big(| b(r,0)|^2 +  |  \si(r,0)  |^2   +|\g(r,\cd,0) |^2_{\l,2}  +|\psi(r,\cd,0)|_{\l,1}^2+|f(r,0,0,0,0 )|^2\big)\mathrm dr \]\).
 \ee
By setting $\ds  M_s:=\dbE ^{\cF_s}\[  \int_t^\i\big(| b(r,0)|^2 +  |  \si(r,0)  |^2   +|\g(r,\cd,0) |^2_{\l,2}  + |\psi(r,\cd,0)|_{\l,1}^2 +|f(r,0,0,0,0 )|^2\big)\mathrm dr \]$, $s\ges t $,
and using Doob's inequality for $p>2$, we get
 \bel{Y-p-sup}\ba{ll}
\ds  \dbE^{\cF_t}\[\sup\limits_{s\in[t,\i)} |Y_s |^p \]=  \dbE^{\cF_t}\[\(\sup\limits_{s\in[t,\i)} |Y_s |^2 \)^\frac p2\]  \les  C_p\dbE ^{\cF_t} \[ \sup\limits_{s\in[t,\i)} |X_s| ^p\] +C_p\dbE ^{\cF_t} \[\(\sup\limits_{s\in[t,\i)} |M_s|\)^\frac p2\] \\
\ns\ds \les C_p\dbE ^{\cF_t} \[ \sup\limits_{s\in[t,\i)} |X_s| ^p\] +  C_p\(\frac{p}{p-2}\)^\frac p2\dbE ^{\cF_t} \big[   |M_\i|^\frac p2\big] \\
\ns\ds  \les  C_p\dbE ^{\cF_t} \[ \sup\limits_{s\in[t,\i)} |X_s| ^p\]\\
 \ns\ds \ +C_p\dbE ^{\cF_t} \[ \dbE ^{\cF_\i}\!\[ \!  \(\!\int_t^\i\!\!\big(| b(r,0)|^2 +  |  \si(r,0)  |^2   +|\g(r,\cd,0) |^2_{\l,2}+  |\psi(r,\cd,0)|_{\l,1}^2+|f(r,0,0,0,0 )|^2\big)\mathrm dr \)^\frac p2\]\]  \!\!\!\!\!\!\!\!\!\!\!\!\!\!\! \!\! \\
\ns\ds  \les  C_p\Big\{ |\xi |^p     + \dbE^{\cF_t} \[\int_t^\i  \big( | b(r,0)|^p +  |  \si(r,0)  |^p   +|\g(r,\cd,0) |^p_{\l,p}  \big)    \mathrm dr \] \!\!\! \\
 \ns\ds \qq\q+ \dbE ^{\cF_t}  \[   \(\int_t^\i\big(| b(r,0)|^2 +  |  \si(r,0)  |^2   +|\g(r,\cd,0) |^2_{\l,2} + |\psi(r,\cd,0)|_{\l,1}^2 +|f(r,0,0,0,0 )|^2\big)\mathrm dr \)^\frac p2\]\Big\}.\!\!\!
\ea\ee
%

  Further,   from \eqref{Y-2-11}, by using  the  BDG inequality, we obtain 
$$\ba{ll}
  \ds  \dbE^{\cF_t} \[   \( \int_s^T |Y_r|^2  \mathrm dr\)^\frac p2+  \( \int_s^T |Z _r|^2\mathrm dr \)^\frac p2+ \( \int_s^T \int_E|K_r(e)| ^2   \m(\mathrm dr,\mathrm d e )\)  ^\frac p2  \]   \\
   %
%
%
%
   \ns\ds \les \dbE^{\cF_t} \[  \(  \int_s^T |Y_r|^2  \mathrm dr+   \int_s^T |Z _r|^2\mathrm dr + \int_s^T \int_E|K_r(e)| ^2 \ti\m(\mathrm dr,\mathrm d e )  + \int_s^T |K _r(\cd)|_{\l,2}^2\mathrm dr  \)^\frac p2\]\\
\ns\ds  \les C_p\dbE^{\cF_t} \[  \(|Y_T |^2      +    \int_s^T\big( |X_r |^{2 }+|\psi(r,\cd,0)|_{\l,1}^2+ |f(r,0,0,0,0 )|^2\big)\mathrm dr \\
 \ns\ds \q - \int_s^T  \lan Y_r,Z_r\mathrm dB_r\ran  -  {\int_s^T \int_E\lan Y_{r-},K_r(e) \ti{\m}(\mathrm dr,\mathrm de)\ran \)^\frac p2}\]\\
\ns\ds  \les C_p \dbE^{\cF_t} \[  |Y_T |^p      +    \( \int_s^T\big( |X_r |^{2 }+|\psi(r,\cd,0)|_{\l,1}^2+ |f(r,0,0,0,0 )|^2\big)\mathrm dr \)^\frac p2   \\
\ns\ds \q+\(\int_s^T  |Y_r|^2\cd|Z_r|^2\mathrm dr\)^\frac p4     {+\(\int_s^T \int_E| Y_{r-}|^2\cd|K_r(e)|^2 \m(\mathrm dr,\mathrm de)\)^\frac p4  \]}  \\
\ns\ds  \les C_p \dbE^{\cF_t} \[  |Y_T |^p      +    \( \int_s^T\big( |X_r |^{2 }+|\psi(r,\cd,0)|_{\l,1}^2+ |f(r,0,0,0,0 )|^2\big)\mathrm dr \)^\frac p2\\
\ns\ds \q+\sup\limits_{r\in [t,T]}|Y_r|^{\frac p2} \cd\(\int_s^T   |Z_r|^2\mathrm dr\)^\frac p4   +\sup\limits_{r\in [t,T]}|Y_r|^{\frac p2}\cd \(\int_s^T \int_E|K_r(e)|^2 \m(\mathrm dr,\mathrm de)\)^\frac p4  \] \\
\ns\ds  \les C_{p,\epsilon} \dbE^{\cF_t} \[  \sup\limits_{r\in [t,T]}|Y_r|^{p}+     \( \int_s^T\big( |X_r |^{2 }+|\psi(r,\cd,0)|_{\l,1}^2+ |f(r,0,0,0,0 )|^2\big)\mathrm dr \)^\frac p2\]\\
\ns\ds\ \q +\epsilon \dbE^{\cF_t}\[ \(\int_s^T   |Z_r|^2\mathrm dr\)^\frac p2 +\(\int_s^T \int_E|K_r(e)|^2\m(\mathrm dr,\mathrm de)\)^\frac p2 \].
\ea$$
%
%

Thus, by  choosing $\epsilon>0$ small enough, we get
$$\ba{ll}
  \ds  \dbE^{\cF_t} \[\(\int_s^T |Y _r|^2\mathrm dr \)^\frac p2+\(\int_s^T |Z _r|^2\mathrm dr \)^\frac p2 + \( \int_s^T \int_E|K_r(e)| ^2   \m(\mathrm dr,\mathrm d e )\)  ^\frac p2 \] \\
\ns\ds \les C_p \dbE^{\cF_t} \[  \sup\limits_{r\in [t,T]}|Y_r|^{p}+   \( \int_s^T\big( |X_r |^{2 }+|\psi(r,\cd,0)|_{\l,1}^2+ |f(r,0,0,0,0 )|^2\big)\mathrm dr \)^\frac p2\].
\ea$$
By combining the latter estimate   with  $$ \(\frac 2p\)^\frac p2\dbE^{\cF_t} \[   \(  \int_s^T \int_E|K_r(e)| ^2  \l(\mathrm de)\mathrm dr ) \)^\frac p2\] \les   \dbE^{\cF_t} \[\(\int_s^T \int_E|K_r(e)|^2\m(\mathrm dr,\mathrm de)\)^\frac p2 \]$$
(see, Lemma 3.1 in Li, Wei \cite{LW-SPA}), we get
$$\ba{ll}
  \ds  \dbE^{\cF_t} \[\(\int_t^T |Y _r|^2\mathrm dr \)^\frac p2+\(\int_t^T |Z _r|^2\mathrm dr \)^\frac p2 +\(\int_t^T  |K_r(\cd)| ^2_ {\l,2}  \mathrm d r    \)^\frac p2 \] \\
\ns\ds \les C_p \dbE^{\cF_t} \[  \sup\limits_{r\in [t,T]}|Y_r|^{p}+  \( \int_s^T\big( |X_r |^{2 }+|\psi(r,\cd,0)|_{\l,1}^2+ |f(r,0,0,0,0 )|^2\big)\mathrm dr \)^\frac p2\].
%
%
\ea$$
By letting $T\to\i$, and using \eqref{SDE-est-2} and \eqref{Y-p-sup}, the stated \eqref{Y-p-sup-1} is proved.

 \end{proof}


From the above result we  get the following  continuous dependence of the solution   of infinite horizon BSDEs with jumps  \eqref{BSDE}.

\bc\label{Y-Y}\sl Let the conditions of Corollary  \ref{LemmaA-2}-(ii) and Proposition \ref{Lemma-BSDE-est-2} hold.
  Denoting by $(\bar Y,\bar Z,\bar K)$  the solution to \eqref{BSDE} corresponding to  $\bar X$ and  another coefficients  $\bar f$, $\bar \psi$  satisfying {\bf (B1)}-{\bf (B3)},    $\bar\eta_{\bar f} >0$,   we have
%
$$
\ba{ll}
 \ds    \dbE^{\cF_t}\[\sup\limits_{s\in[t,\i)} |Y_s -\bar Y_s |^p+\(\int_t^\i|Y _r-\bar Y_r|^2\mathrm d r \)^\frac p2 +\(\int_t^\i|Z _r-\bar Z_r|^2\mathrm d r \)^\frac p2  + \(\int_t^\i |K _r(\cd)-\bar K_r(\cd)|^2_{\l,2} \mathrm dr  \)^\frac p2 \] \\
  \ns\ds \les C_{p} \Big\{ |\xi-\bar \xi|^p   +   \dbE^{\cF_t}  \[   \int_t^\i   \big( |b(r,X_r)-\bar b(r,\bar X_r)|^p +|  \si(r,X_r)  -\bar \si(r,\bar X_r)|^p  +|\g(r,\cd,X_r )-\bar \g(r,\cd,\bar X_r)|^p_{\l,p}   \big) \mathrm dr\]\\
    \ns\ds\qq + \dbE^{\cF_t}  \[\(\int_t^\i   \big( |b(r,X_r)-\bar b(r,\bar X_r)|^2 +|  \si(r,X_r)  -\bar \si(r,\bar X_r)|^2  +|\g(r,\cd,X_r )-\bar \g(r,\cd,\bar X_r)|^2_{\l,2}\\
    \ns\ds\qq\qq    +|f(r,\bar X_r,\bar Y_r,\bar Z_r,\int_E\psi(r,\bar K_r(e),e)  \l(\mathrm de))-\bar f(r,\bar X_r,\bar Y_r,\bar Z_r,\int_E\bar\psi(r,\bar K_r(e),e) \l(\mathrm de))|^2\big) \mathrm dr\)^\frac p2\] \Big\} .
\ea$$
\ec

\bl\label{Lemma-BSDE-1}\sl
Assume {\bf (A1)}-{\bf (A4)}, {\bf (B1)}-{\bf (B3)}  with some  $p\ges 2$,  $\eta_{b,p}>0$ and $\bar\eta_{f}>0$. Then the   infinite horizon BSDE with jumps \eqref{BSDE} admits a unique solution $(Y,Z,K )\in \sS^p_\dbF(t,\i)$.

\el

  This result can   be proved by   employing the method used in  \cite{Peng-Shi-2000, WY-2021} and Corollary \ref{Y-Y}, so that we shall not repeat the details here.
Next we give the comparison theorem for  the infinite horizon BSDE with jumps \eqref{BSDE}.

\bt\label{Th-Com}\sl(Comparison Theorem)  Let $m=1$, $p\ges 2$. Denote by  $(Y^1,Z^1,K^1)$, $(Y^2,Z^2,K^2)\in \sS^p_\dbF(t,\i)$  the solution  of \eqref{BSDE} corresponding to the driver  $f^1$, $f^2$ with the same $\psi$, respectively, which are supposed to satisfy  {\bf (B1)}-{\bf (B3)}. Assume  {$\bar\eta_{f^1} :=  \eta_{f^1} -\frac{L_z ^2+L_k^2|\rho(\cd)|_{\l,2}^2}{2}>0$}    and

{\bf (S)} there exists a constant $\varrho>0$ such that, for all $(\o, s,x,y,z)\in \Omega\times [0,\i)\times  \dbR^n\times\dbR\times \dbR^d$, $k,k'\in \dbR$,  $k\neq k'$ and $e\in E$,
 $$\ba{ll}
\ds 0\les \frac{\psi(s,k,e)-\psi(s,k',e)}{k-k'} \les \varrho(1\wedge|e|) ,\\
\ns\ds\frac{f^1(s,x,y,z,k)-f^1(s,x,y,z,k')}{k-k'}>-\frac 1\varrho ,\ \dbP\mbox{-a.s.}  \ea$$
If for all $(\o,s,x,y,z,k)\in \Omega\times [0,\i) \times\dbR^n\times\dbR\times \dbR^d\times \dbR$, $f^1(s,x,y,z,k)\les f^2(s,x,y,z,k)$, $\dbP$-a.s.,   then $Y^1_t\les Y^2_t $, $t\ges0$, $\dbP$-a.s.

\et

\begin{proof}
 Setting $(\hat Y,\hat Z,\hat K):=(Y^1 -  Y^2 ,Z^1 - Z^2 ,K^1 -K^2 )$, we have
$$
\ba{ll}
 \ds \mathrm  d\hat Y_s=-\(f^1\big(s,X _s,Y _s^1,Z _s^1,\int_E \psi(r,K _s^1(e) , e)\l(\mathrm{d}e)\big)-f^2\big(s,X _s, Y_s^2, Z _s^2,\int_E \psi(r,K _s^2(e), e)\l(\mathrm{d}e)\big)\)\mathrm ds\\
 \ns\ds\qq\q +\hat Z_s\mathrm dB_s+\int_E\hat K_s(e)\ti{\m}(\mathrm ds,\mathrm de)\\
 \ns\ds =-\( I_s^{y}  \hat Y_s+I_s^{z} \hat Z_s  +I_s^{k} \int_E \big(\psi(s,K _s^1(e) , e)-\psi(s,K _s^2(e) , e) \big)\l(\mathrm{d}e)   +\hat f_s
 \)\mathrm ds+\hat Z_s\mathrm dB_s+\int_E\hat K_s(e)\ti{\m}(\mathrm ds,\mathrm de)\\
 \ns\ds =-\( I_s^{y}  \hat Y_s+I_s^{z} \hat Z_s  +I_s^{k} \int_E I_s^\psi(e)\hat K _s(e)   \l(\mathrm{d}e)   +\hat f_s
 \)\mathrm ds+\hat Z_s\mathrm dB_s+\int_E\hat K_s(e)\ti{\m}(\mathrm ds,\mathrm de),\q s\ges t,
 \ea$$
 where, for $s\ges t$, $e\in E$,
$$
\begin{aligned}
& I_s^{y} =\left\{
\begin{aligned}&\!\! \frac {f^1( s,X _s,Y _s^1,Z _s^1,\int_E \psi(s,K _s^1(e) , e)\l(\mathrm{d}e))-f^1 (s,X _s,Y _s^2,Z _s^1,\int_E \psi(s,K _s^1(e) , e)\l(\mathrm{d}e))}{ Y _s^1-Y _s^2 }, \q  \mbox{if }\hat Y _s  \neq 0,\\
&\!\!0, \hskip 12.8 cm \mbox{otherwise};
\end{aligned}\right.\\
& I_s^{z}  =\left\{
\begin{aligned}&\!\!\frac {f^1(s,X _s,Y _s^2,Z _s^1,\int_E \psi(s,K _s^1(e) , e)\l(\mathrm{d}e))-f^1(s,X _s,Y _s^2,Z _s^2,\int_E \psi(s,K _s^1(e) , e)\l(\mathrm{d}e))}{|Z _s^1-Z _s^2|^2 }(Z _s^1-Z _s^2)^\ast, \\
  &\hskip 12.9cm\mbox{if }\hat Z_s \neq 0,\\
&\!\!0, \hskip 12.7cm \mbox{otherwise};
\end{aligned}\right.
\end{aligned}
$$
$$ \begin{aligned}
& I_s^{k}  =\left\{
\begin{aligned}&\!\!\frac {f^1(s,X _{s-},Y _{s-}^2,Z _s^2,\int_E \psi(s,K _s^1(e) , e)\l(\mathrm{d}e))-f^1(s,X _{s-},Y _{s-}^2,Z _s^2,\int_E \psi(s,K _s^2(e) , e)\l(\mathrm{d}e))}{\int_E\big(\psi(s,K _s^1(e) , e)-\psi(s,K _s^2(e) , e)\big)\l(\mathrm de)  }, \\
  &\hskip 7cm   \mbox{if }\int_E\big(\psi(s,K _s^1(e) , e)-\psi(s,K _s^2(e) , e)\big)\l(\mathrm de)  \neq 0,\\
&\!\!0, \hskip6.8 cm \mbox{otherwise};\end{aligned}\right.
%
\end{aligned}$$
and
$$
\begin{aligned}
&\hat f_s:= f^1(X _s,  Y_s^2,   Z_s^2,\int_E   \psi(s,K _s^2(e) , e)\l(\mathrm{d}e))-f^2(X _s,  Y_s^2,   Z_s^2,\int_E   \psi(s,K _s^2(e) , e)\l(\mathrm{d}e)),\\
& I_s^{\psi}(e)  =\left\{
\begin{aligned}&\!\!\frac { \psi(s,K _s^1(e) , e) - \psi(s,K _s^2(e) , e)}{  K _s^1(e)-K _s^2(e)  },  \hskip1 cm  \mbox{if } \hat K_s (e) \neq 0,\\
&\!\!0, \hskip5.5 cm \mbox{otherwise}.\end{aligned}\right.
\end{aligned}$$
Note that, by {\bf (B1)},  {\bf (B2)} and {\bf (S)},  for any $s\ges t$, $e\in E$,
$
I_s^{y} \les -\eta_{f^1},$ $ |I_s^{z} |\les L_z ,$ $| I_s^{k} | \les L_k
$, $0\les  I_s^{\psi}(e)   \les  \varrho(1\wedge|e|)
$ and  $\hat f_s \les 0$,  $\dbP$-a.s. Moreover, $I^k$ and $I^\psi(\cdot)$ are predictable. Hence, for 
%
$$\begin{aligned}
& M_s^{t,W}[I ^{z} ] :=
 \exp\Big\{\int_t^s I_r^{z}\mathrm dB_r-\frac12\int_t^s| I_r^{z}  |^2\mathrm dr  \Big\}, \\
& M_s^{t,\mu}[I ^{k}, I^\psi ]:= \prod\limits_{t<r\les s}\Big(1+\int_EI_r^{k}I_r^\psi(e)  \mu(\{r\},\mathrm de) \Big)\exp\Big\{-\int_t^s\int_EI_r^{k} I_r^\psi(e) \lambda(\mathrm{d}e)\mathrm dr\Big\},\ 0\les t\les s\les  T,
\end{aligned}$$
 we have that 
$M_s^{t}:= M_s^{t,W}[I  ^{z} ]\cd  M_s^{t,\m}[I ^{k}, I^\psi ]$, $s\in[t,T]$,   satisfies
$$\left\{\ba{ll}
 \ds\!\!\! dM_r^{t}= I_r^{z}  M_r^{t}\mathrm dB_r +  \int_E I_r^{k} I_r^\psi(e) M_{r-}^{t}\ti{\mu}(\mathrm dr,\mathrm de),\q r\in[t,T], \\
\ns\ds\!\!\! M_t^{t}=1.
\ea\right.$$
%
By defining  the new probability measure  $  \dbP^M_T$ by $\frac{\mathrm d  \dbP^M}{\mathrm d\dbP}\big|_{\cF_T}= M_T^{t}$, $T>t$, we know that, with respect to $\dbP_T^M$,  $\ds   B_r^M:=B_r-\int_t^rI_r^{z} \mathrm dr$, $r\in[t,T]$, is  a Brownian motion, and $\tilde\mu^M(\mathrm dr,\mathrm de):=\tilde{\mu}(\mathrm dr,\mathrm de)-I_r^{k} I_r^\psi(e) \lambda(\mathrm{d}e)\mathrm dr$ is a  Poisson martingale measure with the new compensator $(1+I_r^{k} I_r^\psi(e) )\lambda (\mathrm{d}e)\mathrm dr$. We denote by  $\{\cF_s^M\}_{s\ges t}$  the new filtration generated by $\{B^M \}$ and $\{ \tilde\mu^M\}$.
Then, we have
$$
   d\hat Y_s=-\big(I_s^{y} \hat Y_s +\hat f_s
 \big)\mathrm ds+\hat Z_s\mathrm dB_s^M +\int_E\hat K_s(e)\ti{\m}^M(\mathrm ds,\mathrm de),\q s\ges t.
 $$

Setting $\ds \cE_r:=\exp\(  \int_t^rI_\t^{y} d\t\)\les \exp\big\{-\eta_{f^1}(r-t)\big\}$, $r\ges t$,  and applying It\^o's formula to $\cE_r\hat Y_r$, we get, $\dbP$-a.s.
$$
 \ds \hat Y_t= \cE_T\hat Y_T+ \int_t^T \cE_r \hat f_r dr-\int_s^T \cE_r\hat Z_r\mathrm dB_r^M- \int_t^T\int_E \cE_r\hat K_r(e)\ti{\m}^M(\mathrm dr,\mathrm de),$$
and, thus, by Bayes' formula,
$$\ba{ll}
 \ds \hat Y_t  = \dbE_{\dbP_T^M}^{\cF_t} \[\cE_T\hat Y_T+ \int_t^T \cE_r  \hat f_r dr\]\les \dbE ^{\cF_t}[M_T^{t}\cE_T\hat Y_T]\les (\dbE ^{\cF_t}[ |M_T^{t}|^2])^\frac12 (\dbE^{\cF_t}[|\cE_T\hat Y_T|^2])^\frac12  \\
\ns\ds  \les  \exp\Big\{\(-\eta_{f^1}+\frac{L_z ^2+L_k^2|\rho(\cd)|_{\l,2}^2}{2} \)(T-t)\Big\}  (\dbE^{\cF_t} [|\hat Y_T|^2])^\frac12 \to 0, \mbox{ as } T\to \i,
\ea$$
where we have used $    \dbE^{\cF_t}[|M_T^{t}|^2]  \les  \exp\big\{ (L_z ^2+L_k^2|\rho(\cd)|_{\l,2}^2)  (T-t) \big\}
 $, $\bar \eta_{f^1}=\eta_{f^1}-\frac{L_z ^2+L_k^2|\rho(\cd)|_{\l,2}^2}{2} >0$, $\hat f_r\les 0$, $r\ges t$ and    Lemma \ref{Yto0}.
Therefore, $Y^1_t\les Y^2_t $, $\dbP$-a.s.

\end{proof}

 In the end of this section, let us still emphasize  that the  way  of the driver $f$ depends on $k$ is more  general than in the most of the literature, and this thanks to    the function  $\psi$. We have not only established the wellposedness of BSDE with jumps \eqref{BSDE}, but also its $L^p$-estimates $(p\ges 2)$ and a comparison theorem.

\section{Stochastic control problem for  infinite horizon forward-backward stochastic differential equations with jumps}

We begin with  the formulation of   the control problem. From now on,  let $m=1$.
Moreover, let $U\subseteq\dbR^l$ be nonempty and  compact, and  the mappings $b: \dbR^n\times U\to\dbR^n$,  $\si: \dbR^n\times U\to\dbR^{n\times d}$, $\g:  E \times\dbR^n\times U\to\dbR^n$, $f:\dbR^n\times\dbR\times\dbR^d\times\dbR\times U\to \dbR$ and $\rho:E\to\dbR$ satisfy   the following assumptions for some $p\ges 2$.

\ss

 {\bf (C1)$_p$}    There exist    nonnegative constants $\ell_b, \ell_\si, \ell_{1}$, a mapping $\ell_\g(\cd)\in L_\l^2(E;[0,1])\cap L_\l^p(E;[0,1])$ and $\a_b>0$ such that, for all $x$, $x'\in\dbR^n$, $e\in E$, $u\in U$,
 $$\ba{ll}
 \ds |b(x,u)-b(x',u)|\les\ell_b|x-x'|, \ |\si(x,u)-\si(x',u)|\les \ell_\si|x-x'|, \\
   \ns\ds    {|\g(e,x,u)-\g(e,x',u)|   \les \ell_1\ell_{\g}(e) |x-x'|},\\
 \ns\ds \lan b(x,u)-b(x',u),x-x'\ran \les -\a_b|x-x'|^2,\\
 \ns\ds \a_{b,p}:= 2\a_b-(p-1)\ell_\si  ^2-\frac{2c_p}p \ell_{\g,2} ^2-  c_p  \ell_{\g,p} ^p>0,\ea$$
 where $\ds\ell_{\g,p}:=\ell_1\(\int_E|\ell_\g(e)|^p\l(\mathrm de)\)^\frac 1p$ and  $c_p$ is defined in \eqref{cp}.

\smallskip

  {\bf (C2)}    There exist  nonnegative constants  $\ell_x, \ell_y,\ell_z,\ell_k$ and  $\a_f>0$ such that,   for all $x,x'\in\dbR^n$, $(y,z,k)$, $(y',z',k')\in\dbR\times\dbR^d\times  \dbR $,  $u\in U$,
   $$\ba{ll}
   \ds|f(x,y,z,k,u)-f(x',y',z',k',u)| \les \ell_x|x-x'|+\ell_y|y-y'| + \ell_z|z-z'|+\ell_k|k-k'|,\\
   \ns\ds\lan f(x,y,z,k,u)-f(x,y',z,k,u),y-y'\ran \les -\a_f|y-y'|^2,\\
   \ns\ds \bar\a_f:= \a_f -\frac{\ell_{z}^2+\ell_{k}^2|\rho(\cd)|_{\l,2}^2}{2}>0.\ea$$

  {\bf (C3)}   There exists a constant $\varrho>0$ such that $0\les \rho(e)\les \varrho(1\wedge|e|)$, $e\in E$.

 {\bf (C4)} For all $(x,y,z,u)\in   \dbR^n\times\dbR\times \dbR^d\times U$, $k,k'\in \dbR$,  $k\neq k'$,
$$\frac{f(x,y,z,k,u)-f(x,y,z,k',u)}{k-k'}> -\frac 1\varrho , \mbox{ where  }\varrho \mbox{ has been introduced   in      {\bf (C3)} }. $$
 Further, for $p\ges 2$ and $t\ges 0$,  we introduce the following set of  admissible  controls,
$$   {\cU^p_{t,\i}} :=\Big\{u\!\mid\! u:\Omega\times [t,\i)\to U  \mbox{ is } \dbF\mbox{-predictable with } \esssup\limits_{t\ges0}\big(\Pi^{1,u,p}_t  + \Pi^{2,u,p}_t \big)<\i, \ \dbP\mbox{-a.s.}\Big\},
$$
where
$$\ba{ll}
 \ds \Pi^{1,u,p}_t:=\dbE^{\cF_t}\[\int_t^\i\big( |b(0,u_s)|^p+|\si(0,u_s)|^p+ |\g(\cd,0,u_s)|_{\l,p}^p   \big)\mathrm ds \\
\ns\ds\qq\qq\qq\qq +  \(\int_t^\i\big( |b(0,u_s)|^2+|\si(0,u_s)|^2+ |\g(\cd,0,u_s)|_{\l,2}^2   \big)\mathrm ds\)^\frac p2\],\\
\ns\ds  \Pi^{2,u,p}_t:=\dbE^{\cF_t}\[ \(\int_t^\i  |f(0,0,0,0,u_s)|^2   \mathrm ds\)^\frac p2\].
\ea $$

 For any  initial state $x\in  \dbR^n$ and $u(\cd)\in \cU^p_{0,\i}$, we consider the following infinite horizon SDE    with jumps,
 \bel{state}
 \left\{\ba{ll}
 \ds\!\!\!\! \mathrm dX^{ x;u}_s=b(X^{x;u}_s,u_s)\mathrm ds+\si(X^{x;u}_s,u_s)  \mathrm dB_s+\int_E\g(e,X^{x;u}_{s-},u_{s})\ti{\m}(\mathrm ds,\mathrm de),\ s\ges 0,\\
 \ns\ds\!\!\!\! X^{ x;u}_0=x,
 \ea\right.\ee
and  the following  BSDE    with jumps
 \bel{BSDE-cost}\ba{ll}
   \ds  Y^{x;u}_s= Y^{x;u}_T+ \int_s^Tf\big(X^{x;u}_r,Y^{x;u}_r,Z^{x;u}_r,\int_E K^{x;u}_r(e)\rho( e)\l(\mathrm{d}e),u_r\big)\mathrm dr  -\int_s^T Z^{x;u}_r\mathrm dB_r\\
   \ns\ds \qq\q  -\int_s^T\int_EK^{x;u}_r(e)\ti{\m}(\mathrm dr,\mathrm de),\q    0\les s\les T <\i.
 \ea\ee
 To emphasize better the main arguments of studying the control problems,  for $(r,k,e)\in[0,\i)\times\dbR\times E$,  we specify  $\psi(r,k,e)$ in the   BSDE \eqref{BSDE}  to be  $\psi(r,k,e)=k\rho(e)$,   just like \eqref{BSDE-cost}.

  Based on  our studies in Section 2,  we have the following results concerning SDE \eqref{state} and BSDE \eqref{BSDE-cost}.

\bl\label{Lemma-SDE-1}\sl Let $p\ges 2$ and  assume {\bf (C1)$_p$}.
For all $u(\cd)\in\cU^p_{0,\i}$,  SDE  \eqref{state} admits a unique solution $X^{ x;u}\in L_\dbF^p(0,\i;\dbR^n)\cap  L_\mathbb{F}^p(\Omega;L^2(0,\infty;\mathbb{R}^n))$.  Further, we have the existence of nonnegative constants $C_p$ independent of $u(\cd)$ such that,  for all $x,x'\in\dbR^n$ and $t\ges 0$,
\bel{SDE-(i)}\ba{ll}
 \ds {\rm(i)} \ \lim\limits_{s\to\i} \dbE^{\cF_t} [  |X_s^{ x;u} |^p]=0,\ \dbP\mbox{-a.s.}\\

\ns\ds  {\rm(ii)}\   \dbE ^{\cF_t}\[\sup\limits_{s\in[t,\i)}  |X_s^{ x;u} |^p+\int_t^\i   |X_r^{ x;u}|^p\mathrm dr +\(\int_t^\i   |X_r^{ x;u}|^2\mathrm dr \)^\frac p2\]  \les C_p\big(\Pi^{1,u,p}_t+|X_t^{ x;u} |^p \big), \ \dbP\mbox{-a.s.}\\
\ns\ds {\rm(iii)}\    \dbE ^{\cF_t}\[\sup\limits_{s\in[t,\i)}|X_s^{ x;u} - X_s^{ x';u} |^p +   \int_t^\i  |X_r^{ x;u} -  X_r^{ x';u} |^p\mathrm dr+\(\int_t^\i   |X_r^{ x;u} -  X_r^{ x';u} |^2\mathrm dr \)^\frac p2\] \\
\ns\ds \qq \les C_p  |X_t^{ x;u}-X_t^{ x';u}|^p, \ \dbP\mbox{-a.s.}
\ea\ee
%
 \el


 %


\bl\label{Le-Yp} \sl Assume  the conditions of Lemma \ref{Lemma-SDE-1},   {\bf (C2)},  {\bf (C3)} hold true. For all  $u(\cd)\in\cU^p_{0,\i}$,
  the infinite horizon BSDE \eqref{BSDE-cost} admits a unique solution $(Y^{x;u},Z^{x;u},K^{x;u} )\in \sS^p_\dbF(0,\i)$. Moreover, there exists some nonnegative constant  $C_p$ independent of $u(\cd)$ such that, for  all $x,x'\in\dbR^n$ and $t\ges 0$, $\dbP$-a.s.,
$$\ba{ll}
\ds{\rm (i)}\ \lim\limits_{s\to\i}\dbE^{\cF_t} [  |Y_s ^{x;u}|^2]   =0,\\
\ns\ds {\rm (ii)}\  \dbE^{\cF_t} \[ \sup\limits_{s\in[t,\i)} |Y_s^{x;u} |^p   + \(  \int_t^\i\big(|Y_r^{x;u}|^2 +  |Z _r^{x;u}|^2+    |K _r^{x;u}(\cd)|_{\l,2}^2 \big) \mathrm dr \)^\frac p2 \] \les  C_p(\Pi^{1,u,p}_t+\Pi^{2,u,p}_t+|X_t^{x;u}|^p),\\
\ns\ds {\rm (iii)}\   \dbE^{\cF_t} \[\sup\limits_{s\in[t,\i)}| Y_s^{x;u}-  Y_s^{x';u}  |^p  \!+ \! \(\!\int_t^\i \!\!\big( | Y_r^{x;u}- Y_r^{x';u} |^2\!+  \! |Z_r^{x;u}-  Z_r^{x';u} |^2\!+\! |K_r^{x;u}(\cd)-  K_r^{x';u}(\cd)  |^2_ \l
 \big) \mathrm dr\! \)^\frac p2 \]\!\!\!\!\!\!\!\\
\ns\ds\qq  \les C_p |X_t^{x;u}-X_t^{x';u}|^p.
\ea$$

\el

Based on the above preparations, we   define  the recursive   cost functional with the help of our     BSDE with jumps  \eqref{BSDE-cost} as follows: for all $u(\cd)\in\cU^p_{0,\i}$,
\bel{cost}
J(x;u(\cd)):=Y_0^{x;u},\q  x \in \dbR^n.
\ee
Then, we can  formulate the following  optimal control problem for given $p\ges 2$.

\ss

 {\bf Problem (OC)$_p$} {\sl For any $x\in\dbR^n$, find $\bar u(\cd)\in\cU^p_{0,\i}$, such that
\bel{bar u} J(x;\bar u(\cdot))=V_p(x):=\sup_{u(\cd)\in\cU^p_{0,\i} }J(x;u(\cdot)).\ee
A control $\bar u(\cd)$ satisfying  \eqref{bar u}  is said to be  optimal for Problem (OC)$_p$, and  by $\bar X(\cd)= X^{x;\bar u}(\cd)$ we denote
  the corresponding optimal state process.
The above defined mapping
 $V_p:\dbR^n\to\dbR$     is called   the    value  function of Problem (OC)$_p$.}

\br{}\sl
Note that   Problem (OC)$_p$ will be studied for any given  $p\ges 2$. The different $p$ effect   the conditions imposed on $b,\si,\g$ and $ f $ through   hypothesis  {\bf(C1)$_p$} and $\cU_{0,\i}^p$, but also the spaces of the  state processes $X^{x;u}$ and the  cost functionals $Y^{x;u}$.
The reason of doing this, rather than choosing $p=2$ as usual, is that the   value function $V_p(\cd)$ may have different properties under different $p$ (or {\bf(C1)$_p$} and $\cU_{0,\i}^p$), which concerns in particular  the semi-convexity of $V_p(\cd)$. The phenomenon will  be  clearly explained by the  proof of Proposition \ref{Le-Semi}.

\er

Before the study of Problem (OC)$_p$, we shall study  some properties of the value function $V_p(\cd)$ defined by \eqref{bar u}.
We have the following first results directly from the Lemmas \ref{Lemma-SDE-1} and  \ref{Le-Yp}.

\bl\label{W-Lip}\sl For any given $p\ges 2$,    let {\bf(C1)$_p$}-{\bf(C3)} hold. Then there exists a constant $C>0$ such that,
for   all $x,x'\in\dbR^n$,
$$ |V_p(x)|\les C(1+|x|),\q  |V_p(x)-V_p(x') | \les C| x-x' |.
 $$

\el

 Next, to prepare  the subsequent research for the   stochastic verification theorem of Problem (OC)$_p$   in Section 5, we  study the semi-convexity of the value function $V_p(\cd)$. For this, we assume  the following    additional conditions on the coefficients $b,\si,\g$ and $f$ of \eqref{state} and \eqref{BSDE-cost}.
\ms

  {\bf (D1)} (i) For all $u\in U$,  $b(x,u)$, $\si(x,u)$   are differentiable in $x\in\dbR^n$, and the corresponding  first-order partial derivatives in  $x\in\dbR^n$ are continuous in $ u \in    U$ and  Lipschitz continuous in $x$, uniformly with respect to $u\in  U$.

   (ii) For all $u\in U$, $e\in E$,   $\g(e,x,u)$ are differentiable in $x\in\dbR^n$, and the corresponding  first-order partial derivatives in  $x\in\dbR^n$ are continuous in $ u \in U$, and there exists  a mapping $  \bar\ell_{\g}(\cd)\in L_\l^2(E;
  \dbR^+)  $ such that, for all $x,x'\in\dbR^n$, $u\in U$,  $e\in E$,
    {  $|\g_x(e,x,u)-\g_x(e,x',u)|\les  \bar\ell_{\g}(e) |x-x'|.$}

\ss
 {\bf (D2)}   $f(x,y,z,k,u)$ is semi-convex in $(x,y,z,k)\in\dbR^n\times \dbR\times\dbR^d\times \dbR,$ uniformly with respect to $u\in U$.

\bp\label{Le-Semi}\sl Let $p>4$ be given. If the conditions  {\bf (C1)$_p$}-{\bf (C4)},  {\bf (D1)}, {\bf (D2)} and
\bel{b4}\a_{b,4}>0,\q \ell_\g(\cd)\in L^4_\l(E;[0,1]),\q \dbE\[\int_0^\i\big( |b(0,u_s)|^4+|\si(0,u_s)|^4+ |\g(\cd,0,u_s)|_{\l,4}^4   \big)\mathrm ds\]<\i,\ee
 hold true, then  $V_p(\cd)$ is semi-convex, i.e., $V_p(\cd) +\k|\cd|^2$ is convex  for some $\k\ges 0$.

\ep
%
%
%
%
%
%
%

We  first   study some auxiliary results before proving  Proposition \ref{Le-Semi}. For convenience, for   $x_1$, $x_2\in\dbR^n$, and  $\d\in [0,1],$  we set $x_\d :=\d x_1+( 1-\d )x_2$ and
 \bel{Not-XYZK}  \widetilde{\varphi}^u := \d \varphi^{x_1;u} +( 1-\d )\varphi^{x_2;u } ,\q \h \varphi^u := \varphi^{x_1;u} -\varphi^{x_2;u },  \mbox{ for } \f=X,Y,Z,K,\mbox{ respectively.}\ee

 \bl\label{Le-X-4} \sl  Under the assumptions of Proposition \ref{Le-Semi},   there exists a constant $C>0$ such that,  $\dbP$-a.s., for all $t\ges 0$ and $u(\cd)\in \cU_{0,\i}^p$,
\bel{SDE-4}   \dbE ^{\cF_t}\[\sup\limits_{s\in[t,\i)}|\widetilde{X} _s^u -X^{x_\d;u}_s |^2+  \int_t^\i| \widetilde{X} _r^u -X^{x_\d;u}_r |^2\mathrm dr \] \les  C\(|\widetilde{X} _t^u-X^{x_\d;u}_t  |^2+   \d^2( 1-\d )^2 |\h X^u _t |^4\).
   \ee
 \el

\begin{proof}
It is necessary to point out that the assumptions in Proposition \ref{Le-Semi}  imply  {\bf (C1)$_4$},   $\a_{b,2}  >0$  and $u(\cd)\in \cU_{0,\i}^4 $, which  will be needed in our proof. In the following, the constant  $C$ can    differ  from line to line.

First we remark that,    for all $s\ges 0$,
 $$\ba{ll}
  \ds \widetilde{X} _s^u -X^{x_\d;u}_s =\d X^{x_1;u }_s +(1-\d)X^{x_2;u }_s -X^{x_\d;u}_s \\
 \ns\ds=\int_0^s\big(J_b(r)+ b(\widetilde{X} _r^u ,u_r  )- b (X^{x_\d;u}_r ,u_r  ) \big)\mathrm dr+\int_0^s\big(J_\si(r)+ \si(\widetilde{X} _r^u ,u_r  )- \si(X^{x_\d;u}_r ,u_r  )\big)\mathrm dB_r\\
 \ns\ds\q+\int_0^s\int_E\big(J_\g(r,e)+ \g(e,\widetilde{X} _{r-} ,u_{r}  )-\g(e,X^{x_\d;u}_{r-} ,u_{r}  )\big)\ti\m(\mathrm dr,\mathrm d e),
 \ea $$
 where $$\ba{ll}
   \ds J_g(r)  :=   \d g\big(X^{x_1;u }_r ,u_r  \big)+(1-\d)g\big(X^{x_2;u}_r ,u_r  \big) - g(\widetilde{X} _r^u ,u_r  \big) ,\q g=b,\si, \mbox{ respectively,}\\
   \ns\ds J_\g(r,e )  :=   \d \g\big(e,X^{x_1;u }_{r-} ,u_{r}  \big)+(1-\d)\g\big(e,X^{x_2;u}_{r-} ,u_{r}  \big) - \g(e, \widetilde{X}^u _{r-} ,u_{r}  \big),\q r\ges 0, \ e\in E.\ea$$
 Note that, from {\bf (D1)}, for all $r\ges  0$,
 $$\ba{ll}
 \ds  |J_g(r)| :=  |\d g\big(X^{x_1;u }_r ,u_r  \big)+(1-\d)g\big(X^{x_2;u}_r ,u_r  \big) - g(\widetilde{X} _r^u ,u_r  \big)  | \\
  %
  \ns\ds\qq\q\les \d(1-\d)|  \h X^u _r |\cd \int_0^1\big|g_x(\widetilde{X} _r^u +\th( 1-\d )   \h X^u _r,u_r \big)-g_x(\widetilde{X} _r^u -\th\d  \h X^u _r ,u_r  \big)\big|\mathrm d\th   \\
%
  \ns\ds\qq\q\les C\d(1-\d)| \h X^u _r |^2, \q  \dbP\mbox{-a.s., } \q  g=b,\si, \mbox{ respectively,}
 \ea $$
 and similarly,
$\ds
 {|J_\g(r,e) |    \les   C   \bar\ell_\g(e) \d(1-\d)| \h X^u _r |^2 }$,   $e\in E$, $\dbP$-a.s., where  $C>0$  is a constant.

 By applying It\^o's formula to $ |\widetilde{X}_s   -X^{x_\d;u}_s |^2$,   for all $s\ges t\ges 0$ and $\e>0$, we have
  $$\ba{ll}
\ds  \dbE^{\cF_t}\[|\widetilde{X} _s^u -X^{x_\d;u}_s |^2 \] =  |\widetilde{X} _t^u-X^{x_\d;u}_t  |^2  + \dbE^{\cF_t}\[ \int_t^s\(2\lan \widetilde{X} _r^u -X^{x_\d;u}_r ,J_b(r)+ b(\widetilde{X} _r^u ,u_r  )- b (X^{x_\d;u}_r ,u_r  ) \ran  \\
  \ns\ds\qq\qq\qq\q +  |J_\si(r)+ \si(\widetilde{X} _r^u ,u_r  )- \si(X^{x_\d;u}_r ,u_r  )|^2+|J_\g(r,\cd) + \g(\cd,\widetilde{X} _r^u ,u_r  )-\g(\cd,X^{x_\d;u}_r ,u_r  )|_{\l,2}^2\)  \mathrm dr \] \\
 %
 \ns\ds \les  |\widetilde{X} _t^u-X^{x_\d;u}_t  |^2  + \dbE^{\cF_t}\[   \int_t^s\((- \a_{b,2 }+3 \e)| \widetilde{X} _r^u -X^{x_\d;u}_r |^2  +\frac1\e  |J_b(r)     |^2  +C_\e \big(| J_\si(r)   |^2+| J_\g (r,\cd)  |_{\l,2}^2 \big)  \) \mathrm dr\]\\
  \ns\ds \les  |\widetilde{X} _t^u-X^{x_\d;u}_t  |^2 - (  \a_{b,2 }-3 \e) \dbE^{\cF_t}\[   \int_t^s | \widetilde{X} _r^u -X^{x_\d;u}_r |^2\mathrm dr\]  + {C}_\e \d^2(1-\d)^2\dbE^{\cF_t}\[  \int_t^s | \h X^u _r |^4  \mathrm dr\]  ,\ \dbP\mbox{-a.s.}
 \ea $$
  %
Then, by choosing  $\e$ small enough and letting $s\to\i $,  from   \eqref{SDE-(i)}-(iii),  we get
\bel{SDE-4-int}   \dbE ^{\cF_t}\[   \int_t^\i| \widetilde{X} _r^u -X^{x_\d;u}_r |^2\mathrm dr \] \les  C_\e\(|\widetilde{X} _t^u-X^{x_\d;u}_t  |^2+   \d^2( 1-\d )^2 |\h X^u _t |^4\), \  \dbP\mbox{-a.s.}
   \ee
We remark that  condition \eqref{b4}  has been used  above.

Further, by applying   It\^o's formula to $ |\widetilde{X}_s^u  -X_s^{x_\d;u}   |^2$ again, for all $T>t\ges 0$, we get
$$\ba{ll}
\ds    {\dbE^{\cF_t}\[\sup\limits_{s\in[t ,T]}|\widetilde{X} _s^u  -X^{x_\d;u}_s |^2\]}\les \dbE^{\cF_t}\[ \int_t^T2| \widetilde{X} _r^u -X^{x_\d;u}_r |\cd|J_b(r)+ b(\widetilde{X} _r^u ,u_r  )- b (X^{x_\d;u}_r ,u_r  )| \mathrm dr\] \\
  \ns\ds\q   +\dbE^{\cF_t}\[\int_t^T\(|J_\si(r)+ \si(\widetilde{X} _r^u ,u_r  )- \si(X^{x_\d;u}_r ,u_r  )|^2+|J_\g(r,\cd) + \g(\cd,\widetilde{X} _r^u ,u_r  )-\g(\cd,X^{x_\d;u}_r ,u_r  )|_{\l,2}^2\)  \mathrm dr \] \\
   \ns\ds\q+ C\dbE^{\cF_t}\[\( \int_t^T| \widetilde{X} _r^u -X^{x_\d;u}_r |^2\cd|J_\si(r)+ \si(\widetilde{X} _r^u ,u_r  )- \si (X^{x_\d;u}_r ,u_r  )|^2 \mathrm dr\)^\frac12\]\\
   \ns\ds\q +C\dbE^{\cF_t}\[\Big|\int_t^T\int_E| \widetilde{X} _{r-}^u -X^{x_\d;u}_{r-} |^2\cd |J_\g(r,e)+ \g(e,\widetilde{X} _{r-}^u ,u_r  )- \g (e,X^{x_\d;u}_{r-} ,u_r  )|^2 {\m(\mathrm dr, \mathrm de)}\Big|^\frac12\] \\
 \ns\ds\les  C\dbE^{\cF_t}\[ \int_t^T\(  \d(1-\d) | \widetilde{X} _r^u -X^{x_\d;u}_r |\cd| \h X^u _r |^2+| \widetilde{X} _r^u -X^{x_\d;u}_r |^2 +  \d^2(1-\d)^2    | \h X^u _r |^4 \)\mathrm dr\]\\
   \ns\ds\q+ C\dbE^{\cF_t}\[\( \sup\limits_{r\in[t,T]}| \widetilde{X} _r^u -X^{x_\d;u}_r |^2 \int_t^T|J_\si(r)+ \si(\widetilde{X} _r^u ,u_r  )- \si (X^{x_\d;u}_r ,u_r  )|^2 \mathrm dr\)^\frac12\]\\
   \ns\ds\q + C\dbE^{\cF_t}\[\Big|\sup\limits_{r\in[t,T]}| \widetilde{X} _r^u -X^{x_\d;u}_r |^2  \int_t^T\int_E |J_\g(r,e)+ \g(e,\widetilde{X} _{r-}^u ,u_{r} )- \g (e,X^{x_\d;u}_{r-} ,u_{r}  )|^2 \m(\mathrm dr,\mathrm de)\Big|^\frac12\] .
    \ea$$
    As, obviously,
    $$\ba{ll}
\ds  \dbE^{\cF_t}\[\Big|\sup\limits_{r\in[t,T]}| \widetilde{X} _r^u -X^{x_\d;u}_r |^2  \int_t^T\int_E |J_\g(r,e)+ \g(e,\widetilde{X} _{r-}^u ,u_{r}  )- \g (e,X^{x_\d;u}_{r-} ,u_{r}  )|^2 \m(\mathrm dr,\mathrm de)\Big|^\frac12\]\\
 \ns  \ds \les  C\dbE^{\cF_t}\[\(\sup\limits_{r\in[t,T]}| \widetilde{X} _r^u -X^{x_\d;u}_r |^2  \int_t^T\int_E \( |\bar\ell_\g(e)|^2   \d^2(1-\d)^2| \h X^u _{r-} |^4  +\ell_\g(e)^2 | \widetilde{X} _{r-}^u  -  X^{x_\d;u}_{r-}  |^2 \)\m(\mathrm dr,\mathrm de)\)^\frac12\]\\

 \ns  \ds \les \frac 14  \dbE^{\cF_t}\[\sup\limits_{r\in[t,T]}| \widetilde{X} _r^u -X^{x_\d;u}_r |^2\]+  C\dbE^{\cF_t}\[\int_t^T\int_E \( |\bar\ell_\g(e)|^2   \d^2(1-\d)^2| \h X^u _r |^4  +\ell_\g(e)^2 | \widetilde{X} _r^u  -  X^{x_\d;u}_r  |^2 \)\l(\mathrm de)\mathrm dr\],
    \ea$$
it follows from the above estimate that
    $$\ba{ll}
 \ds\dbE^{\cF_t}\[\sup\limits_{s\in[t ,T]}|\widetilde{X} _s^u  -X^{x_\d;u}_s |^2\]\\
   \ns\ds \les C  \dbE^{\cF_t}\[ \int_t^T| \widetilde{X} _r^u -X^{x_\d;u}_r |^2 \mathrm dr\] + C^2\d^2(1-\d)^2\dbE ^{\cF_t}\[\int_t^T   | \h X^u _r |^4\mathrm dr\]+\frac12 \dbE^{\cF_t}\[\sup\limits_{s\in[t,T]}|\widetilde{X} _s^u -X^{x_\d;u}_s |^2\].
  \ea$$
 Thus, considering \eqref{SDE-4-int} and \eqref{SDE-(i)}-(iii), we obtain
 $$
   \ds\dbE^{\cF_t}\[\sup\limits_{s\in[t ,T]}|\widetilde{X} _s^u  -X^{x_\d;u}_s |^2\]
 \les   C\(|\widetilde{X} _t^u-X^{x_\d;u}_t  |^2+   \d^2( 1-\d )^2 |\h X^u _t |^4\)  ,\q  \dbP\mbox{-a.s.}
  $$
Letting $T\to\i$, \eqref{SDE-4} is derived by  combining the above estimate with \eqref{SDE-4-int}.

 \end{proof}

With the notations introduced in \eqref{Not-XYZK}, we have
\bel{ti-Y}\ba{ll}
  \ds  \widetilde{Y}^u _s  = \widetilde{Y}^u _T+\int_s^T   \(\d f\big(X^{x_1;u }_r,Y^{x_1;u }_r,Z^{x_1;u }_r ,\int_EK^{x_1;u }_r(e)\rho(e)\l(\mathrm de),  u_r  \big)\\
 \ns\ds\hskip3cm +(1-\d)f\big(X^{x_2;u }_r,Y^{x_2;u }_r,Z^{x_2;u }_r ,\int_EK^{x_2;u }_r(e)\rho(e)\l(\mathrm de),  u_r \big)\)\mathrm dr \\
 \ds\qq \q -\int_s^T \widetilde{Z}^u _r  \mathrm dB_r -\int_s^T \int_E \widetilde{K}^u _r(e)  \ti{\m}(\mathrm dr,\mathrm de),\q 0\les s\les T<\i.
   \ea \ee
   Then, we have the following result.
%

\bl\label{Le-y-cY}\sl Let us assume  the conditions of Proposition \ref{Le-Semi}. Then,  for all $s\ges 0$, $u(\cd)\in \cU_{0,\i}^p$,
$$\widetilde{Y}^u _s\ges \cY^u _s ,\q \dbP\mbox{-a.s.},$$
  where $\cY^u$ satisfies   the following infinite horizon BSDE with jumps,
\bel{hat-Y}\ba{ll}
 \ds  \cY^u _s  = \cY^u _T  +\!\int_s^T\!\!\[ f \big( X^{x_\d;u}_r,\cY^u_r -\cC_1(r)-K\d(1-\d)\cC_2(r)^2,\cZ^u_r ,\int_E \cK^u_r(e)\rho(e)\l(\mathrm de),   u_r  \big)\\
\ns\ds \qq\ -\ell_x|\widetilde X_r- X^{x_\d;u}_r|-\ell_y\big(\cC_1(r)+K\d(1-\d)\cC_2(r)^2\big)- K \d(1-\d) \big(| \h X^u _r |^2+|\h  Z ^{ u }_r |^2+|\h  K ^{ u }_r(\cd) |_{\l,2}^2\big)\]\mathrm dr\!\!\!\!  \\
\ns\ds\qq \  -\int_s^T\cZ^u _r  \mathrm dB_r-\int_s^T\int_E \cK^u _r (e) \ti\mu(\mathrm dr,\mathrm de),\q \mbox{ for all } 0\les s\les T<\i,
   \ea \ee
with   $\cC_1(r):= |\widetilde X_r-X_r^{x_\d;u}|$, $\cC_2(r):= | \h X^u _r |$, $r\ges 0$.

\el

 Before the proof, we remark that, for every  $u(\cd)\in \cU_{0,\i}^p$, BSDE \eqref{hat-Y} admits  under   the conditions of Proposition \ref{Le-Semi}  a unique $\dbF$-adapted solution.

\begin{proof}
Using {\bf (D2)}, the Lipschitz continuity of $f$ with respect to its variables as well as Lemma \ref{Le-Yp}-(iii), we have that there is a constant $K>0$ such that  for all $r\ges 0$,
$$\ba{ll}
 \ds \d f\big(X^{x_1;u }_r,Y^{x_1;u }_r,Z^{x_1;u }_r ,\int_EK^{x_1;u }_r(e)\rho(e)\l(\mathrm de),  u_r  \big)  +(1-\d)f\big(X^{x_2;u }_r,Y^{x_2;u }_r,Z^{x_2;u }_r ,\int_EK^{x_2;u }_r(e)\rho(e)\l(\mathrm de),  u_r \big)\\
\ns\ds \ges f \big(\widetilde X_r ,\widetilde Y_r ,\widetilde Z_r ,\int_E\widetilde K_r(e)\rho(e)\l(\mathrm de),   u_r  \big)-K\d(1-\d)\(|\h X^u _r |^2+|\h  Y ^{ u }_r |^2+|\h  Z ^{ u }_r |^2+| \rho(\cd)|_{\l,2}^2\cd|\h  K ^{ u }_r(\cd)|_{\l,2}^2\)\\
\ns\ds \ges f \big( X^{x_\d;u}_r,\widetilde Y_r -\cC_1(r)-K\d(1-\d)\cC_2(r)^2,\widetilde Z_r ,\int_E\widetilde K_r(e)\rho(e)\l(\mathrm de),   u_r  \big)-\ell_x|\widetilde X_r^u- X^{x_\d;u}_r|\\
\ns\ds \q -\ell_y\(\cC_1(r)+K\d(1-\d)\cC_2(r)^2\)-K \d(1-\d) \(| \h X^u _r |^2+|\h  Z ^{ u }_r |^2+|\h  K ^{ u }_r(\cd) |_{\l,2}^2\).
\ea$$
The stated result can be concluded    by using the comparison theorem (Theorem \ref{Th-Com}).

\end{proof}

From now on, the constant $K$ is that used  in BSDE \eqref{hat-Y}.
We introduce  $\cD_r:=\cC_1(r)+K\d(1-\d)\cC_2(r)^2 $ and $\sY^u_r:=\cY^u_r-\cD_r , $ $r\ges 0$. Then, from
  \eqref{hat-Y}  we get
\bel{equ-sY}\ba{ll}
 \ds  \sY^u_s =\sY^u_T+ \int_s^T\[f \big( X^{x_\d;u}_r,\sY^u_r , \cZ^u_r ,\int_E  \cK^u_r(e)\rho(e)\l(\mathrm de),   u_r  \big)-\ell_x|\widetilde X_r^u- X^{x_\d;u}_r|-\ell_y\cD_r \\
\ns\ds \qq\qq\qq\qq  -K \d(1-\d)\big(|\h X^u _r |^2+ |\h  Z ^{ u }_r |^2+|\h  K ^{ u }_r(\cd) |_{\l,2}^2\big) \] \mathrm dr \\
 \ds\qq -\int_s^T \cZ^u_r \mathrm dB_r-\int_s^T\int_E\cK^u _r (e) \ti\mu(\mathrm dr,\mathrm de),\q \mbox{ for all } 0\les s\les T<\i.
 \ea \ee
%
%
 Note that, under the conditions of Proposition \ref{Le-Semi}, from \eqref{SDE-(i)}-(iii) and Lemma \ref{Le-X-4},  we get some $C_1>0$ such that
 \bel{D-r}\ba{ll}
 \ds \dbE ^{\cF_t}\[\int_t^\i |\cD_r|^2\mathrm dr\]\les  2\dbE ^{\cF_t}\[\int_t^\i ( |\widetilde{X} _{r}^u-X_r^{x_\d;u}|^2+ K^2\d^2(1-\d)^2| \h X^u _r |^4)\mathrm d r\]\\
 \ns\ds \les  C_1\(|\widetilde{X} _t^u-X^{x_\d;u}_t  |^2+   \d^2( 1-\d )^2 |\h X^u _t |^4\) ,\q \dbP\mbox{-a.s.}\ea\ee
%


\bl\label{Le-y-y-4}\sl Under the assumptions of Proposition \ref{Le-Semi},
we get a constant $\k>0$ such that, for any $u(\cd)\in \cU_{0,\i}^p$,
$$  |\sY^u _0 -Y^{x_\d;u}_0| \les   \k\d  (1-\d)  |x_1-x_2|^2.
 $$

\el

\begin{proof}
First,
%
 by {applying} It\^o's formula to $|\sY^u _s -Y^{x_\d;u}_s|^2$, we have, for all $T>s\ges 0$ and $\e>0$,
  $$\ba{ll}
 \ds |\sY^u _s -Y^{x_\d;u}_s|^2+ \dbE^{\cF_s}\[\int_s^T\big(| \cZ^u _r-Z^{x_\d;u}_r |^2 +| \cK^u _r(\cd)-K^{x_\d;u}_r(\cd)|^2_{\l,2} \big)\mathrm dr \] \\
  \ns\ds =  \dbE^{\cF_s}\[|\sY^u _T -Y^{x_\d;u}_T|^2\\
  \ns\ds\q  + 2 \int^T_s  (\sY^u_r -Y^{x_\d;u}_r)\(  -\ell_x|\widetilde{X} _{r}^u- X^{x_\d;u}_r|-\ell_y\cD_r-K \d(1-\d)\big(|\h X^u _r |^2  + |\h  Z ^{ u }_r |^2+|\h  K ^{ u }_r(\cd) |_{\l,2}^2\big)  \\
 \ns\ds \q +      f \big( X^{x_\d;u}_r,\sY^u_r ,\cZ^u_r ,\int_E\cK^u_r(e)\rho(e)\l(\mathrm de),   u_r  \big)-f\big(X^{x_\d;u}_r,Y^{x_\d;u}_r,Z^{x_\d;u}_r,\int_E K^{x_\d;u}_r(e)\rho( e)\l(\mathrm{d}e),u_r\big) \)\mathrm  dr\]\\
 %
     \ns\ds\les  \dbE^{\cF_s}\[|\sY^u _T -Y^{x_\d;u}_T|^2   -2\a_f \int^T_s|\sY^u_r -Y^{x_\d;u}_r|^2 \mathrm dr\\
 \ns\ds\q  + 2  \int^T_s|\sY^u_r -Y^{x_\d;u}_r| \cd\(\ell_z|\cZ^u_r-Z^{x_\d;u}_r|+ \ell_k\int_E|\cK^u_r(e)-K^{x_\d;u}_r(e)|\cd|\rho(e)|\l(\mathrm de)   \)\mathrm dr\\
 \ns\ds\q  + 2  \int^T_s( \sY^u_r -Y^{x_\d;u}_r)\cd\(-\ell_x|\widetilde{X} _{r}^u- X^{x_\d;u}_r|-\ell_y\cD_r -K\d(1-\d) \big(|\h X^u _r |^2 + |\h  Z ^{ u }_r |^2+|\h  K ^{ u }_r(\cd) |_{\l,2}^2\big)  \)\mathrm  dr\].
 \ea $$
Observe that, by using Lemma \ref{Le-Yp}-(iii) for the case $p>4$, we get, for $ \vartheta>0,$   $1<\m<2,$ and  $\n>2$ with $\frac 1 \m+\frac 1 \n=1 $,
$$\ba{ll}  \ds 2K\d(1-\d)\dbE^{\cF_s}\[  \int^T_s( \sY^u_r -Y^{x_\d;u}_r)   (|\h  Z ^{ u }_r |^2+|\h  K ^{ u }_r(\cd) |_{\l,2}^2   )\mathrm  dr\]\\
  \ns\ds\les   2K\d(1-\d)\dbE^{\cF_s}\[ \sup\limits_{r\in[s,T]}| \sY^u_r -Y^{x_\d;u}_r| \int^T_s   (|\h  Z ^{ u }_r |^2+|\h  K ^{ u }_r(\cd) |_{\l,2}^2   )\mathrm  dr\]\\
    \ns\ds\les   2K\d(1-\d)\(\dbE^{\cF_s}\[ \sup\limits_{r\in[s,T]}| \sY^u_r -Y^{x_\d;u}_r|^\m\]\)^\frac 1\m\Bigg(\!\! \(\dbE^{\cF_s}\[\(\int^T_s   \! |\h  Z ^{ u }_r |^2  \mathrm  dr\)^\n\]\)^\frac 1\n+\(\dbE^{\cF_s}\[\(\int^T_s \!   |\h  K ^{ u }_r(\cd) |_{\l,2}^2    \mathrm  dr\)^\n\]\)^\frac 1\n\! \Bigg)\\
    \ns\ds\les  \vartheta\(\dbE^{\cF_s}\[ \sup\limits_{r\in[s,T]}| \sY^u_r -Y^{x_\d;u}_r|^\m\]\)^\frac 2\m\\
    \ns\ds\q +C_\vartheta\d^2(1-\d)^2 \Bigg(\!\! \(\dbE^{\cF_s}\[\(\int^T_s    |\h  Z ^{ u }_r |^2  \mathrm  dr\)^\n\]\)^\frac 2\n+\(\dbE^{\cF_s}\[\(\int^T_s    |\h  K ^{ u }_r(\cd) |_{\l,2}^2    \mathrm  dr\)^\n\]\)^\frac 2\n\!\Bigg)\\
    \ns\ds\les  \vartheta\(\dbE^{\cF_s}\[ \sup\limits_{r\in[s,T]}| \sY^u_r -Y^{x_\d;u}_r|^\m\]\)^\frac 2\m+C_\vartheta \d^2(1-\d)^2 |\h X^u _s|^4 .
  \ea$$
 Consequently, from our above estimate we deduce with the help of  Lemma \ref{Le-X-4} that
 $$\ba{ll}
 \ds |\sY^u _s -Y^{x_\d;u}_s|^2+ \dbE^{\cF_s}\[\int_s^T\big(| \cZ^u _r-Z^{x_\d;u}_r |^2 +| \cK^u _r(\cd)-K^{x_\d;u}_r(\cd)|^2_{\l,2} \big)\mathrm dr \] \\
    \ns\ds\les  \dbE^{\cF_s}\[|\sY^u _T -Y^{x_\d;u}_T|^2   -(2\a_f-
    \ell_z^2-\ell_k^2|\rho(\cd)|_{\l,2}^2-7\e)  \int^T_s|\sY^u_r -Y^{x_\d;u}_r|^2 \mathrm dr \\
 \ns\ds\q  +   \int^T_s \(\frac{\ell_z^2}{\ell_z^2+\e}|\cZ^u_r-Z^{x_\d;u}_r|^2+ \frac{\ell_k^2|\rho(\cd)|_{\l,2}^2}{\ell_k^2|\rho(\cd)|_{\l,2}^2+\e} |\cK^u_r(\cd)-K^{x_\d;u}_r(\cd)|^2_{\l,2}  +\frac{\ell_x^2}{\e}|\widetilde{X} _{r}^u- X^{x_\d;u}_r|^2\\
 \ns\ds\qq\qq\q  +\frac{\ell_y^2}{\e}|\cD_r|^2 +\frac{(K\d)^2(1-\d)^2}{\e} |\h X^u _r |^4 \)\mathrm dr \]+ \vartheta\(\dbE^{\cF_s} \[\sup\limits_{r\in[s,T]}|\sY^u_r -Y^{x_\d;u}_r|^\m\]\)^\frac2\m   +C_\vartheta\d^2(1-\d)^2  |\h X^u _s|^4 \\
 %
    \ns\ds  \les  \dbE^{\cF_s}[|\sY^u _T -Y^{x_\d;u}_T|^2]   -(2\bar\a_{f} -7\e) \dbE^{\cF_s}\[\int^T_s|\sY^u_r -Y^{x_\d;u}_r|^2 \mathrm dr\]\\
 \ns\ds\q   +  \dbE^{\cF_s}\[\int^T_s \(\frac{\ell_z^2}{\ell_z^2+\e}|\cZ^u_r-Z^{x_\d;u}_r|^2+ \frac{\ell_k^2|\rho(\cd)|_{\l,2}^2}{\ell_k^2|\rho(\cd)|_{\l,2}^2+\e} |\cK^u_r(\cd)-K^{x_\d;u}_r(\cd)|^2_{\l,2} + C_{\e,\d}|\cD_r|^2\)\mathrm dr \]\\
\ns\ds\q   + \vartheta\(\dbE^{\cF_s} \[\sup\limits_{r\in[s,T]}|\sY^u_r -Y^{x_\d;u}_r|^\m\]\)^\frac2\m +C_\vartheta  \d^2(1-\d)^2|\h X^u _s|^4+ { C|\widetilde X_s-X_s^{x_\d;u}|^2},
 \ea $$
where $\bar\a_f$   is defined in {\bf (C2)}.

 By choosing $\e>0$ sufficiently small to ensure $2\bar\a_{f} -7\e>0$,  we get,    for all $s\in[0,T]$,
  $$\ba{ll}
 \ds |\sY^u _s -Y^{x_\d;u}_s|^2+  \dbE^{\cF_s}\[\int^T_s\(|\sY^u_r -Y^{x_\d;u}_r|^2  + | \cZ^u _r-Z^{x_\d;u}_r |^2  +| \cK^u _r(\cd)-K^{x_\d;u}_r(\cd)|^2_{\l,2}\) \mathrm dr \] \\
    \ns\ds\les C\(  \d^2(1-\d)^2 |\h X^u _s|^4+ {  |\widetilde{X} _{s}^u-X_s^{x_\d;u}|^2}+  \dbE^{\cF_s}\[|\sY^u _T -Y^{x_\d;u}_T|^2+   \int^T_0|\cD_r|^2\mathrm dr\]  \\
     \ns\ds\qq +  \vartheta\(\dbE^{\cF_s} \[\sup\limits_{r\in[0,T]}|\sY^u_r -Y^{x_\d;u}_r|^\m\]\)^\frac2\m  \).
 \ea $$
Furthermore, using Doob's martingale inequality and  \eqref{D-r},
 we have
  $$\ba{ll}
 \ds \dbE\[\sup\limits_{s\in[0,T]}|\sY^u _s -Y^{x_\d;u}_s|^2\] \\
 \ns\ds \les  C\(  \dbE\[ \d^2(1-\d)^2\sup\limits_{s\in[0,T]}|\h X^u _s|^4+ { \sup\limits_{s\in[0,T]} |\widetilde{X} _{s}^u-X_s^{x_\d;u}|^2}+|\sY^u _T -Y^{x_\d;u}_T|^2     +  \int^T_0|\cD_r|^2\mathrm dr\]\\
 \ns\ds \qq\qq + \vartheta\dbE\[\sup\limits_{s\in[0,T]}\(\dbE^{\cF_s} \[\sup\limits_{r\in[0,T]}|\sY^u_r -Y^{x_\d;u}_r|^\m\)^\frac2\m\]\) \\
    \ds\les C\(\d^2(1-\d)^2 |x_1-x_2|^4  + \dbE [|\sY^u _T -Y^{x_\d;u}_T|^2  ]+    \vartheta\(\frac 2{2-\m}\)^\frac2\m\dbE\[\sup\limits_{r\in[0,T]} |\sY^u_r -Y^{x_\d;u}_r|^2\]\).
 \ea $$
 Therefore, by choosing $\vartheta$ sufficently small and applying Lemma \ref{Le-X-4},  we see that for  some constant $\k>0$ it holds 
$$  |\sY^u _0 -Y^{x_\d;u}_0|^2\les   \dbE\[\sup\limits_{s\in[0,T]}|\sY^u _r -Y^{x_\d;u}_r|^2\]\les \k^2\( \dbE [|\sY^u _T -Y^{x_\d;u}_T|^2 ]  +   \d^2(1-\d)^2 |x_1-x_2|^4\).
  $$
  Applying  Lemma \ref{Le-Yp} to $\sY^u$ and $Y^{x_\d;u}$, we also have    $ \dbE [|\sY^u _T|^2 ] \to 0$ and $ \dbE [| Y^{x_\d;u}_T|^2 ]   \to 0$, as $T\to\i$.
Hence, letting $T\to\i$ in the above inequality, we get the stated result.

\end{proof}

With the above results, we can now  complete the proof of  Proposition  \ref{Le-Semi}.

\ss

 \no{\bf Proof of  Proposition \ref{Le-Semi}}
  We just need to prove  $-V_p(\cd)$ is semi-concave.
By denoting $\widetilde V_p(\cd):=-V_p(\cd)$, we have  $\widetilde V_p(x):=\inf\limits_{u(\cd)\in\cU_{0,\i}^p}(-Y_0^{x;u})$, $x\in\dbR^n$.
Then, for all $x_1,x_2\in\dbR^n$, $\delta\in[0,1]$ and,
 for arbitrarily small  $\e>0$, there exists some $u^\e(\cd)\in \cU_{0,\i}^p$, such that
 \bel{equmm}  \d\widetilde V_p(x_1)+(1-\d)\widetilde V_p(x_2)-\widetilde V_p(x_\d ) \les -\d   Y_0^{x_1;u^\e} -(1-\d)Y_0^{x_2;u^\e}+Y_0^{x_\d;u^\e} +\e=-\widetilde Y_0^{u^\e} +Y_0^{x_\d;u^\e} +\e.
 \ee
 From  the    Lemmas \ref{Le-y-cY} and  \ref{Le-y-y-4}, we have that there is a constant $\k>0$ such that
 $$
 Y_0^{x_\d;u^\e} -\ti Y_0^{u^\e}\les  Y_0^{x_\d;u^\e} -\cY_0^{u^\e} = Y_0^{x_\d;u^\e} -\sY_0^{u^\e} -\cD_0\les  \k\d(1-\d)|x_1-x_2|^2.
 $$
Hence, by the arbitrariness of $\e$,
$  V_p(x)$ is semi-convex in $x\in\dbR^n$.
\endpf

\section{Dynamic programming principles and HJB equations with integral-differential operators}

In this section, we will   carry out the research of Problem (OC)$_p$ by employing the approach of   dynamic programming principle (DPP, for short), which shall  associate Problem (OC)$_p$ with  a kind of integral-PDE of HJB type.    The following study can go through with any given $p\ges 2$, so that we shall restrict ourselves to  $p=2$ here,   and we   set $\cU:=\cU_{0,\i}^2$.  All the subscripts $p$ will be  omitted in this section.

\subsection{Dynamic programming principles}

The following is the dynamic programming principle (DPP) of Problem (OC).

\bp\label{DPP}\sl
Under {\bf (C1)}-{\bf (C4)}, for all $x\in\dbR^n$ and  $t\ges 0$, the following holds true,
$$V(x)= \sup\limits_{u(\cd)\in\cU}G_{0,t}^{x;u}[V(X_t^{x;u})].$$

\ep

\br{}\sl Recall  that  $G_{0,t}^{x;u}[\cd]$ in the  above relation is    the backward stochastic  semigroup introduced by Peng \cite{Peng-1997} for  stochastic recursive control problems. For   the convenience of the reader,   we recall its definition, but  we also refer  to \cite{BHL-2011, LP-2009}.
Given the initial state $x\in\dbR^n$  for SDE \eqref{state}, for all $T>0$, $s\in[0,T]$, $\z\in L_{\cF_T}^2(\Omega;\dbR)$ and $u(\cd)\in \cU $,    the   backward  stochastic semigroup  $G^{x;u}_{s,T}[\cd]$  is defined by
 $G^{x;u}_{s,T}[\z]=y_s,$ $ s\in [0,T],$
where  $(y,z)$ is the unique solution of the BSDE:
 \bel{BSDE-BS}\left\{
\ba{ll}
 \ds\!\!\!  \mathrm dy_s=-f(X^{x;u}_s,y_s,z_s,\int_E k_s(e)\rho( e)\l(\mathrm{d}e),u_s)\mathrm ds  +z_s\mathrm dB_s+\int_Ek_s(e)\ti{\m}(\mathrm ds,\mathrm de),\ s\in[0,T],\\
 \ds\!\!\! y_T=\z,
 \ea\right.\ee
with $X^{x;u}$ satisfying \eqref{state}.

\er

Before  presenting the proof of Proposition \ref{DPP},   some preparations and auxiliary results are needed.
First, we construct a measurable metric dynamical system through defining a measurable and measure-preserving shift; we refer  to \cite{LZ-2019, BLZ-2021}.  Let $\th_t:\Omega\to\Omega$, $t\ges 0$, be a measurable mapping on $(\Omega,\cF,\dbP)$  defined by
$$\ba{ll}
 \ds \th_t\circ B_s:=B_{s+t}-B_t,\q s\ges 0,\\
 \ds \th_t\circ \mu\big((r,s]\times A\big):=\mu\big((r+t,s+t]\times A\big),\q r<s,\  A\in\cB(E). \ea$$
 Then, for all $s,t\ges 0$, we have

 (i) $\dbP\cd [\th_t]^{-1}=\dbP$;

 (ii) $\th_0=I$, where $I$ is the identity transformation on $\Omega$;

 (iii) $\th_s\circ \th_t=\th_{s+t}$.
\\
By setting  $B^t_r:= B_{r+t}-B_t$,  $\mu^t\big((r,s]\times A\big):=\mu\big((r+t,s+t]\times A\big)$,  $s>r\ges 0$, $A\in\cB(E)$, we consider the filtration $\dbF^t=(\cF_s^t)_{s\ges 0}$ with $\cF_s^t:=\cF_s^{B^t}\otimes\cF_s^{\m^t}$   augmented by all $\dbP$-null sets.
Let us denote the set  of the  time-shifted admissible controls  by $\cU^t:=\{u^t=\th_t\circ u\mid u(\cd)\in \cU\}$.

Now, for any $(t,x)\in [0,\i)\times\dbR^n$,
we introduce  the following initial time-shifted  version of  \eqref{state},
\bel{state-tx}   X^{t, x;u}_s=x+\int_t^sb(X^{t,x;u}_r,u_r )\mathrm dr+\int_t^s\si(X^{t,x;u}_r,u_r )\mathrm dB_r   +\int_t^s\int_E\g(e,X^{t,x;u}_{r-},u_{r} )\ti{\m}(\mathrm dr,\mathrm de),\ s\ges t ,\!\! \ee
and the corresponding  BSDE with jumps:
 \bel{BSDE-t}
\ba{ll}
 \ds\!\!\!  Y^{t,x;u}_s=Y^{t,x;u}_T+\int_s^Tf\big(X^{t,x;u}_r,Y^{t,x;u}_r,Z^{t,x;u}_r,\int_E K^{t,x;u}_r(e)\rho(e)\l(\mathrm{d}e),u_r\big)\mathrm dr  \\
 \ds\!\!\!\qq\qq-\int_s^TZ^{t,x;u}_r\mathrm dB_r-\int_s^T \int_EK^{t,x;u}_r(e)\ti{\m}(\mathrm dr,\mathrm de),\q t\les s\les T<\i.
 \ea\ee

 \bl{}\sl Under {\bf (C1)}-{\bf (C3)}, for all $t,s\ges 0$, $u(\cd)\in\cU$, we have
   $$\ba{c}
   \ds \th_t\circ X^{x;u}_s=X^{t,x;\widetilde u   }_{s+t},\q \dbP\mbox{-a.s.},\\
\ns\ds (\th_t\circ Y^{x;u}_s,\th_t\circ Z^{x;u}_s,  \th_t\circ K^{x;u}_s(\cd))= (Y^{t,x;\widetilde u  }_{s+t},Z^{t,x;\widetilde u  }_{s+t}, K^{t,x;\widetilde u  }_{s+t}(\cd)),\q \dbP\mbox{-a.s., }\ea$$
where \bel{wt-u} \widetilde u_s:=\left\{\ba{ll}
                   \ds\!\!\! u_0,\qq   s\in[0,t),\\
                    \ns\ds\!\!\! u^t_{s-t},\q  s\in[t,\i),
                    \ea \right. \mbox{ with  }u_0\in U \mbox{ being given arbitrarily}.
 \ee
 \el

 \begin{proof}

For any $t\ges 0$, by applying the transformation $\th_t$ to SDE \eqref{state}, we get
\bel{SDE-2}
 \ba{ll}
 \ds\!\!\! \th_t\circ X^{x;u}_s 
 =x+\int_0^sb( \th_t\circ X^{x;u}_r,  u_r^t )\mathrm dr+\int_0^s \si(\th_t\circ X^{x;u}_r,  u_r^t )\mathrm dB_r^t\\
 \ns\ds\qq\qq  +\int_0^s\int_E \g(e,\th_t\circ X^{x;u}_{r-}, u_{r} ^t)\ti{\m}^t(\mathrm dr,\mathrm de),\q s\ges 0.

 \ea  \ee
Observe that   
$\widetilde u _\cd  \in\cU^t$, and that   $X^{x;\widetilde u   }_{s+t}$ satisfies the following SDE:
   \bel{SDE-2-1}
 \ba{ll}
  \ds\!\!\! X^{t,x;\widetilde u  }_{s+t}=x+\int_t^{s+t} b(X^{t,x;\widetilde u  }_r,\widetilde u _{r} )\mathrm dr+\int_t^{s+t} \si(X^{t,x;\widetilde u }_r,\widetilde u _{r} )\mathrm dB_r +\int_t^{s+t}\int_E\g(e,X^{t,x;\widetilde u  }_{r-},\widetilde u _{r} )\ti{\m}(\mathrm dr,\mathrm de)\\
 \ns\ds  =x+\int_t^{s+t} b(X^{t,x;\widetilde u  }_r,u_{r-t}^t )\mathrm dr+\int_t^{s+t} \si(X^{t,x;\widetilde u  }_r,u_{r-t}^t )\mathrm dB_r  +\int_t^{s+t}\int_E\g(e,X^{t,x;\widetilde u  }_{r-},u_{ r-t }^t )\ti{\m}(\mathrm dr,\mathrm de)\\
\ns\ds =x+\int_0^s b(X^{t,x;\widetilde u  }_{r+t},u_{r}^t )\mathrm dr+\int_0^s \si(X^{t,x;\widetilde u  }_{r+t},u_{r}^t )\mathrm dB_r^t +\int_0^{s}\int_E\g(e,X^{t,x;\widetilde u  }_{(r+t)-},u_{r}^t )\ti{\m}^t(\mathrm dr,\mathrm de),\q s\ges 0.
 \ea  \ee
 From the SDEs  \eqref{SDE-2} and \eqref{SDE-2-1}, by using the uniqueness of the solution of these SDEs, we get  $$\th_t\circ X^{x;u}_s=X^{t,x;\widetilde u   }_{s+t},\q  \dbP\mbox{-a.s.}, \ t,s\ges 0.$$

 Similarly, by applying $\th_t$ on the both sides of  BSDE  \eqref{BSDE-cost}, we have
$$
\ba{ll}
 \ds\!\!\! \th_t\circ Y^{x;u}_s=\th_t\circ Y^{x;u}_T  +\int_s^Tf\(\th_t\circ X^{x;u}_r,\th_t\circ Y^{x;u}_r,\th_t\circ Z^{x;u}_r,\int_E \th_t\circ K^{x;u}_r(e)\rho(e)\l(\mathrm{d}e),u_r^t\)\mathrm dr\\
 \ns\ds\!\!\!\qq\qq\q -\int_s^T\th_t\circ Z^{x;u}_r\mathrm dB_r^t-\int_s^T\int_E\th_t\circ K^{x;u}_r(e)\ti{\m}^t(\mathrm dr,\mathrm de),\q 0\les s\les T<\i.
 \ea$$
On the other hand, from  BSDE \eqref{BSDE-t}  we also  obtain
$$
\ba{ll}
 \ds\!\!\! Y^{t,x;\widetilde u   }_{s+t}=Y^{t,x;\widetilde u   }_{T+t}+\int_{s+t}^{T+t}f (X^{t,x;\widetilde u   }_r,Y^{t,x;\widetilde u   }_r,Z^{t,x;\widetilde u   }_r,\int_EK^{t,x;\widetilde u   }_r(e)\rho(e)\l(\mathrm{d}e),\widetilde u  _r)\mathrm dr\\
\ns\ds\qq\qq\qq -\int_{s+t}^{T+t} Z^{t,x;\widetilde u   }_r \mathrm dB_r-\int_{s+t}^{T+t}\int_EK^{t,x;\widetilde u   }_{r}(e)\ti{\m}(\mathrm dr,\mathrm de)\\
%
\ns\ds=Y^{t,x;\widetilde u   }_{T+t}+\int_{s}^{T}f (X^{t,x;\widetilde u   }_{r+t},Y^{t,x;\widetilde u   }_{r+t},Z^{t,x;\widetilde u   }_{r+t},\int_EK^{t,x;\widetilde u   }_{r+t}(e)\rho( e)\l(\mathrm{d}e),  u^t_{r} )\mathrm dr\\
\ns\ds\qq\qq\qq -\int_{s}^{T} Z^{t,x;\widetilde u   }_{r+t} \mathrm dB_r^t-\int_{s}^{T}\int_EK^{t,x;\widetilde u   }_{r+t}(e)\ti{\m}^t(\mathrm dr,\mathrm de), \q 0\les s\les T<\i.
\ea$$
The uniqueness of the solution of BSDE with jumps implies   that, for all $t,s\ges 0$,  
$$(\th_t\circ Y^{x;u}_s,\th_t\circ Z^{x;u}_s,  \th_t\circ K^{x;u}_s(\cd))= (Y^{t,x;\widetilde u   }_{s+t},Z^{t,x;\widetilde u   }_{s+t}, K^{t,x;\widetilde u   }_{s+t}(\cd)),\ \dbP\mbox{-a.s.}$$

\end{proof}

\br{}\sl
Note that, for all $x\in\dbR^n$, $V(x)$ is deterministic. Therefore, for all $x\in\dbR^n$, $t\ges 0$,
\bel{W-th}
V(x)= \th_t\circ V(x)= \esssup_{u(\cd)\in\cU }\big(\th_t\circ Y_0^{x;u} \big)=\esssup_{u(\cd)\in\cU } Y^{t,x; \widetilde u   }_{t}=\esssup_{u(\cd)\in\cU^t } Y^{t,x; u   }_{t} .\ee
%
\er

Now we   present the proof of Proposition  \ref{DPP}.
\ss

\no{\bf Proof of Proposition \ref{DPP}} To simplify the  notation, we
 set  $\bar V(x):= \sup\limits_{u(\cd)\in\cU }G_{0,t}^{x;u}[V(X_t^{x;u})].$

  First,  we prove $\bar V(x)\les V(x)$, $x\in\dbR^n$.
Obviously, for any $\e>0$, there exists some $u^\e(\cd)\in\cU$ such that
$$\bar V(x)\les   G_{0,t}^{x;u^\e}[V(X_t^{x;u^\e})]+\e.$$
On the other hand, from  \eqref{W-th}, by using a standard argument (see, eg.,   \cite{Peng-1997}  or \cite{BHL-2011}), we get 
%
$$V(X_t^{x;u^\e})= \esssup_{u(\cd)\in\cU} Y^{t,y; \widetilde u}_{t}|_{y=X_t^{x;u^\e}}=\esssup_{u(\cd)\in\cU} Y^{t,X_t^{x;u^\e}; \widetilde u}_{t}
$$
(Recall \eqref{wt-u} for the definition of $\widetilde u$ given $u\in\cU$). 
\\
To estimate $\esssup\limits_{u(\cd)\in\cU} Y^{t,X_t^{x;u^\e}; \widetilde u }_{t}$, we note that there exists some sequence $\{u_j\}_{j\ges 1}\subset\cU$ such that
$$\esssup\limits_{u(\cd)\in\cU} Y^{t,X_t^{x;u^\e}; \widetilde u }_{t}=\sup\limits_{j\ges 1}Y^{t,X_t^{x;u^\e}; \widetilde u_j }_{t},\q  \dbP\mbox{-a.s.}
$$
For $j\ges 1$,  by setting $$\ti\D_j:=\Big\{  \esssup\limits_{u(\cd)\in\cU} Y^{t,X_t^{x;u^\e}; \widetilde u }_{t}\les Y^{t,X_t^{x;u^\e}; \widetilde u_j }_{t}+\e\Big\}\in \cF_t,
$$
we   construct an $(\Omega,\cF_t)$-partition $\{\D_j\}_{j\ges 1}$ with $\D_1=\ti\D_1$, $\ds \D_j:=\tilde\D_j\backslash ( \cup_{i=1}^{j-1}\ti\D_j)$, $j\ges 2$.

Defining  $\ds\mathbbm{u}_\e:=\sum_{j\ges1}u_jI_{ \D_j}$, we have $\mathbbm{u}_\e(\cd)\in\cU$ and
$
Y^{t,X_t^{x;u^\e}; \widetilde { {\mathbbm{u}}}_\e }_{t}=\sum\limits_{j=1}^\i Y^{t,X_t^{x;u^\e}; \widetilde u_j}_{t}I_{ \D_j}
 $,  where $\widetilde { {\mathbbm{u}}}_\e $ is defined  by    \eqref{wt-u} given $  \mathbbm{u}_\e\in\cU$. 
Therefore,
$$ V(X_t^{x,u^\e})= \esssup_{u(\cd)\in\cU} Y^{t,X_t^{x;u^\e}; \widetilde u }_{t} \les Y^{t,X_t^{x;u^\e};\widetilde { {\mathbbm{u}}}_\e }_{t}+\e=Y^{0,x; \mathbf{u}^\e}_{t}+\e ,
 $$
  where  $\mathbf{u}^\e_s :=(u^\e\oplus_t\widetilde {\mathbbm{u}}_\e)_s:=\left\{\ba{ll}\ns\ds\!\!\! u^\e_s,\qq\q \qq\qq  \qq  \ \ s\in[0,t),\\
 \ns\ds\!\!\!  (\mathbbm{u}_\e)_{s-t}^t=(\th_t\circ \mathbbm{u}_\e)_{s-t}, \q  s\in[t,\i).\ea\right.$
Note that $\mathbf{u}^\e(\cd) \in\cU$.
%

 %

%

Thus, by the continuous dependence of the solutions of BSDEs with jumps, we get
$$\ba{ll}
\ds \bar V(x) \les  G_{0,t}^{x;u^\e}[V(X_t^{x;u^\e})]+\e  \les G_{0,t}^{x;u^\e}[Y^{0,x;\mathbf{u}^\e}_{t}+\e]+\e   \les  Y^{0,x;\mathbf{u}^\e}_{0}+C\e   \les \sup\limits_{u(\cd)\in\cU}  Y^{x; u}_{0}+C\e .\ea$$
Due to the arbitrariness  of $\e>0$, we get
$  \bar V(x) \les      \sup\limits_{u(\cd)\in\cU}  Y^{x; u}_{0}  =V (x),$ $ x\in\dbR^n.  $

 \ms

Let us now  prove $\bar V(x)\ges V(x)$, $x\in\dbR^n$. By the definition of  $\bar V(\cd)$, for all $x\in\dbR^n$, $u(\cd)\in\cU$,
 $$\bar V(x)\ges  G_{0,t}^{x;u}[V(X_t^{x;u})] .$$
Let $\{\L_i\}_{i\ges 1}\subset\cB(\dbR^n)$ be a decomposition of $\dbR^n$ such that $\sum\limits_{i\ges 1}\L_i=\dbR^n$ and $\diam{(\L_i)}<\e$, $i\ges 1$. For arbitrarily  chosen but fixed  $y_i\in \L_i$, $i\ges 1$, we define $[X_t^{x;u}]:=\sum\limits_{i\ges1}y_i1_{\{X_t^{x;u}\in\L_i\}}$. Then, for all   $u(\cd)\in\cU$,
 $$| X_t^{x;u}-[X_t^{x;u}]|\les \e,   \mbox{ everywhere on }\Omega.$$

Moreover, by \eqref{W-th}, using the comparison of BSDEs with jumps and Lemma \ref{W-Lip}, we have, for all $x\in\dbR^n$, $u(\cd)\in\cU$,
 $$\ba{ll}
 \ds \bar V(x) \ges G_{0,t}^{x;u}\big[V(X_t^{x;u})\big]  \ges G_{0,t}^{x;u}\big[V([X_t^{x;u}])-C\e\big]  \\
 \ns\ds   \ges G_{0,t}^{x;u}\[\sum\limits_{i\ges1} 1_{\{X_t^{x;u}\in\L_i\}} V(y_i) \]-C\e = G_{0,t}^{x;u}\[\sum\limits_{i\ges1} 1_{\{X_t^{x;u}\in\L_i\}}\esssup\limits_{v\in\cU }Y_0^{y_i,v} \]-C\e \\
  \ns\ds = G_{0,t}^{x;u}\[\sum\limits_{i\ges1} 1_{\{X_t^{x;u}\in\L_i\}}\esssup\limits_{v\in\cU }Y_t^{t,y_i,\tilde v} \]-C\e = G_{0,t}^{x;u}\[ \esssup\limits_{v\in\cU }Y_t^{t,[X_t^{x;u}],\tilde v} \]-C\e\\
 \ns\ds \ges   G_{0,t}^{x;u}\[ \esssup\limits_{v\in\cU }Y_t^{t, X_t^{x;u} ;\tilde v} -C\e\]-C\e\ges G_{0,t}^{x;u}\[ \esssup\limits_{v\in\cU }Y_t^{ x;u\oplus\tilde v} \]-C\e\\
 \ns\ds
 \ges     G_{0,t}^{x;u}\[   Y_t^{ x;u} \]-C\e = Y^{x;u }_{0}-C\e.\ea$$
%
%
Therefore,
 $  \bar V(x)  \ges    \sup\limits_{u(\cd)\in\cU} Y^{x; u }_{0}  -C\e=V(x)-C\e.
 $
Finally, using the arbitrariness of $\e>0$, we get
 $
  \bar V(x) \ges  V(x), $ $x\in\dbR^n $.  The stated relation  in Proposition \ref{DPP} is obtained.

 \endpf

\subsection{HJB equations with integral-differential operators}

We consider the following HJB equation,
\bel{HJB-W}
\sup\limits_{u\in U}\[\cL^{u}V(x)  +\cB^uV(x) +f(x,V(x),DV(x).\si(x,u),\cC^uV(x),u) \]=0, \q  x\in\dbR^n,
 \ee
which involves the following  differential and integral   operators,
$$\ba{ll}\cL^{u}V(x): =DV(x).b(x,u) +\frac12\tr\big(\si\si^\top(x,u)D^2V(x)\big),\\
\ns\ds  \cB^uV(x):=\int_{E}\big[V(x+\g(e,x,u))-V(x)-DV(x).\g(e,x,u)\big]\l(\mathrm{d}e),\\
\ns\ds \cC^uV(x):=\int_E\big[ V(x+\g(e,x,u))-V(x)\big]\rho(e)\l(\mathrm{d}e),\q x\in\dbR^n,\ u\in U.\ea$$

In this part, we aim to associate  the value function $V(\cd)$ in \eqref{bar u} with the above HJB equation.
 Note that, under our hypotheses {\bf (C1)}-{\bf (C4)},     the value function $V(\cd)$ is not necessarily smooth. Therefore, we shall resort to a kind of  weak solution, the viscosity solution, which was  introduced     by Crandall, Lions \cite{CL-1983}; we also refer  to   Crandall et al. \cite{CLP-1992}.
Let us recall  the definition of a viscosity solution for  \eqref{HJB-W}.

\bde\label{Def-vis}\sl $V(\cd)\in C(\dbR^n)$ is said to be

(i) a viscosity subsolution of HJB equation \eqref{HJB-W}, if for any $\f\in C^2(\dbR^n)$  and any sufficiently small $\d>0$,
$$  \sup\limits_{u\in U}\[ \cL^{u}\f(x) +\cB^u_\d(V,\f)(x) +f(x,V(x),D\f(x).\si(x,u),\cC^u_\d(V,\f)(x),u) \]\ges 0 $$
holds at any local maximum point $x\in\dbR^n$ of $V-\f$, where
$$\ba{ll}
 \ds \cB^u_\d(V,\f)(x):=\int_{E_\d}\big[\f(x+\g(e,x,u))-\f(x)-D\f(x).\g(e,x,u)\big]\l(\mathrm{d}e)\\
 \ds\qq\qq\qq\qq +\int_{E_\d^c}\big[V(x+\g(e,x,u))-V(x)-D\f(x).\g(e,x,u)\big]\l(\mathrm{d}e),\\
 \ds \cC^u_\d(V,\f)(x):=\int_{E_\d}\big[ \f(x+\g(e,x,u))-\f(x)\big]\rho(e)\l(\mathrm{d}e)+\int_{E_\d^c}\big[ V(x+\g(e,x,u))-V(x)\big]\rho(e)\l(\mathrm{d}e),
\ea$$
and $E_\d = \{e \in E\mid  |e|<\d\}$.

(ii) a viscosity supersolution of HJB equation \eqref{HJB-W}, if for any $\f\in C^2(\dbR^n)$ and   sufficiently small $\d>0$,
$$ \sup\limits_{u\in U}\[ \cL^{u}\f(x) +\cB^u_\d(V,\f)(x) +f(x,V(x),D\f(x),D\f(x).\si(x,u),\cC^u_\d(V,\f)(x),u) \]\les 0  $$
holds at any local minimum point $x\in\dbR^n$ of $V-\f$.

(iii) a viscosity solution of HJB equation \eqref{HJB-W}, if it is both the viscosity supersolution and subsolution of \eqref{HJB-W}.

\ede

\br\label{Re-Def}\sl
(i) In Definition \ref{Def-vis}  we can
replace $\cB^u_\d(V,\f)(x)$, $\cC^u_\d(V,\f)(x)$ by $\cB^u\f(x)$ and $\cC^u\f(x)$, respectively.
For the details, the readers can refer to \cite{BBP-1997, BHL-2011}.

(ii) Due to the linear growth of $V(\cd)$ (see Lemma \ref{W-Lip}), a standard argument allows to replace     the local maximum and minimum  points in Definition \ref{Def-vis}    by the global ones.
\er

 In what follows we present  the existence and the  uniqueness of the    viscosity solution of HJB equation \eqref{HJB-W}.

\bt\label{Th-exist}\sl  Under {\bf (C1)}-{\bf (C4)}, the value function $V(\cd)$ is the  unique  viscosity solution of HJB equation \eqref{HJB-W} in the following space
$$\Th=\Big\{\phi\in C(\dbR^n) \mid \lim\limits_{|x|\to\i}\phi(x)\exp\big\{-\ti A\big[\log\big((|x|^2+1)^\frac12\big)\big]\big\} =0, \mbox{ for some }\ti A>0 \Big\}.$$

\et

The proof of Theorem \ref{Th-exist} will be given after  a sequel of  auxiliary results.
Note that, in view of the   density argument,  the space $ C^2(\dbR^n)$ of the test functions $\f(\cd)$ in Definition \ref{Def-vis} can be replaced by $ C^\i(\dbR^n)$.
Moreover, taking into account the corresponding properties of $V(\cd)$, we
can assume additionally that the test function $\f $ and its derivatives have at most polynomial
growth as $|x|\to+\i$.
For such  $\f(\cd)\in C^\i(\dbR^n)$ and $(x,y,z,k(\cd), u)\in \dbR^n\times \dbR\times \dbR^d\times L_{\l}^2(E;\dbR)\times U$, we introduce
$$
  F(x,y,z,k(\cd),u):=  \cL^u\f(x)  +\cB^u\f(x)   +f\big(x,y+\f(x),z+D \f(x). \si(x,u),\int_Ek(e)\rho(e)\l(\mathrm{d}e)+\cC^u\f(x),u\big).
  $$
%
%
%
%
Then, for any $T>0$ and $u(\cd)\in \cU$,  let us  consider the following both  BSDEs
$$\ba{ll}
 \ds \left\{
\ba{ll}
 \ds\!\!\!  \mathrm d Y_r^{1,u}=- F\big(X^{x;u}_r,Y_r^{1,u},Z_r^{1,u},\int_E K_r^{1,u}(e)\rho(e)\l(\mathrm{d}e),u_r \big)\mathrm dr  + Z_r^{1,u}\mathrm dB_r+\int_EK_r^{1,u}(e)\ti{\m}(\mathrm dr,\mathrm de), \  r\in [0,T],\\
 \ds\!\!\!Y_T^{1,u}=0,
 \ea\right.  \ea$$
$$\ba{ll} \ds \left\{
\ba{ll}
 \ds\!\!\! \mathrm  d Y_r^{2,u}=-F\big(x,Y_r^{2,u},Z_r^{2,u},\int_E K_r^{2,u}(e)\rho(e)\l(\mathrm{d}e),u_r \big)\mathrm dr    +Z_r^{2,u}\mathrm dB_r+\int_EK_r^{2,u}(e)\ti{\m}(\mathrm dr,\mathrm de), \  r\in [0,T],\\
 \ds  \!\!\!Y_T^{2,u}=0,
 \ea \right.\ea$$
 and  the   ODE
$$\left\{
\ba{ll}
  \ds\!\!\!  \mathrm d Y_r^0=-F_0(x,Y_r^0,0,0 )\mathrm dr ,\q r\in[0,T],\\
 \ds  \!\!\!Y_T^{0}=0,
 \ea  \right.$$
with  $F_0(x,y,z,k(\cd)):=\sup\limits_{u\in U} F(x,y,z,k(\cd),u)$.

\ss
The above BSDEs have the following properties:

\bl\label{Lemma-Y-1} \sl For all $ s\in [0,T]$ and $u(\cd)\in\cU$,
$G_{s,T}^{x;u}[\f(X_T^{x;u})]-\f(X_s^{x;u})=Y_s^{1,u},$ $\dbP\mbox{-a.s. }  $

\el

\bl\label{Le-Y1-2}\sl For all $ s\in [0,T]$ and $u(\cd)\in\cU$, we have
$ |Y_s^{1,u}-Y_s^{2,u}|\les  CT^{\frac32} , $ $ \dbP\mbox{-a.s. }   $
\el

\bl\label{Le-Y1-3} \sl For all $s\in[0,T]$,
$ \ds Y_s^0=\esssup_{u(\cd)\in \cU} Y_s^{2,u},$ $\dbP\mbox{-a.s.  } $

\el

Note  that in the above statements  $T>0$ is arbitrarily  fixed,  and so the existence and the uniqueness of the solutions   of the above three equations are classical and known under our assumptions. Moreover,    the  proofs of these three Lemmas are standard and use   the classical methods of  finite horizon control problem involving  BSDEs with jumps,  we refer  to \cite{LP-2009, BHL-2011, LW-AMO}.
 Therefore, we shall  not  give their proofs again.

\ss

Based on these statements, we give the proof  of Theorem \ref{Th-exist}.

\ss

\no {\bf Proof  of Theorem \ref{Th-exist}:}
 First, we focus on the case of a supersolution.
Let $\f\in C^2(\dbR^n)$, and  $x\in\dbR^n$ be a    global minimum point of $V-\f$. Without loss of generality  we assume $V(x)=\f(x).$
By the DPP (Proposition \ref{DPP}), for all $T>0$  we have
 $$\f(x)= V(x)= \sup\limits_{u(\cd)\in\cU}G_{0,T}^{x;u}[V(X_T^{x;u})].$$
As
$V(y)\ges \f(y)$, $y\in\dbR^n$, it follows from    the comparison theorem for  BSDE with jumps that $G_{0,T}^{x;u}[V(X_T^{x;u})]\ges G_{0,T}^{x;u}[\f(X_T^{x;u})]$, $u\in\cU$, and so
  $\f( x) \ges  \sup\limits_{u(\cd)\in\cU}G_{0,T}^{x;u}[\f(X_T^{x;u})]$, $T>0$.
Thus,  by Lemma \ref{Lemma-Y-1}, we have
\bel{supY1}  \sup\limits_{u(\cd)\in\cU}Y_0^{1 ,u}= \sup\limits_{u(\cd)\in\cU} G_{0,T}^{x;u}[\f(X_T^{x;u})]-\f(x) \les 0.\ee
Further, by the Lemmas \ref{Le-Y1-2} and  \ref{Le-Y1-3}, and \eqref{supY1},
  $Y^0_0=\sup\limits_{u(\cd)\in\cU}  Y_0^{2,u}\les   C T^{\frac32}.$ That is,
      $ \ds \int_0^T F_0(x,Y_r^0,0,0)\mathrm dr=Y_0^0 \les C T^{\frac32}.$
Dividing this latter inequality by $T$ and letting  $T\to0$, we get $F_0(x, 0,0,0)\les 0$. Hence,
$$  \sup\limits_{u\in U}\[\cL^{u}\f(x) +\cB^u\f(x) +f(x,V(x),D\f(x),D\f(x).\si(x,u),\cC^u\f(x),u) \]\les 0.  $$
Thus, $V(\cd)$ is a viscosity supersolution of  HJB equation \eqref{HJB-W}.
A similar argument can be used  for   the case of viscosity subsolution.
In this way, we complete the proof  that  $V(\cd)$ is a viscosity  solution of \eqref{HJB-W}.

 Furthermore, by making adequate  modifications in the proof of the uniqueness of HJB equations in parabolic case (we refer  to  Section 5 in \cite{BHL-2011}), we   get the uniqueness of    the viscosity solution   of \eqref{HJB-W} in $\Theta$.

\endpf

\br\label{Re-uni}\sl
Note that  the growth condition of  the functions in $\Th$ is weaker than the polynomial growth.
  The value function $V(\cd)$ is indeed of  linear growth, so that the uniqueness of the viscosity solution of HJB equation \eqref{HJB-W} also holds in
    $C_{pol}(\dbR^n):=\big\{\phi\in C(\dbR^n)\mid \phi(x) \mbox{ is   polynomial growth  in }x\in\dbR^n\big  \}.$

\er

\br\label{Re-p=2}\sl By going through  the whole  section, it is not hard to see that  all the results in Section 4  are true   for  any given and fixed  $p\ges 2$. It seems that HJB equation \eqref{HJB-W} is not concerned  by the special choice of $p$. However,  they are linked together by the imposed conditions of $b,\si,\g$ and $f$.

\er

\section{Stochastic verification theorems}

The results in the previous sections present a   probabilistic interpretation of  the viscosity solution of HJB equation \eqref{HJB-W}.  Next, we shall try to construct an optimal control  of Problem (OC)$_p$
  from  the solution of HJB equation \eqref{HJB-W}. That is, we shall  establish a stochastic verification theorem for Problem (OC)$_p$. Such a
   study will be carried out
  in the framework  of   classical solutions but also in that  of  viscosity solutions.  Let $p\ges2$ again   in this section.

We first   introduce the following definition of admissible feedback control laws.

\begin{definition}\sl   A measurable mapping $\mathbbm{u}: \mathbb R^n \rightarrow U$ is called   {\it admissible  feedback control law}, if for all
	$x\in \mathbb R^n $,    the following SDE with jumps
	\bel{SDE-u}
	\left\{
	\ba{ll}
	 \ds\!\!\!  \mathrm dX^{x;\mathbbm{u}}_r  = b\big ( X^{x;\mathbbm{u}}_r ,\mathbbm{u}(
	X^{x;\mathbbm{u}}_r)   \big)\mathrm dr + \sigma\big( X^{x;\mathbbm{u}}_r,\mathbbm{u}(X^{x;\mathbbm{u}}_r)   \big)\mathrm dB_r  +\int_E \g\big(e, X^{x;\mathbbm{u}}_{r-},\mathbbm{u}(X^{x;\mathbbm{u}}_{r-})   \big)\ti{\mu}(\mathrm dr,\mathrm de)  , \ r\ges  0,\\
 \ds\!\!\!  X^{x;\mathbbm{u}}_0  = x,
	\ea
	\right.
	\ee
and the BSDE with jumps
\bel{BSDE-u}\ba{ll}
 \ds    Y^{x;\mathbbm{u}}_s=Y^{x;\mathbbm{u}}_T +\int_s^T f\big(X^{x;\mathbbm{u}}_r,Y^{x;\mathbbm{u}}_r,Z^{x;\mathbbm{u}}_r,\int_E K^{x;\mathbbm{u}}_r(e)\rho(e)\l(\mathrm{d}e), \mathbbm{u}(X^{x;\mathbbm{u}}_r) \big)\mathrm dr  -\int_s^TZ^{x;\mathbbm{u}}_r\mathrm dB_r\\
 \ds\qq\qq -\int_s^T\int_EK^{x;\mathbbm{u}}_r(e)\ti{\mu}(\mathrm dr,\mathrm de), \q   0\les s\les T<\i, \ea\ee
admit  a unique   strong solution  $ X^{x;\mathbbm{u}}\in L_\dbF^p(0,\i;\dbR^n) $ and  $(Y^{x;\mathbbm{u}} ,Z^{x;\mathbbm{u}} ,K^{x;\mathbbm{u}}(\cd))  \in\sS^p_\dbF(0,\i)$, respectively,  and $\mathbbm{u}(X^{x;\mathbbm{u}}_\cd)\in \cU_{0,\i}^p$.

We denote by $\sU^p$  the set of all such admissible feedback control laws.
\end{definition}

\subsection{Stochastic verification theorem:  the classical solution case}\label{Sec_SVP-C}
First, for  $x\in\dbR^n, $ $u\in U$ and $\Psi (\cd)\in C^2(\dbR^n)$, by defining
$$
\cH(x,\Psi(x),D\Psi(x),D^2\Psi(x),u) :=\cL^{u}\Psi(x) +\cB^u\Psi(x)   +f(x,\Psi(x),D\Psi(x).\si(x,u),\cC^u\Psi(x) ,u) ,
$$
we can  rewrite the HJB equation \eqref{HJB-W} as follows,
\bel{HJB-W-1}
\sup\limits_{u\in U}\cH(x,V(x),DV(x),D^2V(x),u)=0, \q  x\in\dbR^n.
\ee
Let us consider a measurable    mapping    $ \mathbbm{u}: \dbR^n\times\dbR\times\dbR^n\times\dbS^n\to U$ such that 
 $$\mathbbm{u}(x,y,p,A)\in\argmax_{u\in U}\big\{\cH(x,y,p,A ,u)\big\}, \q x\in\dbR^n,\  (y,p,A)\in \dbR\times\dbR^n\times \dbS^n  $$
 %
%
and we put
 \bel{OC-C}\mathbbm{u}^\Psi(x):=\mathbbm{u}(x,\Psi(x),D\Psi(x),D^2\Psi(x)),\q x\in\dbR^n, \  \Psi (\cd)\in C^2(\dbR^n).\ee

Now, we present   a stochastic verification theorem for Problem (OC)$_p$ in the case of  classical solutions.

	 \begin{theorem}\label{SVT-class}\sl
	 Let $p\ges 2$, and assume  {\bf(C1)$_p$}-{\bf(C4)}. Let $W( \cdot) \in {C^{2}}({\mathbb R^n})$ be a classical solution of the  HJB equation  \eqref{HJB-W-1}.
 	Then,
  	for all $x\in {\mathbb R^n}$ and $u(\cd) \in \cU_{0,\i}^p$, we have
 		$ W(x) \ges J(x;u(\cd))$.
Furthermore, if $\mathbbm{u}^W(\cd)$ belongs to $\sU^p$, then $ {\mathbbm{u}}^W\big( X^{x; {\mathbbm{u}^W}}_\cd\big)$ is an optimal control of Problem (OC)$_p$, where $X^{x; {\mathbbm{u}^W}}(\cd)$ satisfies
\eqref{SDE-u} with $ {\mathbbm{u}^W}(\cd)$,
%
and  the classical solution $W(\cd)$ is just the value function $V_p(\cd)$ in \eqref{bar u}, i.e.,
  $$W(x)=J\big(x; {\mathbbm{u}^W}(X^{x; {\mathbbm{u}^W}}_\cd)\big)=V_p(x),\q x\in\dbR^n.$$

\end{theorem}

 	 \begin{proof}
 	 	{\emph{Step 1.}}
For all $ x \in   {\mathbb R^n},$ $u(\cd)\in \cU_{0,\i}^p$,	by applying It\^o's formula to $W(X^{x;u} )$, we get
 	 	\bel{BSDE-W}\ba{ll}
        \ds \mathrm d W\big(X^{x;u}_r\big)
        %
  =   DW\big(X^{x;u}_r\big). \sigma\big (X^{x;u}_r,u_r\big)  \mathrm dB_r+ \int_E[W(X^{x;u}_{r-}+\g(e,X^{x;u}_{r-},u_{r}))-W( X^{x;u}_{r-})]\ti{\m}(\mathrm dr,\mathrm de) \\
       \ns\ds\qq\qq\q  -  \(f(X^{x;u}_r,W(X^{x;u}_r), DW(X^{x;u}_r).\si(X^{x;u}_r,u_r),\cC^uW(X^{x;u}_r) ,u_r) \\
       \ns\ds\qq\qq\qq\qq\qq\qq -\cH(X^{x;u}_r,W(X^{x;u}_r),DW(X^{x;u}_r),D^2W(X^{x;u}_r),u_r) \)\mathrm dr,\q  r\ges 0.
 	 	\ea\ee
By the fact that $W(\cd)$ is a classical solution  of \eqref{HJB-W}, 
for all $r\ges 0$ and $u(\cd)\in \cU_{0,\i}^p$,  we get
 \bel{compar-f}\ba{ll}
  \ds
    f\big(X^{x;u}_r,W\big(X^{x;u}_r\big),DW\big(X^{x;u}_r\big). \si(X^{x;u}_r,u_r),\cC^uW(X^{x;u}_r),u_r\big) \\
\ns\ds -\cH\big(r,X^{x;u}_r,W\big(X^{x;u}_r\big),DW\big(X^{x;u}_r\big),D^2W(X^{x;u}_r\big),u_r\big)\\
\ns\ds  \ges f\big(X^{x;u}_r,W\big(X^{x;u}_r\big),DW\big(X^{x;u}_r\big).\si(X^{x;u}_r,u_r),\cC^uW(X^{x;u}_r),u_r\big) ,\q \dbP\mbox{-a.s.}
\ea\ee
Applying the comparison theorem (Theorem \ref{Th-Com}) for infinite horizon  BSDEs with jumps  to \eqref{BSDE-W} and \eqref{BSDE-cost},   we get
$W\big(X^{x;u}_s\big) \ges  Y^{x;u}_s,$ $ \dbP\mbox{-a.s.,  } s\ges 0.$
In particular, for $s=0$  we get
 	\bel{Step1}
 	 		W(x) \ges Y^{x;u}_0= J\big(x;u(\cd)\big),\quad\mathrm{for\ all}\ u(\cd) \in \cU_{0,\i}^p .\ee

{\emph{Step 2.}}  For $
  {\mathbbm{u}}^W(\cd)\in\sU^p$  defined in \eqref{OC-C},   let $X^{x;  {\mathbbm{u}^W}}_\cd$ and $(Y^{x;{\mathbbm{u}^W}}_\cd,Z^{x;{\mathbbm{u}^W}}_\cd,K^{x;{\mathbbm{u}^W}}_\cd)$ be  the  solutions  of   SDE \eqref{SDE-u} and BSDE \eqref{BSDE-u}, respectively.   Applying It\^o's formula to $W(X^{x;{\mathbbm{u}^W}}_\cd)$, we have \eqref{BSDE-W} with $u(\cd)$ being replaced by  $ {\mathbbm{u}^W}(X^{x; {\mathbbm{u}^W}}_\cd)$.
%

Further, from   \eqref{OC-C},  we know \eqref{compar-f}  with $ {\mathbbm{u}^W}(X^{x; {\mathbbm{u}^W}}_\cd)$  instead of $u(\cdot)$ becomes an equality. By the uniqueness of the solution of the BSDE with jumps  we get
 $W\big(X^{x; {\mathbbm{u}}^W}_s\big) =  Y^{x;{\mathbbm{u}}^W}_s,$ $\dbP\mbox{-a.s.,} \ s\ges0$,
and,  in particular,
  	$W(x) = Y^{x; {\mathbbm{u}}^W}_0$, $x\in\dbR^n$.
Then, combined with \eqref{Step1},  for any $x\in\dbR^n$, we get
 $$W(x) =J\big(x;{\mathbbm{u}^W}(X^{x; {\mathbbm{u}}^W}_\cd) \big)=\sup_{u(\cd)\in\cU_{0,\i}^p}J\big(x;u(\cd)\big)=V_p(x).$$
Consequently,  the classical solution $W(\cd)$ of HJB equation  \eqref{HJB-W} is indeed the value function $V_p(\cd)$ of Problem (OC)$_p$,
and
$ {\mathbbm{u}^W}(X^{x; {\mathbbm{u}}^W}(\cd))$ is an optimal control of Problem (OC)$_p$.

 	 \end{proof}

\br{}\sl
  The function  $ {\mathbbm{u}^W}(\cd)\in\sU^p$ in Theorem \ref{SVT-class}  is said to be an optimal feedback control law of Problem (OC)$_p$.
 In fact, $ {\mathbbm{u}^W}(\cd)\in\sU^p$ is an ideal requirement. It is not easy to obtain  the needed regularity of $  {\mathbbm{u} }(\cd) $ ensuring  $ {\mathbbm{u}^W}(\cd)\in\sU^p$, especially  in an infinite horizon control problem.
   Here we   just formally give a stochastic verification theorem  for Problem (OC)$_p$ in the framework of  the classical solutions, without the further exploration of the existence of $  {\mathbbm{u} }(\cd) $ in \eqref{OC-C}.

\er

%
%
%
%
%
%
%
%

\subsection{Stochastic verification theorem: the viscosity  solution case}\label{Sec_SVP-C}

Now we are interested in the    stochastic verification theorem for  Problem (OC)$_p$ in the framework of viscosity solutions.
In order  to state the stochastic  verification theorem  more clearly in this case,
  we rewrite  the HJB equation \eqref{HJB-W}   as follows:
\bel{HJB-W-2}
\sup\limits_{u\in U}\dbH(x,V(x),DV(x),D^2V(x),\cB^uV(x),\cC^uV(x),u)=0, \q  x\in\dbR^n,
 \ee
where, for  $x\in\dbR^n, $ $\Psi (\cd)\in C^2(\dbR^n)$, $u\in U$,
 \bel{H-2}
\ba{ll}
 \ds
\dbH(x,\Psi(x),D\Psi(x),D^2\Psi(x),\cB^u\Psi(x),\cC^u\Psi(x),u) \\
\ns\ds :=\cL^{u}\Psi(x) +\cB^u\Psi(x)   +f(x,\Psi(x),D\Psi(x).\si(x,u),\cC^u\Psi(x) ,u) .
\ea \ee

However,   viscosity solutions  are not necessarily smooth.  Therefore, we need to   introduce  the  notion  of weak derivatives
 (second-order   subdifferentials)   for our studies.

%
\bde\sl  Let $ x \in  \dbR^n$ and $\Phi\in C( \dbR^n)$. The  {\it second-order  subjet} of $\Phi$ at  $x$ is defined as
$$  D^{2,-}_{x}\Phi(x):=  \Big\{(P,Q)\in \dbR^n\times \dbS^n\Big|\liminf\limits_{y\to x}\frac 1{|y- x|^2}\big[\Phi(y)-\Phi(x) -\lan P,y-x\ran-\frac 12(y-x)^\top Q(y-x) \big]\ges 0\Big\}.
%
%
$$
%
%
\ede

\bl\label{L11}\sl Let  $\Phi\in C(\mathbb{R}^n)$ and $x_0\in  \dbR^n$ be given. Then,
 $(P,Q)\in D_{x}^{2,-}\Phi(x_0)$ if and only if there exists a function $\f\in  C^{2}( \mathbb{R}^n)$ such that, for any $ x\in \dbR^n$, $x\neq x_0$,
$\f(x)<\Phi(x),$  and $$\big(\f(x_0),D\f(x_0),D^2\f(x_0)\big)=\big(\Phi(x_0),P,Q\big).$$
%
%
Moreover, if for some $k\ges 1$, $x\in  \dbR^n$,
 \bel{poly}|\Phi(x)|\les C(1+|x|^k),\ee
then we can choose $\f$ such that $\f$,   $\f_x$, $\f_{xx}$  also satisfy \eqref{poly} with different constants $C$.

\el

The details of the above result  can be found in \cite{YZ, ZYL, GSZ1}.
It  provides    smooth test functions to characterize the subdifferentials, which  will play a key role in the study of the verification theorem.
For convenience, we denote by $ \Xi (P,Q,\Phi)(x)$   the set of  the smooth functions characterizing the second-order  subdifferentials    of $\Phi(x)$ at $x\in\dbR^n$  as  indicated in Lemma \ref{L11}. More precisely, for $(P,Q)\in D_x^{2,-}\Phi (x)$, we define
\bel{Th-space}
\ba{ll}
\Xi (P,Q,\Phi )(x)=\Big\{\f\in C^2(\dbR^n)\mid  \Phi-\f \mbox{ attains a strict minimum  over }\dbR^n  \mbox{ at }x\in\dbR^n, \\
\ns\ds\hskip 4.8cm \mbox{  and }\Phi(x) = \f(x),\  (P,Q)=(D\f(x),D^2\f(x))  \Big\}.
%
%
%
\ea
\ee
We also need the following well known   auxiliary results.

%

\bl\label{le22}\sl
Let $g\in C([0,T])$ and extend $g$ to $(-\i,+\i)$ by setting $ g(t)=g(t^+\wedge T)  ,$ $  t\in (-\i,+\i)$.
  Suppose that  there is a function  $\rho(\cd)\in L^1(0,T;\mathbb{R})$, such that
$$\frac{g(t+h)-g(t)}{h}\ges \rho(t),\ \mbox{a.e.}, \ t\in [0,T],\ 0<h\les T-t.$$
Then,
$$g(\b)-g(\a)\ges \int_\a^\b \liminf\limits_{h\to 0^+}\frac{g(t+h)-g(t)}{h}\mathrm dt,\qq 0\les \a <\b\les T . $$


\el
%
%
%


%
%
%
%
%
%
%
%
%
%
Here we strengthen   the previous condition  {\bf(C4)}  as follows.

\ss
 {\bf(C4)$'$} For all $(s,x,y,z)\in [0,\i)\times  \dbR^n\times\dbR\times \dbR^d$, $f$ is non-decreasing in $k\in\dbR$.

 \ss
We have  the following  stochastic verification theorem of Problem (OC)$_p$  in the framework of viscosity solutions.

\bt\label{SVT-VS}\sl  Given any $p>3,$  we suppose the assumptions   {\bf (C1)$_p$}-{\bf(C3)} and  {\bf(C4)$'$}.
Let $W(\cd)\in C_{pol}(\dbR^n)$  be a semi-convex    viscosity solution of   HJB equation   \eqref{HJB-W-2} with semi-convexity constant $\kappa>0$.
For  $ x\in \dbR^n$, let $(\bar u(\cd),X^{x;\bar {u}}(\cd))\in \cU_{0,\i}^p\times L_\dbF^p(0,\i;\dbR^n) \cap L_\dbF^p(\Omega;L^2(0,\i;\dbR^n)) $ be an admissible pair and  $(Y^{x; \bar{u}}(\cd),Z^{x; \bar{u}}(\cd),K^{x; \bar{u}}(\cd))\in\sS^p_\dbF(0,\i)$ solve    BSDE \eqref{BSDE-cost}  under the control process $\bar u(\cd) $.
Assume that, for all $T>0$,  there exists  a pair
$
\big(\bar P, \bar Q\big)\in  L^\i_\dbF(0,T;\dbR^n)\times L^\i_\dbF(0,T;\dbS^n)
$
such that
$$\left\{\ba{ll}
 \ds\!\!\! {\rm(i)}\   \big( \bar P_s , \bar Q_s \big)\in D^{2,-}_{x}W\big(X^{x; \bar{u} }_s \big) ;\\
\ns\ds\!\!\! {\rm(ii)}\      \bar P_s .\si\big( X^{x; \bar{u}}_s ,\bar{u}_s\big)=Z^{x; \bar{u} }_s;\\
\ns\ds\!\!\!  {\rm(iii)} \  {W (X^{x; \bar{u} }_{s-}+\g(e, X^{x; \bar{u} }_{s-}  ,\bar u_{s} ))-W(X^{x; \bar{u} }_{s-} )=K_s^{x;\bar u}(e)},\q e\in E ; \\
\ns\ds \!\!\! {\rm (iv)}\   \dbE\big[\dbH(X^{x; \bar{u} }_s,Y^{x; \bar{u} }_s,\bar P_s, \bar Q_s,\cB^u  \f(s,X^{x; \bar{u} }_s),\cC^u   \f(s,X^{x; \bar{u} }_s),\bar u_s)\big]   \ges0, \mbox { with some }  \f\in\Xi (\bar P_s,\bar Q_s,W)(X^{x; \bar{u} }_s ) ; \\
\ns\ds\!\!\! {\rm(v)}  \ { \liminf\limits_{T\to\i}\dbE[W(X^{x;\bar u }_T )]=0} ,
\ea\right.$$
hold true,  $\dbP\mbox{-a.s., a.e., } s\in[0,T]$.
Then, $\bar u(\cd)$ is an optimal control of Problem (OC)$_p$.

\et

\br{}\sl Let us look at $\Xi (\bar P_s,\bar Q_s,W)(X^{x; \bar{u} }_s )$, which in fact for the fixed  $(s,\o)\in [0,\i)\times\Omega $.  That is, for any fixed $(s,\o)\in [0,\i)\times\Omega $,
\bel{Th-space}
\ba{ll}
\ds\Xi (P_s(\o),Q_s(\o),W)\big(X^{x; \bar{u}}_s(\o) \big)\\
\ns\ds=\Big\{\f(s,\o,\cd)\in C^2(\dbR^n)\mid W(\cd)-\f(s,\o,\cd) \mbox{ attains a strict minimum  over } \dbR^n  \mbox{ at }X^{x; \bar{u} }_s(\o)\in\dbR^n,   \\
\ns\ds\qq\qq\ \mbox{  and }   W(X^{x; \bar{u} }_s(\o)) = \f(s,\o,X^{x; \bar{u} }_s(\o)),\  (P_s(\o),Q_s(\o))=\big(D\f,D^2\f\big)(s,\o,X^{x; \bar{u} }_s(\o)) \Big\}.
%
%
%
\ea
\ee

\er

\begin{proof}
According to the uniqueness of the viscosity solution of  \eqref{HJB-W-2} in $C_{pol}(\dbR^n)$ (Remark \ref{Re-uni}), we know that, for all $x\in \dbR^n$ and $u(\cd)\in\cU_{0,\i}^p$,
\bel{ee1}
W(x)=V_p(x)\ges J\big(x;u(\cd)\big).
\ee

For any $s\ges t$,   define $\cF_{s }^{B,t}:= \si\{B_r:t\les r\les s \}\vee\mathcal{N}_{\dbP_1}$,
$\cF_{s}^{\m,t}:=\Big(\bigcap\limits_{ \varsigma>s}\sigma\big\{\mu\big((\t,r]\times \Delta\big),\ t\les \t\les r\les \varsigma , \ \Delta\in \mathcal{B}(E)\big\}\)\vee\mathcal{N}_{\dbP_2}$,
and $\cF_{s}^t:=\cF_{s}^{B,t}\otimes \cF_{s}^{\m,t}$ augmented by all the $\dbP$-null sets in $\cF$.
For any $T>0$, we also fix some point $t_0\in [0,T]$ such that   (i)
 holds true.
 Under  $\cF_{t_0}(=\cF_{t_0}^0)$,
 %
we identify the conditional probability $\dbP(\cd\mid\cF_{t_0} ) $ with its regular conditional variant. So, for fixed
 $\o_0=(\o'_0,p'_0)=(\o'_0(\cd\wedge t_0),p'_0(\cd\wedge t_0))\in \Omega$,  the  probability $\dbP(\cd\mid\cF_{t_0} )(\o_0)$  is well defined.
In this new probability space $\big(\Omega,\cF,\dbP(\cd\mid\cF_{t_0} )(\o_0)\big)$, the random variables $X^{x; \bar u }_{t_0}$,  $\bar P_{t_0}$, $\bar Q_{t_0}$ are almost surely deterministic constants and equal to  $X^{x; \bar u} _{t_0}(\o_0)$, $\bar P_{t_0}(\o_0)$, $\bar Q_{t_0}(\o_0)$, respectively. Note that, in this probability space,  $B_s$, $s\ges t_0$, is still a standard Brownian motion with  $B_{t_0}=B_{t_0}(\o_0')$ almost surely, and $\m$ is  a Poisson random measure restricted   to $[t_0,\i) $ with $\m((0,t_0], \D)=\m(p'_0,(0,t_0]\times  \D)$, $\D\in\cB(E)$.

 Now the space is equipped with the  new filtration $\{\cF_s^{t_0}\}_{s\ges t_0 }$, and
the control process $\{\bar {u}_s\}_{s\ges t_0}$ is adapted to the new filtration.
 For $\o_0$, the process $X^{x; \bar u }(\cd) $ is a solution of \eqref{state} on $[t_0,\i)$  in $\big(\Omega,\cF,\dbP(\cd\mid\cF_{t_0})(\o_0)\big)$ with the  initial condition $X^{x;  \bar u }_{t_0}=X^{x; \bar u }_{t_0}(\o_0)$.

From now on, we  keep in mind that the above  $(t_0,\o_0)$ is fixed.  From $\big( \bar P_{t_0} , \bar Q_{t_0} \big)\in D^{2,-}_{x}W\big(X^{x; \bar{u} }_{t_0} \big)$,   Lemma \ref{L11} and  \eqref{Th-space}, we know there exists a function $ \f(t_0,\o_0,\cd) \in  C^{2}( \mathbb{R}^n)$ denoted by $\bar\f(\cd)$, such that  $W-\bar\f$ attains a strict minimum over $ \dbR^n$ at $ X^{x;\bar u }_{t_0}(\o_0) $, and
\bel{equ999} \(\bar\f\big(X^{x;\bar u }_{t_0}(\o_0) \big),D\bar\f\big(X^{x;\bar u }_{t_0}(\o_0)\big),D^2\bar\f \big(X^{x;\bar u }_{t_0}(\o_0)\big)\) =\(W\big(X^{x;\bar {u} }_{t_0}(\o_0)\big), \bar P_{t_0}(\o_0),\bar Q_{t_0}(\o_0)\). \ee
 %
Moreover,  as $W(\cd)\in C_{pol}(\dbR^n)$ (in fact,  $W$ has  linear growth), $\bar \f$ can be chosen such that
$\bar\f,$ $ D\bar\f,$ and $ D^2\bar\f $ are also of  linear  growth in $x$ (see Lemma \ref{L11}).
Due to the choice above $\bar\f$ is  deterministic.
%
%

 For any $h\in (0,T-t_0]$, applying It\^o's formula to $\bar\f\big(X^{x; \bar u }_\cd\big)$ on $[t_0,t_0+h]$, we  have
\bel{f-ito}
 \ba{ll}
\ds \bar \f\big(X^{x; \bar u }_{t_0+h}\big)-\bar\f\big(X^{x;\bar u }_{t_0}\big)  =\int_{t_0}^{t_0+h}\big[  D\bar\f\big(X^{x;\bar u }_r \big). b\big(X^{x;\bar u }_r , \bar{u}_r\big)  + \frac12\tr\big(\si\si^\top   (X^{x;\bar u }_r , \bar{u}_r )  D^2\bar\f  (X^{x;\bar u }_r  )\big) \big]\mathrm dr\\
    \ns\ds\ +\int_{t_0}^{t_0+h} \int_E\big[\bar \f(X^{x;u}_{r}+\g(e,X^{x;u}_{r},u_r))-\bar \f( X^{x;u}_{r})-  D\bar \f(X^{x;u}_{r}). \g(e,X^{x;u}_{r},u_r) \big]\l(\mathrm{d}e)\mathrm dr \\
 \ns\ds\ +\int_{t_0}^{t_0+h}   D\bar\f\big(X^{x; \bar u }_r \big).\si\big(X^{x;\bar u }_r , \bar{u}_r\big)  \mathrm dB_r + \int_{t_0}^{t_0+h}\!\!\int_E\big[\bar \f(X^{x;u}_{r-}+\g(e,X^{x;u}_{r-},u_{r}))-\bar \f( X^{x;u}_{r-})\big]\ti{\m}(\mathrm dr,\mathrm de).
 \ea
\ee
Note that  condition {\bf (C1)$_p$}, the regularity properties \eqref{Lemma-SDE-1} of $X^{x; \bar u }$ and  the choice  of $\bar\f\in  C^{2}( \mathbb{R}^n)$ satisfying linear growth, imply us that all the integrals on $[t_0,t_0+h]$ in the above equality make sense.

 Taking the conditional expectation $\dbE_{t_0}[\cd]:=\dbE\big[\cd\mid\cF_{t_0}    \big](\o_0) $ in \eqref{f-ito}, we get
 \bel{W-W-h}
 \ba{ll}
 \ds \frac 1h\dbE_{t_0}\big[W\big( X^{x;\bar u }_{t_0+h}\big)-W\big( X^{x;\bar u }_{t_0}\big) \big] \ges \frac 1h\dbE_{t_0}\big[\bar\f\big( X^{x;\bar u }_{t_0+h}\big)-\bar\f\big(X^{x;\bar u }_{t_0}\big) \big]\\
 \ns\ds= \frac 1h\dbE_{t_0} \[\int_{t_0}^{t_0+h}\(D\bar\f\big(X^{x;\bar u }_r \big). b\big(X^{x;\bar u }_r , \bar{u}_r\big)  + \frac12\tr\big(\si\si^\top  (X^{x;\bar u }_r ,\bar{u}_r )  D^2\bar\f  (r,X^{x;\bar u }_r )\big)  + \cB^{\bar u}\bar \f(X^{x;\bar u}_{r })   \)\mathrm dr \] .
 \ea
 \ee
  We have that,
 for all $h\in (0,T-t_0]$,   the functions  $\ds\dbE_{t_0} [D\bar\f\big(X^{x;\bar u }_\cd\big). b\big(X^{x;\bar u }_\cd , \bar{u}_\cd\big)]$,
 $ \dbE_{t_0} \big[\tr\big(\si\si^\top  ( X^{x;\bar u }_\cd ,\bar{u}_\cd )  D^2\bar\f  (X^{x;\bar u }_\cd )\big)\big]$  and
  $ \dbE_{t_0}[\cB^{\bar u}\bar \f(X^{x;\bar u}_{\cd })]  $  are  Lebesgue integrable  on $[t_0,t_0+h]$. In fact, when   $p>3$, we have
%
$$\ba{ll}
\ds  \int_{t_0}^{t_0+h}\!\!\big|\dbE_{t_0} \big[D\bar\f\big(X^{x;\bar u }_r \big). b\big(X^{x;\bar u }_r , \bar{u}_r\big)\big]\big|\mathrm dr + \int_{t_0}^{t_0+h}\big|  \dbE_{t_0} \big[\tr\big(\si\si^\top  (X^{x;\bar u }_r ,\bar{u}_r )  D^2\bar\f (X^{x;\bar u }_r )\big)\big]\big|\mathrm dr\\
\ns\ds  \les C\dbE_{t_0} \[ \int_{t_0}^{t_0+h} \!\!\!\big(1+|X^{x;\bar u }_r|\big)\big (|b(0, \bar{u}_r )|+|X^{x;\bar u }_r|+|\si (0, \bar{u}_r )|^2+|X^{x;\bar u }_r|^2\big)\mathrm dr \]\\
%
 \ns \ds  \les C\dbE_{t_0}\[\int_{t_0}^{t_0+h}  \big ( |X^{x;\bar u }_r|^2+|X^{x;\bar u }_r|^3+|\si (0, \bar{u}_r )|^2\cd|X^{x;\bar u }_r|  +|b (0, \bar{u}_r )|^2 +|\si (0, \bar{u}_r )|^2\big )\mathrm dr \]\\
   \ns \ds  \les C_{p}   \dbE_{t_0}\[1+\sup\limits_{r\in[t_0,t_0+h]}|X^{x;\bar u }_r|^p  +\sup\limits_{r\in[t_0,t_0+h]}|X^{x;\bar u }_r|^\frac{p}{p-2}  +\(\int_{t_0}^{t_0+h}  |\si(0, \bar{u}_r )|^2  \mathrm dr\)^\frac p2\]  \\
   \ns  \ds\q+C_{p}   \(\dbE_{t_0}\[\(\int_{t_0}^{t_0+h}  |b (0, \bar{u}_r )|^2 \mathrm dr\)^\frac p2\] \)^\frac 2p    + C_{p}   \(\dbE_{t_0}\[\(\int_{t_0}^{t_0+h}  |\si(0, \bar{u}_r )|^2  \mathrm dr\)^\frac p2\] \)^\frac 2  p \\
    \ns\ds  \les C_{p}   \dbE_{t_0}\[1+\sup\limits_{r\in[t_0,t_0+h]}|X^{x;\bar u }_r|^p  +\(\int_{t_0}^{t_0+h}  |\si(0, \bar{u}_r )|^2  \mathrm dr\)^\frac p2\] \\
    \ns \ds\q +C_{p}   \(\dbE_{t_0}\[\(\int_{t_0}^{t_0+h}  |b (0, \bar{u}_r )|^2 \mathrm dr\)^\frac p2\] \)^\frac 2p   + C_{p}   \(\dbE_{t_0}\[\(\int_{t_0}^{t_0+h}  |\si(0, \bar{u}_r )|^2  \mathrm dr\)^\frac p2\] \)^\frac 2  p     <\i,\ea$$
   %
and
$$\ba{ll}
\ds \!\!\! \!\!\!\!\!\!\!\!\!\!\!\!\!\!\!\!\!\!\!\!\int_{t_0}^{t_0+h}|  \dbE_{t_0}[\cB^{\bar u}\bar \f(X^{x;\bar u}_{r })]|  \mathrm dr\\
 %
%
\ns\ds\!\!\!\!\!\!\!\!\!\!\!\!\!\!\!\!\!\!\!\!\!\!\!\les   \dbE_{t_0}\[\int_{t_0}^{t_0+h} \int_{E}\int_0^1(1-\th)|D^2\bar \f (X^{x;\bar u}_{r } +\th\g(e, X^{x;\bar u}_{r },\bar{u}_r ))|\cd |\g(e, X^{x;\bar u}_{r },\bar{u}_r )|^2 \mathrm d\th\l(\mathrm{d}e)  \mathrm dr\]  \\
\ns\ds\!\!\!\!\!\!\!\!\!\!\!\!\!\!\!\!\!\!\!\!\!\!\!\les  C \dbE_{t_0}\[\int_{t_0}^{t_0+h}\int_{E} \big(1+|X^{x;\bar u }_r|+|\g(e, X^{x;\bar u}_{r  },\bar{u}_r ) |\big) \cd|\g(e, X^{x;\bar u}_{r  },\bar{u}_r )|^2 \l (\mathrm{d}e)  \mathrm dr\]  \\
 \ea$$
 $$\ba{ll}
 %
  \ns\ds\les  C \dbE_{t_0}\[ \big(1+\sup\limits_{r\in[t_0,t_0+h]}|X^{x;\bar u }_r|\big)  \int_{t_0}^{t_0+h}\!\!\int_{E}|\g(e, X^{x;\bar u}_{r  },\bar{u}_r )|^2 \l (\mathrm{d}e)  \mathrm dr+\int_{t_0}^{t_0+h}\!\!\int_{E}  |\g(e, X^{x;\bar u}_{r  },\bar{u}_r )|^3 \l (\mathrm{d}e)  \mathrm dr\]  \\
   \ns\ds\les  C_p \dbE_{t_0}\[1+   \sup\limits_{r\in[t_0,t_0+h]}|X^{x;\bar u }_r|^{  p }  + \(\int_{t_0}^{t_0+h}\!\!\int_{E}|\g(e, X^{x;\bar u}_{r  },\bar{u}_r )|^2 \l (\mathrm{d}e)  \mathrm dr\)^\frac p2\] \\
   \ns\ds\q  +C   \dbE_{t_0}\[\int_{t_0}^{t_0+h}\int_{E}  |\g(e, X^{x;\bar u}_{r  },\bar{u}_r )|^3 \l (\mathrm{d}e)  \mathrm dr\]    <\i.
\ea$$
In the above,  we have used $p>3$, Lemma \ref{Lemma-SDE-1}, $u(\cd)\in \cU_{0,\i}^p$ and
%
%
$$\ba{ll}
  \ds  \dbE_{t_0}\[  \int_{t_0}^{t_0+h}\int_{E}  |\g(e, X^{x;\bar u}_{r  },\bar{u}_r )|^3 \l (\mathrm{d}e)  \mathrm dr\] \\
   \ns\ds = \dbE_{t_0}\[  \int_{t_0}^{t_0+h}\int_{E}  |\g(e, X^{x;\bar u}_{r  },\bar{u}_r )|^{\frac{2p-6}{p-2}}\cd  |\g(e, X^{x;\bar u}_{r  },\bar{u}_r )|^{\frac{p}{p-2}} \l (\mathrm{d}e)  \mathrm dr\]  \\
    \ns\ds  \les \dbE_{t_0}\[  \int_{t_0}^{t_0+h}\int_{E} \( \frac{|\g(e, X^{x;\bar u}_{r  },\bar{u}_r )|^{\frac{2p-6}{p-2}\cd\frac{p-2}{p-3}}}{\frac{p-2}{p-3}}+\frac{ |\g(e, X^{x;\bar u}_{r  },\bar{u}_r )|^{\frac{p}{p-2}\cd p-2} }{p-2} \)\l (\mathrm{d}e)  \mathrm dr\]  \\
  \ns\ds  \les C_p\dbE_{t_0}\[  \int_{t_0}^{t_0+h}\int_{E} \big( |\g(e, X^{x;\bar u}_{r  },\bar{u}_r )|^2+  |\g(e, X^{x;\bar u}_{r  },\bar{u}_r )|^p\big)\l (\mathrm{d}e)  \mathrm dr\]  \\
 \ns\ds \les C_p \(\dbE_{t_0}\[\(\int_{t_0}^{t_0+h} \int_E   ( |\g(e,0,\bar{u}_r )|^2+\ell_\g(e)^2  |X^{x;\bar u}_{r  }|^2) \l(\mathrm de)\mathrm dr\)^\frac p2\]\)^\frac 2p\\
 \ns\ds\qq +  C_p \dbE_{t_0}\[\int_{t_0}^{t_0+h} \int_E   ( |\g(e,0,\bar{u}_r )|^p+\ell_\g(e)^p  |X^{x;\bar u}_{r  }|^p) \l(\mathrm de)\mathrm dr \] <\i.
\ea$$

Then,  we can apply the  Lebesgue differentiation theorem to \eqref{W-W-h},
 \bel{ineq-111}
 \ba{ll}
   \ds  \liminf\limits_{h\to 0^+}\frac 1h\dbE_{t_0}\Big[W\big(X^{x; \bar u }_{t_0+h}\big)-W\big(X^{x;\bar u }_{t_0}\big) \Big]\\
 \ns\ds\ges\liminf\limits_{h\to0^+}\frac 1h\dbE_{t_0} \[\int_{t_0}^{t_0+h}\!\! \Big(  D\bar\f\big(X^{x;\bar u }_r \big). b\big(X^{x;\bar u }_r , \bar{u}_r\big)  + \frac12\tr\big(\si\si^\top  (X^{x;\bar u }_r ,\bar{u}_r )  D^2\bar\f  (X^{x;\bar u }_r )\big)+ \cB^{\bar u}\bar \f(X^{x;\bar u}_{r })\)\mathrm dr  \] \!\!\!\!\!\!\!\!  \\
 \ns\ds=  D\bar\f\big(X^{x;\bar u }_{t_0}\big). b\big(X^{x;\bar u }_{t_0}, \bar u_{t_0} \big)   + \frac12\tr\big(\si\si^\top  (X^{x;\bar u }_{t_0},\bar{u}_{t_0} ) D^2\bar\f  (X^{x;\bar u }_{t_0} )\big)   + \cB^{\bar u}\bar \f(X^{x;\bar u}_{t_0}). \
 \ea
\ee
The arbitrariness of   the point  $t_0\in[0,T]$ in the above implies that  the set of the points $t_0$  satisfying condition (i) and  \eqref{ineq-111} is of full measure in $[0,T]$.

Next, we claim that, for    the above   $t_0$ lying in $[0,T] $,
 and for any  $h>0 $ with  $t_0+h\les T $,
\bel{ineq-000}\ba{ll}
 \ds {\rm(a)}\ \   \frac 1h\dbE_{t_0}\big[ W\big( X^{x; \bar u }_{t_0+h}\big)-W\big( X^{x; \bar u }_{t_0} \big)   \big]\ges  - C \big(1+  |X^{x; \bar u }_{t_0}|^2   \big), \\
 \ns\ds {\rm(b)}\ \  \frac 1h \dbE\big[ W\big( X^{x; \bar u }_{t_0+h}\big)-W\big( X^{x; \bar u }_{t_0} \big)  \big]\ges -  C \big (1 +|x |^2   \big)  .
 \ea\ee
In fact, from the semi-convexity of $W(\cd)$ and $\big(\bar P_{t_0} , \bar Q_{t_0} \big)\in D^{2,-}_{x}W(X^{x; \bar u }_{t_0})$, we have,    for all $h\in (0,T-t_0]$,
\bel{4.23} \ds W\big( X^{x; \bar u }_{t_0+h}\big)-W\big( X^{x; \bar u }_{t_0} \big)\ges  \bar P_{t_0} .( X^{x; \bar u }_{t_0+h}-X^{x; \bar u }_{t_0})-\k|X^{x; \bar u }_{t_0+h}-X^{x; \bar u }_{t_0} |^2.\ee
%
Note that   $ |\bar P_{t_0} |+|\bar Q_{t_0} |\les C(1+|X^{x; \bar u }_{t_0} | ).$
Combined with  estimate \eqref{SDE-(i)}, we get
$$
 \ba{ll}
  %
 \ds  \dbE _{t_0}\big[  -\bar P_{t_0} . (X^{x; \bar u }_{t_0+h}-X^{x; \bar u }_{t_0} ) \big]
  %
= \dbE_{t_0} \[ \big\langle -\bar P_{t_0},\int_{t_0}^{t_0+h}b\big(X^{x; \bar u }_r ,\bar{u}_r\big)\mathrm dr\big\rangle  \]\\
\ns\ds \les  \(\dbE_{t_0}\big[|\bar P_{t_0}|^2  \big]\)^\frac12\(\dbE_{t_0}\[\(\int_{t_0}^{t_0+h}b\big(X^{x; \bar u }_r ,\bar{u}_r\big)\mathrm dr\)^2   \] \)^\frac12\les Ch\(1+\dbE_{t_0}\[\sup\limits_{ r\in [t_0,t_0+h]}|X^{x; \bar u }_r|^2  \]\) \\
\ns\ds \les Ch\(1+\(\dbE_{t_0}\[\sup\limits_{ r\in [t_0,t_0+h]}|X^{x; \bar u }_r|^p  \]\)^\frac 2p\) \les   Ch\(1+\big( 1+|X^{x; \bar u }_{t_0}|^p   \big)^\frac 2p\)\les   Ch\big(  1+|X^{x; \bar u }_{t_0}|^2   \big)  ,
\ea$$
and
$$
 \ba{ll}
  \ds   \dbE _{t_0}\big[|X^{x; \bar u }_{t_0+h}-X^{x; \bar u }_{t_0} |^2  \big] \\
  \ns\ds \les 3 \dbE _{t_0}\[ \(\!\int_{t_0}^{t_0+h}b\big(X^{x; \bar u }_r ,\bar{u}_r\big)\mathrm dr  \)^2  + \(\!\int_{t_0}^{t_0+h}\si\big(X^{x; \bar u }_r ,\bar{u}_r\big)\mathrm dB_r\!   \)^2 + \(\int_{t_0}^{t_0+h}\!\int_E\g\big(e,X^{x; \bar u }_{r-} ,\bar{u}_{r }\big)\ti{\m}(\mathrm dr, \mathrm de  \)^2\]\\
\ns\ds\les Ch \(1+\dbE _{t_0}\[\sup\limits_{ r\in [t_0,t_0+h]}|X^{x; \bar u }_r|^2  \]\)\les   Ch\big(  1+|X^{x; \bar u }_{t_0}|^2   \big) .
\ea$$
All the above constants $C$ can be different but they do not depend on $t_0$.
Substituting these two estimates into \eqref{4.23}, we get \eqref{ineq-000}-(a).
Furthermore, by taking the expectation on the both sides
of \eqref{ineq-000}-(a) and using  $\dbE  \big[  |X^{x; \bar u }_{t_0}|^2  \big] \les \Big(\dbE\big[  |X^{x; \bar u }_{t_0}|^p\big]\Big)^\frac {2}{p}\les C(1+|x|^2) $, we get   \eqref{ineq-000}-(b).

By taking expectation on the both sides of \eqref{ineq-111},  and  applying Fatou's Lemma (using \eqref{ineq-000}-(a)), we have
 $$
 \ba{ll}
 \ds \liminf\limits_{h\to 0^+}\frac 1h\dbE [W\big(X^{x;\bar u }_{t_0+h}\big)-W\big(X^{x;\bar u }_{t_0}\big) ] = \liminf\limits_{h\to 0^+}\frac 1h\dbE\[\dbE_{t_0}\big[W\big(X^{x;\bar u }_{t_0+h}\big)-W\big(X^{x;\bar u }_{t_0}\big) \big]\]\\
 \ns\ds\ges \dbE\[\liminf\limits_{h\to 0^+}\frac 1h\dbE_{t_0}\big[W\big(X^{x;\bar u }_{t_0+h}\big)-W\big(X^{x;\bar u }_{t_0}\big) \big]\]\\
 \ns\ds \ges  \dbE\[ \bar P_{t_0}. b\big(X^{x;\bar u }_{t_0},\bar{u}_{t_0}\big)  + \frac12\tr\big(\si\si^\top  (X^{x;\bar u }_{t_0},\bar{u}_{t_0})\bar Q_{t_0}\big) + \cB^{\bar u}\bar \f(X^{x;\bar u}_{t_0})\]\\
 \ns\ds{\ges -\dbE\[   f\big(X^{x;\bar u }_{t_0} ,Y^{x;\bar u }_{t_0}  ,\bar P_{t_0} .\sigma(X^{x;\bar u }_{t_0},\bar{u}_{t_0}),\int_E\big[ \bar \f(X^{x;\bar u}_{{t_0}}+\g(e,X^{x;\bar u}_{{t_0}},\bar u_{t_0}))-\bar \f( X^{x;\bar u}_{{t_0}})\big]\rho(e)\l(\mathrm{d}e),
 \bar{u}_{t_0}\big)\]}\\
 \ns\ds{ \ges -\dbE\[  f\big(X^{x;\bar u }_{t_0} ,Y^{x;\bar u }_{t_0}  ,\bar P_{t_0} .\sigma(X^{x;\bar u }_{t_0},\bar{u}_{t_0}),\int_E\big[ W(X^{x;\bar u}_{{t_0}}+\g(e,X^{x;\bar u}_{{t_0}},\bar u_{t_0}))-W ( X^{x;\bar u}_{{t_0}})\big]\rho(e)\l(\mathrm{d}e),
 \bar{u}_{t_0}\big)\]},
 %
 \ea
 $$
 where we also use \eqref{equ999}, (iv) and {\bf(C4)$'$}.
 %

 %
For any $T>0$, due to the set of such points $t_0$ being of full measure in $[0,T]$, by  applying Lemma \ref{le22} (needing \eqref{ineq-000}-(b)),   we have
  $$
 \ba{ll}
   \ds  \dbE\big[W\big(X^{x;\bar u }_T \big)-W(x) \big] \\
  %
   \ns\ds\ges -\dbE \int_0^{T}   f\big(X^{x;\bar u }_s ,Y^{x;\bar u }_s  ,\bar P_s .\sigma(X^{x;\bar u }_s ,\bar{u}_s),\int_E\big[ W(X^{x;\bar u}_{s}+\g(e,X^{x;\bar u}_{s},\bar u_s))-W( X^{x;\bar u}_{s})\big]\rho(e)\l(\mathrm{d}e),
 \bar{u}_s\big)\mathrm ds  \\
  \ns\ds=-\dbE \int_0^{T}    f\big(X^{x;\bar u }_s ,Y^{x;\bar u }_s  ,Z^{x;\bar u }_s ,\int_EK^{x;\bar u }_s(e)\rho(e)\l(\mathrm{d}e),
 \bar{u}_s\big)\mathrm ds  ,
 \ea$$
   where we have used the conditions (ii) and  (iii).
 Finally, letting $T\to \i$ and using  (v), we have
  $$
  W(x) \les \dbE \int_0^\i f\big(X^{x;\bar u }_s ,Y^{x;\bar u }_s  ,Z^{x;\bar u }_s ,\int_EK^{x;\bar u }_s(e)\rho(e)\l(\mathrm{d}e),
 \bar{u}_s\big)\mathrm ds   =J\big(x; \bar u(\cd)\big).
 $$
 Combined with \eqref{ee1}, we get
 $
 W(x)=  V_p(x)=J\big(x; \bar u (\cd)\big), $  $x\in\dbR^n,
$
which means $\bar u(\cd)$ is an  optimal control of Problem (OC)$_p$.

\end{proof}

\br\label{Re-feed-law}\sl   As we have  seen in the  above proof,  the semi-convexity of the viscosity solution of HJB equation \eqref{HJB-W-2} is crucial.  Fortunately, the value function $V_p(\cd)$ of  Problem (OC)$_p$ as a viscosity solution of \eqref{HJB-W-2}  really
 satisfies this property under some additional conditions on $b$, $\si$, $g$ and $f$, referring Proposition \ref{Le-Semi}.  It is a pity
that  the semi-convexity holds true   just when $p>4$. In other hands,  Theorem \ref{SVT-VS} is applicable to   Problem (OC)$_p$ with $p>4$.  It is different from   finite horizon stochastic control problems, we refer  to \cite{CL-arxiv, GSZ1, GSZ2,  LHW-2023}.

%
%
%
\er
 

\subsection{Optimal Feedback  Control Laws}
\par  In this subsection, we present the construction of    optimal feedback control laws of Problem (OC)$_p$ from the viscosity solution of HJB equation \eqref{HJB-W-2}. The study here  is restricted to $p>4$, due to Remark \ref{Re-feed-law}.

\bl\label{th4.10}\sl Let  {\bf (C1)$_p$}-{\bf (C3)}, {\bf(C4)$'$},  {\bf (D1)}, {\bf (D2)}   and  \eqref{b4} hold true. Then the value function $V_p(\cd)$ defined in \eqref{bar u} is the only function in $C_{pol}(\dbR^n)$ satisfying Lemma \ref{W-Lip}, Proposition \ref{Le-Semi} and   the following: for all $x\in   \dbR^n$,  $(P,Q)\in D^{2,-}_{x}V_p(x)$ and any  $\f\in \Xi  (P,Q,V_p)(x)$,
\bel{VIS-inequ}
   \sup_{u\in U}\dbH(x,V_p(x),P,Q,\cB^u\f(x),\cC^u\f(x),u) \ges 0.
\ee

\el

\begin{proof}
	From Theorem \ref{Th-exist}, we know the value function $V_p(\cd)\in C_{pol}(\dbR^n)$ is the   unique viscosity solution of \eqref{HJB-W-2}.  By Lemma \ref{L11}, for every $(P,Q)\in D^{2,-}_{x}V_p(x)$,  we  get the existence of the functions $\f\in  \Xi (P,Q,V_p)(x)$.
For any such test function  $\f$, we obtain
	$$  \sup_{u\in U}\dbH(x, V_p(x),D\f (x),D^2\f (x),\cB^u\f(x),\cC^u\f(x),u)  \ges 0,$$
from   Definition \ref{Def-vis}, Remark \ref{Re-Def} and  Theorem \ref{Th-exist}.
So, \eqref{VIS-inequ} holds true.
	%
	%
	On the other hand, the uniqueness comes from the uniqueness of the viscosity solution of  \eqref{HJB-W-2} in $C_{pol}(\dbR^n) $.

\end{proof}

\bt\label{th4.11} \sl Assume {\bf (C1)$_p$}-{\bf (C3)}, {\bf(C4)$'$},  {\bf (D1)}, {\bf (D2)}   and  \eqref{b4} hold true. Let $W\in   C_{pol}( \dbR^n)$   be      a viscosity solution of
\eqref{HJB-W-2}. Then, for all $x\in\dbR^n$,
\bel{equ012}
\ds\sup_{(P,Q,\f,u)\in D^{2,-}_{x}W(x)\times  \Xi  (P,Q,W)(x) \times U}\[\dbH(x, W(x),P,Q,\cB^u\f(x),\cC^u\f(x),u)\]\ges 0 . 
\ee
Furthermore,  if $\mathbbm{u}(\cdot)\in\sU^p$ and    $\cP ,$ $\cQ $ are measurable functions
satisfying $( \cP (x),\cQ (x))\in D^{2,-}_{x}W(x)$  for all $x\in\dbR^n$, and for any $T>0$,
\bel{equ015}\left\{\ba{ll}
 \ds\!\!\! {\rm (i)} \     \cP (X^{x;\mathbbm{u}}_s).\si\big(X^{x;\mathbbm{u}}_s ,\mathbbm{u}(X^{x;\mathbbm{u}}_s)  \big)=Z^{x;\mathbbm{u}}_s, \\
\ns\ds\!\!\! {\rm(ii)} \  W \big(X^{x; \mathbbm{u}}_{s-}+\g(e, X^{x;\mathbbm{u}}_{s-}  ,\mathbbm{u}(X^{x;\mathbbm{u}}_ {s-}) )\big)-W \big(X^{x; \mathbbm{u} }_{s-} \big)=K_s^{x;\mathbbm{u}}(e),\q e\in E;\\
\ns\ds\!\!\! {\rm (iii)} \  \dbE\big[\dbH\big( X^{x;\mathbbm{u}}_s ,Y^{x;\mathbbm{u}}_s,  \cP (X^{x;\mathbbm{u}}_s),\cQ (X^{x;\mathbbm{u}}_s),\cB^{\mathbbm{u}}\f(X^{x;\mathbbm{u}}_s),C^{\mathbbm{u}} \f(X^{x;\mathbbm{u}}_s),\mathbbm{u}(X^{x;\mathbbm{u}}_s) \big) \big]\ges 0,\\
\ns\ds\hskip 5cm  \mbox { with some }   \f\in\Xi (\cP(X^{x; \mathbbm{u} }_s ),\cQ (X^{x; \mathbbm{u} }_s ),W )(X^{x; \mathbbm{u} }_s ) ;\\
\ns\ds\!\!\! {\rm(iv)}  \ \liminf\limits_{T\to\i}\dbE[W\big(X^{x;\mathbbm{u}}_T \big)]=0,
\ea\right.\ee
  hold, $\dbP$-a.s. a.e.  $s\in[0,T] $,
where $ X^{x;\mathbbm{u} }(\cd)$, $(Y^{x;\mathbbm{u} },Z^{x;\mathbbm{u} },K^{x;\mathbbm{u} })$ satisfy   \eqref{SDE-u} and  \eqref{BSDE-u} with $ \mathbbm{u} (X^{x;\mathbbm{u} }_\cd) \in \cU _{0,\i}^p$, respectively.
Then,     $ \mathbbm{u}(\cd)$ is an optimal feedback  control law   of Problem (OC)$_p$.

\et

\begin{proof}
	From the uniqueness of   viscosity solution of HJB equation \eqref{HJB-W-2} and   Lemma \ref{th4.10},  \eqref{equ012} holds obviously.

  Next, for any $x\in\dbR^n$, the admissible feedback control law  $\mathbbm{u}(\cd)$ and the solution $X^{x;\mathbbm{u} } $  of \eqref{SDE-u}, we set
	$$
	\bar u(s):= \mathbbm{u}(X^{x;\mathbbm{u}}_s ),\enspace \bar P(s):= \cP (X^{x;\mathbbm{u}}_s ),\enspace  \bar Q(s) := \cQ (X^{x;\mathbbm{u}}_s ),\enspace s\ges 0.
 $$
Using \eqref{equ015},  we know $\bar u(\cd)$, $X^{x;\mathbbm{u} }$,  $(Y^{x;\mathbbm{u} },Z^{x;\mathbbm{u} },K^{x;\mathbbm{u} })$ and  the above $(\bar P, \bar Q)$ satisfy (i)-(v) in Theorem \ref{SVT-VS}. Therefore,  $ \bar u(\cd)$ is an optimal control, i.e.,  $\mathbbm{u}(\cd)$ is an optimal feedback control law.
	
\end{proof}

Finally, we have a look at the procedures of finding the optimal feedback control law.
 Under  {\bf (C1)$_p$}-{\bf (C3)},   {\bf(C4)$'$}, {\bf (D1)},  {\bf (D2)}   and  \eqref{b4}, once we get a viscosity solution $W\in   C_{pol}( \dbR^n)$ of
\eqref{HJB-W-2},   the first thing is to find a measurable selection $\mathbbm{u}( \cd)$  and
$( \cP (x),\cQ(x) )\in D^{2,-}_{x}W(x)$ and $\f \in \Xi  (\cP(x),\cQ(x),W)(x)$ to achieve the maximum of
 $ \dbH(x, W(x),P,Q,\cB^u\f(x),\cC^u\f(x),u) $
over $D^{2,-}_{x}W(x)\times \Xi  (P,Q,W)(x)\times U$.
Further,  we need to make sure $\mathbbm{u}( \cd)$ belongs to $\sU^p$ and \eqref{equ015} is valid under  $( \cP (x),\cQ(x) )\in D^{2,-}_{x}W(x)$ and $\f \in \Xi  (\cP(x),\cQ(x),W)(x)$, which ensure  the measurable selection $\mathbbm{u}( \cd)$  to be the true optimal feedback control law.

 The above construction process   is actually a big project, and there are still some research  blanks needed to be solved.
We guess the current popular theory of   reinforcement learning may behavior better in this regard.
Due to the  space limitation, we now stop here and  hope to get some results about the above      in the future works.

%

\end{document}